\documentclass[12pt]{article}
\usepackage{preprint-layout}%
\usepackage{preprint-notation}%
\usepackage{doi}
\usepackage{hyperref}
\hypersetup{hidelinks}
\usepackage{extdash}

\DeclareMathOperator{\Div}{Div}
\DeclareMathOperator*{\essinf}{ess\,inf}%

\usepackage{amsmath}
\usepackage{mathtools}
\mathtoolsset{showonlyrefs=false}
\allowdisplaybreaks

\makeatletter
\newcommand{\proofparagraph}[1]{%
  \ifnum\prevgraf=0\relax
    {\bfseries\boldmath{#1}}\hspace{0.6em}%
  \else
    \par\addvspace{.5\baselineskip}%
    \noindent{\bfseries\boldmath{#1}}\hspace{0.6em}%
  \fi
}
\makeatother

\usepackage{booktabs,ctable, multirow}%
\usepackage{siunitx}%
\newcommand{\na}{\num[group-digits = none]}%

\title{Cut Finite Element Methods for Convection-Diffusion in\\ Mixed-Dimensional Domains}
\date{\today}
\author{Erik Burman,
\quad
Peter Hansbo,
\quad
Mats G. Larson,
\\
Karl Larsson,
\quad
Shantiram Mahata
}
\date{\today}

\begin{document}
\maketitle

\begin{abstract}
We develop a cut finite element method (CutFEM) for convection--diffusion problems posed on mixed-dimensional domains, i.e., unions of manifolds of different dimensions arranged in a hierarchical structure where lower-dimensional components form parts of the boundaries of higher-dimensional ones. Such domains arise, for instance, in the modeling of fractured porous media with intersecting fractures.
The model problem is formulated in a compact abstract form using mixed-dimensional directional derivative and divergence operators, which allows the problem to be expressed in a way that closely resembles the classical convection--diffusion equation.
The proposed CutFEM is based on a fixed background mesh that does not conform to the geometry, with each manifold component represented through its associated active mesh. The method employs continuous piecewise linear elements together with weak enforcement of coupling conditions and suitable stabilization.
We prove a priori energy norm error estimates under a global
uniform-diffusion assumption, with corresponding extensions to
solutions of reduced regularity $u\in H^s$, $1\le s<2$, and derive
conditional estimates for the globally pure-convection case.
Partially degenerate configurations, in which diffusion is present
only on selected components, are explored numerically. The experiments
report convergence in both the energy and $L^2$-norms and illustrate
the performance of the method.
\end{abstract}

\section{Introduction}

\paragraph{Mixed-Dimensional Domains.}
Many physical processes occur in media with complex internal structure where lower-dimensional features strongly influence transport. Typical examples include flow and transport in fractured porous media \cite{NoBo17,BoNoYo18}, thin conductive layers in composite materials \cite{MR1417476,10.1093/oso/9780198565543.001.0001}, and networks of channels or interfaces embedded in a higher-dimensional bulk medium \cite{formaggia2009cardiovascular}. In such settings, the geometry naturally consists of objects of different dimensions that interact with one another. Rather than resolving every thin structure in a fully three-dimensional mesh, a common modeling strategy is to represent these features as lower-dimensional manifolds embedded in the surrounding bulk domain and couple the resulting equations through suitable interface conditions.

Such geometries are referred to as \emph{mixed-dimensional domains}. A mixed-dimen\-sional domain in $\mathbb{R}^n$ consists of manifolds of dimension $d=0,\dots,n$ that are connected hierarchically so that lower-dimensional components reside on the boundaries of higher-dimensional ones. For example, in three dimensions the domain may consist of three-dimensional bulk regions, two-dimensional surfaces representing fractures or thin layers, one-dimensional curves representing fracture intersections, and zero-dimensional points corresponding to junctions. Mixed-dimensional formulations have become increasingly important in the modeling of fractured porous media, where fractures and their intersections are represented in reduced dimension rather than explicitly resolving their thickness in the computational mesh. These formulations introduce significant analytical challenges due to coupling across dimensions and the lack of systematic tools for their effective analysis.

\paragraph{New Contributions.}
In this work we develop and analyze a stabilized CutFEM for convection--diffusion problems posed on mixed-dimensional domains. The method provides a unified analytical and numerical framework, together with a stable unfitted finite element discretization and corresponding a priori error analysis. More specifically, the contributions of the paper can be summarized as follows:
\begin{itemize}
\item We establish a unified operator-based framework for mixed-dimensional convection--diffusion problems based on abstract directional derivative and divergence operators, together with mixed-dimensional interface operators that couple neighboring manifolds. The formulation yields a compact representation that closely parallels the classical equation on domains in $\mathbb{R}^n$.

\item We develop a CutFEM discretization in which the mixed-dimensional geometry is embedded in a fixed background mesh. Each manifold component is associated with an active mesh of its own consisting of intersecting elements, eliminating the need for mesh-fitting while enforcing all inter-manifold coupling weakly in a consistent manner. The method provides a unified treatment of all manifolds regardless of their dimension within a single discretization framework.

\item We design a stabilized formulation that combines
Galerkin--least-squares stabilization for the convection operator with
full-gradient stabilization on cut elements. We derive a priori error
estimates in the natural energy norm under a global uniform-diffusion
assumption, with corresponding extensions to solutions with reduced
regularity and conditional estimates for globally pure convection.
Partially degenerate configurations are explored numerically outside
the precise assumptions of the analysis. The experiments report
convergence in both the energy and $L^2$-norms. The method is robust in
convection-dominated regimes and straightforward to implement.
\end{itemize}

\paragraph{Earlier Work.}
A large body of work has been devoted to mixed-dimensional and reduced models, in particular in the context of flow in fractured porous media and related applications. Early approaches model fractures as lower-dimensional interfaces coupled to bulk equations, see, e.g., \cite{MaJaRob05,MR1911534,AnBoHu09}. A functional analytic framework for mixed-dimensional PDEs has been developed in \cite{BoNoVa17,NoBo17}, and numerical discretizations have been studied extensively, including finite volume, virtual element, and discontinuous Galerkin methods, cf.~\cite{BoNoYo18,FumKei17,AnFaRuVe19,MR3038028}. While these works provide powerful modeling and discretization approaches, they do not provide a unified operator-based framework for convection--diffusion problems that facilitates both analysis and discretization across dimensions.

More recently, CutFEMs have emerged as a flexible approach for handling PDEs on embedded and intersecting geometries. For surface partial differential equations, also known as the trace finite element method, the approach was introduced in \cite{OlReGr09} and further developed in many directions, including stabilization techniques \cite{BurHanLar15}, higher-order approximations \cite{Reu15}, and discontinuous Galerkin formulations \cite{BurHanLarMas17}. Applications to transport problems and coupled bulk--surface systems can be found in \cite{OlsReuXu14-b,BuHaLaZa19,GroOlsReu15,BurHanLarZah16}. We also refer to the overview articles \cite{BurClaHanLarMas15,BuHaLaZa25} and the references therein. In the mixed-dimensional setting, CutFEM has been applied to flow and transport problems in fractured domains in \cite{BuHaLa20,BuHaLaLa19}. However, existing CutFEM approaches are primarily developed for single manifolds or bulk--surface couplings and do not directly extend to a unified treatment of general mixed-dimensional geometries with multiple interacting manifolds.

A related class of mixed-dimensional problems arises in $3$d--$1$d coupling, where lower-dimensional structures act as singular sources in the bulk, for instance in tissue perfusion and vascular flow modeling \cite{MR2439847,formaggia2009cardiovascular}. These problems are closely connected to elliptic equations with Dirac measures on lower-dimensional sets, for which finite element methods in weighted Sobolev spaces and corresponding error estimates have been developed in \cite{MR2888310,MR3239767}. While outside the scope of this contribution, they highlight the reduced regularity and need for specialized discretizations.

Despite these advances, a systematic framework for convection--diffusion problems on mixed-dimensional geometries that supports both robust discretization and rigorous error analysis remains largely undeveloped, particularly in the convection-dominated regime. 

\paragraph{Outline.}
The remainder of the paper is organized as follows. In Section~2 we
introduce mixed-dimensional domains, define the mixed-dimensional interface operators and
mixed-dimensional differential operators, and formulate the model
convection--diffusion problem together with a well-posedness result for
positive diffusion under an explicit coercivity condition.
Section~3 presents the CutFEM discretization. Section~4 is devoted to
energy-error estimates for both $H^2$ and lower-regularity solutions
in the globally diffusive and globally pure-convection regimes.
Numerical experiments are presented in Section~5. Finally, Section~6
contains concluding remarks and directions for future work.

\section{The Model Problem}
\label{sec:model-problem}

\subsection{The Domain and Function Spaces}
\label{sec:defs}

\paragraph{Mixed-Dimensional Domain.}
Let $\Omega$ be a domain in $\IR^n$ that admits a partition into manifolds of varying dimensions. We assume that the components are arranged hierarchically so that lower-dimensional manifolds arise as parts of the boundaries of higher-dimensional manifolds. For instance, when $n=3$, the domain may consist of three-dimensional bulk regions whose interior boundaries include two-dimensional embedded surfaces, which in turn may meet along one-dimensional intersection curves that terminate at zero-dimensional intersection points, see Figure~\ref{fig:mixeddim-domain}.

\begin{figure}
\centering
\begin{subfigure}[t]{.49\linewidth}\centering
\includegraphics[trim=100px 40px 100px 30px, clip,width=0.87\linewidth]{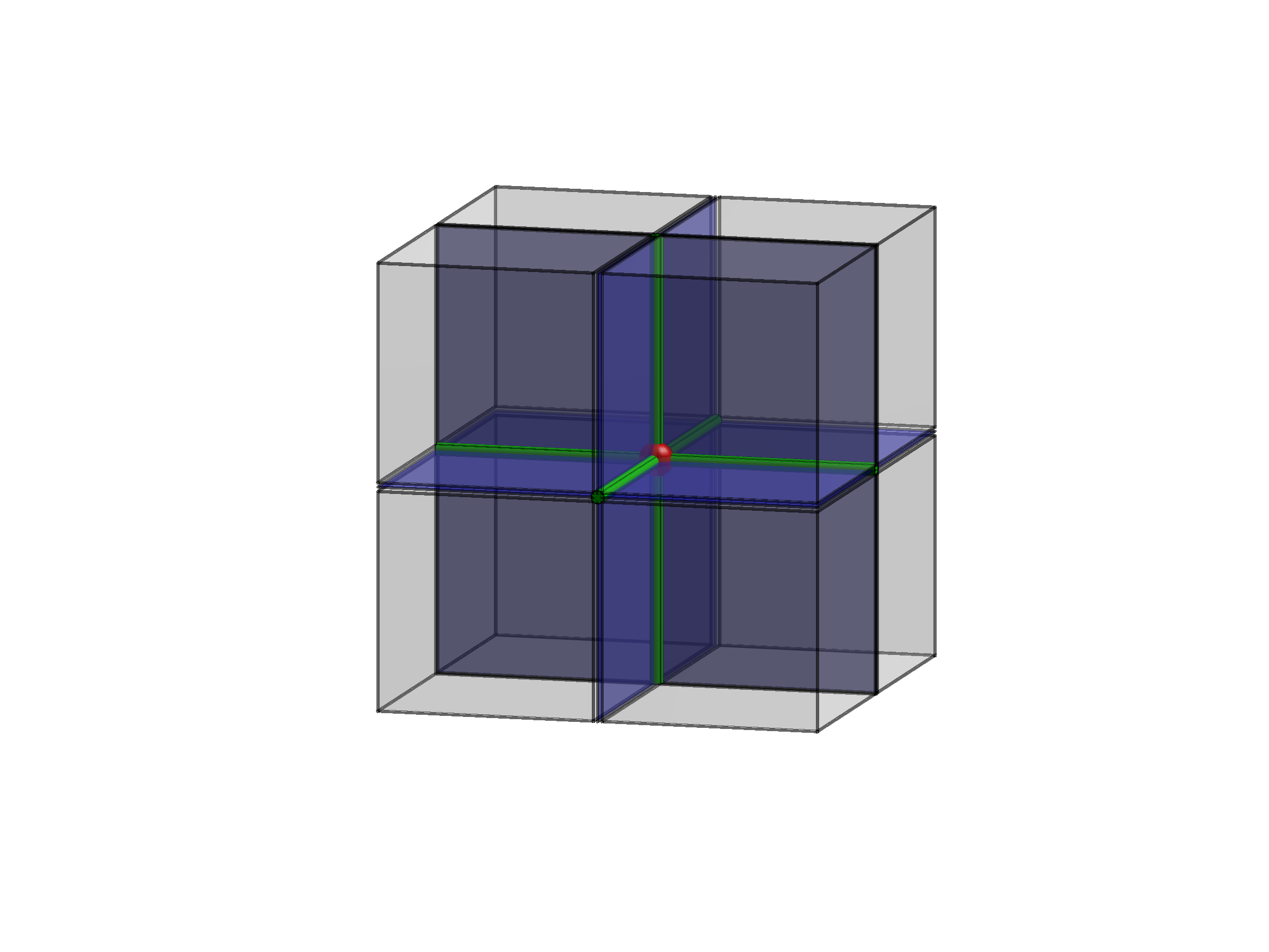}
\caption{Mixed-dimensional domain $\mcO$} \label{fig:domain-ex-a}
\end{subfigure}
\begin{subfigure}[t]{.49\linewidth}\centering
\includegraphics[trim=100px 40px 100px 30px, clip,width=0.87\linewidth]{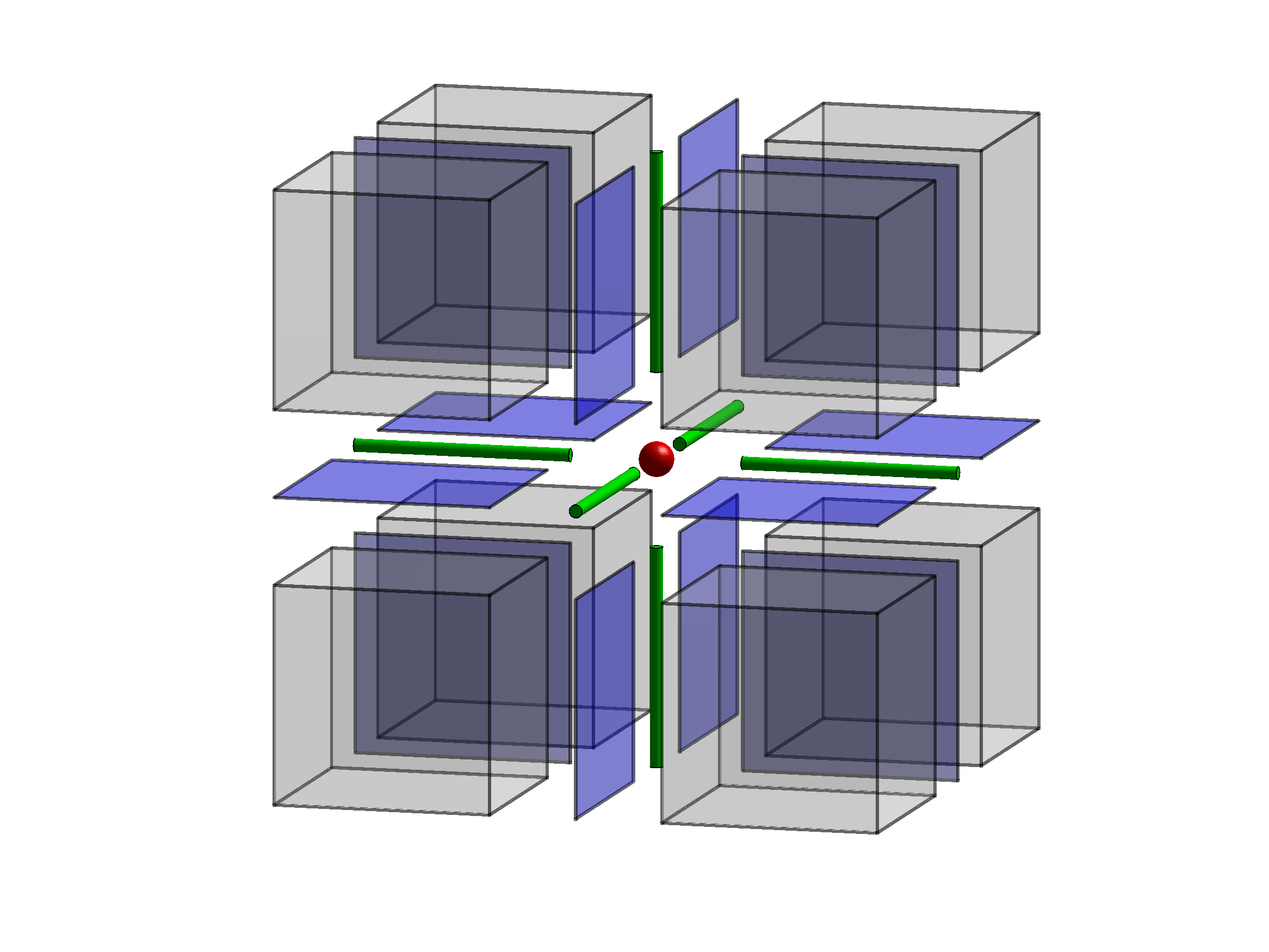}
\caption{Exploded view of $\mcO$} \label{fig:domain-ex-b}
\end{subfigure}
\caption{Illustration of a mixed-dimensional domain $\mcO$ in three dimensions ($n=3$) consisting of $n_3=8$ cubes, $n_2=12$ squares, $n_1=6$ lines, and $n_0=1$ point.}
\label{fig:mixeddim-domain}
\end{figure}

The domain $\Omega$ is partitioned into strata of different dimensions such that
\begin{equation}
\Omega = \bigcup_{d=0}^n \Omega_d
\label{eq:domain-decomposition}
\end{equation}
where $\Omega_d$ consists of all $d$-dimensional components of $\Omega$.
Each $\Omega_d$ is further partitioned as
\begin{equation}
\Omega_d = \bigcup_{i=1}^{n_d} \Omega_{d,i}
\label{eq:domain-partition}
\end{equation}
such that each $\Omega_{d,i}$ is a smooth $d$-dimensional manifold with boundary $\partial\Omega_{d,i}$.
We write
\begin{equation}
\mcO_d := \{\Omega_{d,i}\}_{i=1}^{n_d},
\qquad
\mcO := \{\mcO_d\}_{d=0}^n
\end{equation}
for the collection of $d$-dimensional components and the resulting mixed-dimensional domain, respectively.
The partition satisfies
\begin{equation}
\partial\Omega_{d,i} \subset \partial\Omega \cup \bigcup_{l=0}^{d-1}\Omega_l
\qquad i = 1,\dots,n_d, \quad d = 0,\dots,n
\label{eq:boundary-condition}
\end{equation}

We define the boundary sets associated with this partition
\begin{equation}
\partial\mcO_d = \bigsqcup_{i=1}^{n_d} \partial\Omega_{d,i},
\qquad
\partial\mcO = \bigsqcup_{d=0}^n \partial\mcO_d
\label{eq:boundary-sets}
\end{equation}
where $\sqcup$ denotes disjoint union. See Figure~\ref{fig:normals} for an illustration of the geometry and the associated notation.

\begin{figure}
\centering
\begin{subfigure}[t]{.32\linewidth}\centering
\includegraphics[width=0.87\linewidth]{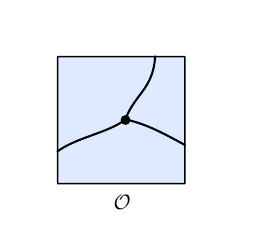}
\caption{Mixed-dimensional domain}
\label{fig:domain-illu-a}
\end{subfigure}
\begin{subfigure}[t]{.32\linewidth}\centering
\includegraphics[width=0.87\linewidth]{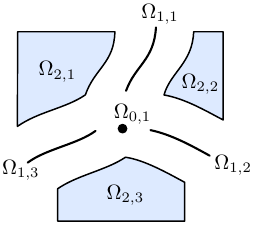}
\caption{Partition notation}
\label{fig:domain-illu-b}
\end{subfigure}
\begin{subfigure}[t]{.32\linewidth}\centering
\includegraphics[width=0.87\linewidth]{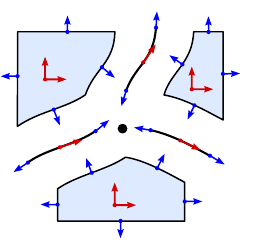}
\caption{Tangential (red) and normal (blue) vector fields}
\label{fig:domain-illu-c}
\end{subfigure}
\caption{Illustration of a mixed-dimensional domain. (a) A domain $\mcO$ where $n=2$, $n_0=1$, $n_1=3$, and $n_2=3$. (b) Notation for the partition of $\mcO$. (c) Tangential vector fields (red) and exterior unit normal fields (blue).}
\label{fig:normals}
\end{figure}

\paragraph{Sobolev Spaces on $\mcO$ and $\partial\mcO$.}
Let $H^s(\Omega_{d,i})$ denote the Sobolev space of order $s$ on the manifold $\Omega_{d,i} \in \mcO_d$ with scalar product $(v,w)_{H^s(\Omega_{d,i})}$.
An element $v \in H^s(\mcO)$ should thus be understood as a collection of component functions $v_{d,i} \in H^s(\Omega_{d,i})$.
Note that we do not impose continuity between different components; instead, their coupling is imposed weakly in our formulation.

We define mixed-dimensional Sobolev spaces by
\begin{equation}
H^s(\mcO_d) = \prod_{i=1}^{n_d} H^s(\Omega_{d,i}),
\qquad
H^s(\mcO) = \prod_{d=0}^n H^s(\mcO_d)
\label{eq:sobolev-spaces}
\end{equation}
with scalar products
\begin{equation}
(v_d,w_d)_{H^s(\mcO_d)} = \sum_{i=1}^{n_d} (v_{d,i},w_{d,i})_{H^s(\Omega_{d,i})},
\qquad
(v,w)_{H^s(\mcO)} = \sum_{d=0}^n (v_d,w_d)_{H^s(\mcO_d)}
\label{eq:sobolev-inner-products}
\end{equation}
The corresponding norms are denoted $\|v_d\|_{H^s(\mcO_d)}$ and $\|v\|_{H^s(\mcO)}$.
For $d=0$ we set $H^s(\Omega_{0,i}) = \IR$ with norm $\|v\|^2_{H^s(\Omega_{0,i})} = v^2$.
For $s=0$ we write $L^2(\mcO_d) = H^0(\mcO_d)$ and $L^2(\mcO) = H^0(\mcO)$ with scalar products $(v_d,w_d)_{\mcO_d}$ and $(v,w)_\mcO$ and norms $\|v_d\|_{\mcO_d}$ and $\|v\|_\mcO$.

On $\partial\mcO$ we define
\begin{equation}
L^2(\partial\mcO) = \prod_{d=1}^n \prod_{i=1}^{n_d} L^2(\partial\Omega_{d,i})
\label{eq:l2-boundary}
\end{equation}
We equip the components in $\partial\mcO_d$ with the natural $(d-1)$-dimensional measure. Components of dimension at most $d-2$ therefore have measure zero, and hence
\begin{equation}
(v,w)_{\partial\mcO_d} = \sum_{i=1}^{n_d} (v,w)_{\partial\Omega_{d,i}} = \sum_{i=1}^{n_d} (v,w)_{\partial\Omega_{d,i}\cap\Omega_{d-1}} + \sum_{i=1}^{n_d} (v,w)_{\partial\Omega_{d,i}\cap\partial\Omega}
\label{eq:boundary-inner-product}
\end{equation}

\paragraph{Tangential and Normal Vector Fields.}
We say that $a=(a_d)_{d=0}^n$ is a tangential vector field on $\mcO$
if $a_d=(a_{d,i})_{i=1}^{n_d}$ and each $a_{d,i}$ is a tangential
vector field on the manifold $\Omega_{d,i}\in\mcO_d$.

We define the unit exterior normal vector field $\nu$ on $\partial\mcO$ by
\begin{equation}
\nu|_{\partial\Omega_{d,i}} = \nu_{d,i}
\label{eq:normal-definition}
\end{equation}
where $\nu_{d,i}$ is the unit vector field tangent to $\Omega_{d,i}$, orthogonal to $\partial\Omega_{d,i}$, and exterior to $\Omega_{d,i}$. See Figure~\ref{fig:domain-illu-c} for an illustration of the tangential and normal vector fields.

The pointwise dot product $a\cdot b$ of two tangential vector fields $a$ and $b$ on $\mcO$ is the scalar field $(a\cdot b)_d = a_d\cdot b_d$ on each $\mcO_d$, $d=0,\dots,n$. Thus the dot product is defined componentwise on the mixed-dimensional domain.

\paragraph{Extension Operator.}
For each smooth component $\Omega_{d,i}$, fix a smooth prolongation
$\widetilde\Omega_{d,i}$ across $\partial\Omega_{d,i}$ and a bounded
intrinsic extension
\begin{equation}
\widetilde E_{d,i}:H^s(\Omega_{d,i})
\longrightarrow H^s(\widetilde\Omega_{d,i})
\end{equation}
Such an intrinsic extension is constructed for surfaces with boundary
in \cite[Section~3.4]{BuHaLaLaMa15}.
For sufficiently small $\delta>0$, set
\(
U_\delta^n(\Omega_{d,i})
:=
\{x\in\mathbb R^n:
\operatorname{dist}(x,\Omega_{d,i})<\delta\}.
\)
After reducing $\delta$ if necessary, let
\(
p_{d,i}:U_\delta^n(\Omega_{d,i})
\longrightarrow\widetilde\Omega_{d,i}
\)
denote the closest-point projection.
We define the
normal extension by
\begin{equation}
E_{d,i}v
:=
(\widetilde E_{d,i}v)\circ p_{d,i}
\in H^s(U_\delta^n(\Omega_{d,i})),
\qquad 0\le s\le2
\label{eq:extension-operator}
\end{equation}
For $d=n$, this construction is understood as an ordinary extension
to a neighborhood, with $p_{n,i}$ equal to the identity; for $d=0$,
the extension is taken to be constant. The operator $E_{d,i}$ is bounded in $H^s$; see \cite{BuHaLaMa16} for the closest-point construction in arbitrary
codimension. We write $v_{d,i}^e=E_{d,i}v_{d,i}$ and let $v^e=Ev$ denote the
componentwise extension of a mixed-dimensional function $v$. When no ambiguity arises, we suppress the superscript $e$ and use $v$
also for its extension. This prescribed extension does
not rely on the existence of a Sobolev trace from the ambient
neighborhood onto $\Omega_{d,i}$.

\paragraph{Tangential Gradient.}
Let $\nabla_d$ denote the tangential gradient on $\Omega_d$. This operator acts componentwise and differentiates only along the tangent directions of each manifold. We define
\begin{equation}
\nabla v = (\nabla_d v_d)_{d=1}^n
\label{eq:tangential-gradient}
\end{equation}
For $d=0$, we set $\nabla_0v_0=0$.
Let $P_{d,i}(x):\IR^n\to T_x\Omega_{d,i}$ denote the orthogonal
projection onto the tangent space. For smooth $v$,
\begin{equation}
\nabla_{d,i}v
=
P_{d,i}\nabla_{\IR^n}v^e
\qquad\text{on }\Omega_{d,i}
\end{equation}
For $v\in H^1(\Omega_{d,i})$, the tangential gradient is defined
intrinsically in the weak sense, and the above identity follows by
density.

These geometric and functional-analytic constructions allow us to define traces and differential operators componentwise on the mixed-dimensional domain. The coupling between neighboring dimensions is then encoded by the mixed-dimensional interface operators introduced in the next subsection.

\subsection{Mixed-Dimensional Interface Operators}
\label{sec:jump-operators}

The mixed-dimensional structure implies that quantities defined on manifolds of different dimensions interact through traces on adjacent boundaries. The following operators encode these interactions between neighboring dimensions.

\paragraph{Upward Trace-Sum Operator.}
For $d=0,\dots,n-1$, the upward trace-sum operator
\begin{equation}
\llbracket \cdot \rrbracket_d : L^2(\partial \mcO_{d+1}) \to L^2(\mcO_d)
\label{eq:upward-jump-space}
\end{equation}
collects traces from $(d+1)$-dimensional manifolds onto adjacent $d$-dimensional manifolds. It is defined by
\begin{align}
\llbracket v_{d+1} \rrbracket_d \big|_{\Omega_{d,i}}
=\displaystyle
\sum_{j=1}^{n_{d+1}} v_{d+1,j}\big|_{\partial \Omega_{d+1,j} \cap \Omega_{d,i}},
\qquad d=0,\dots,n-1
\label{eq:upward-jump-definition}
\end{align}
Components not incident on $\Omega_{d,i}$ are understood to contribute
zero to the sum.
We also write
$\llbracket v_{d+1}\rrbracket_{d,i}
:=\llbracket v_{d+1}\rrbracket_d|_{\Omega_{d,i}}$
for the corresponding component restriction.
When the dimension index is omitted, $\llbracket v\rrbracket$ denotes the mixed-dimensional field defined by
\begin{equation}
\bigl(\llbracket v\rrbracket\bigr)_d
=
\llbracket v_{d+1}\rrbracket_d,
\qquad d=0,\dots,n-1
\end{equation}
with $\bigl(\llbracket v\rrbracket\bigr)_n=0$.
The operator is algebraically a sum of incident traces. When applied to an oriented conormal quantity, such as $\nu_{d+1}\cdot q_{d+1}$, it gives the usual signed flux jump across a two-sided interface and, more generally, the Kirchhoff sum of incident fluxes at a junction. For quantities whose sign is independent of the conormal orientation, it instead simply sums the incident traces. This operator satisfies the transfer identity
\begin{equation}
(v_{d+1},w_d)_{\partial \mcO_{d+1}\cap\mcO_d}
=
(\llbracket v_{d+1} \rrbracket_d , w_d)_{\mcO_d},
\qquad
w_d \in L^2(\mcO_d)
\label{eq:boundary-domain}
\end{equation}
where $w_d$ is copied to each incident boundary component. Any
lower-dimensional factor appearing inside a trace-sum is understood in the
same way.

\paragraph{Downward Jump Operator.}
For $d=1,\dots,n$, the downward jump operator
\begin{equation}
[\cdot]_d : L^2(\partial \mcO_d) \times L^2(\mcO_{d-1}) \to L^2(\partial \mcO_d)
\label{eq:downward-jump-space}
\end{equation}
measures the mismatch between quantities on neighboring dimensions. For each incident pair $\Omega_{d,i}$ and $\Omega_{d-1,j}$, we define
\begin{align}
[v]_{d,i}\big|_{\partial \Omega_{d,i} \cap \Omega_{d-1,j}}
&=
v_{d,i}\big|_{\partial \Omega_{d,i} \cap \Omega_{d-1,j}}
-
v_{d-1,j}\big|_{\partial \Omega_{d,i} \cap \Omega_{d-1,j}}
\label{eq:downward-jump-definition}
\\
[v]_{d,i}\big|_{\partial \Omega_{d,i} \cap \partial \Omega}
&=
v_{d,i}\big|_{\partial \Omega_{d,i} \cap \partial \Omega}
\label{eq:downward-jump-boundary}
\end{align}
For $d=0$, we set $[v]_0=0$.
When the dimension index is omitted, $[v]$ denotes the field on $\partial\mcO$ whose restriction to $\partial\mcO_d$ is $[v]_d$.

Thus the mixed-dimensional interface operators provide the coupling between the different components of the domain. In particular, only neighboring manifolds whose dimensions differ by one interact.

\subsection{Mixed-Dimensional Differential Operators}
\label{sec:mixed-dimensional-operators}

We now introduce mixed-dimensional analogues of the directional derivative, divergence, and elliptic operators. The essential feature of these operators is that they combine tangential differentiation along each manifold with exchange terms describing transport between neighboring manifolds whose dimensions differ by one.

\paragraph{Directional Derivative.}
Let $\beta$ denote the convection field, which we assume to be a smooth tangential vector field on $\mcO$, i.e. $\beta_{d,i}$ is a smooth tangential vector field on each $\Omega_{d,i}\in\mcO_d$. Let $\nu$ denote the unit exterior normal vector field on $\partial\mcO$ defined in Section \ref{sec:defs}. The directional derivative in the direction $\beta$ is defined componentwise by
\begin{align}
( D_{\beta} v )_{d,i}
=
\begin{cases}
\displaystyle
-\llbracket \nu_1\cdot\beta_1[v]_1\rrbracket_{0,i},
& d=0
\\[6pt]
\beta_{d,i}\cdot\nabla_{d,i} v_{d,i}
-\llbracket
\nu_{d+1}\cdot\beta_{d+1}[v]_{d+1}
\rrbracket_{d,i},
& d=1,\dots,n-1
\\[6pt]
\beta_{n,i}\cdot\nabla_{n,i} v_{n,i},
& d=n
\end{cases}
\label{eq:directional-derivative-component}
\end{align}
The first term represents tangential transport along $\Omega_d$, while the remaining terms describe exchange with adjacent $(d+1)$-dimensional manifolds. Using the mixed-dimensional interface operators, this definition can be written compactly as
\begin{equation}
(D_\beta v)_d
=
\beta_d\cdot\nabla_d v_d
-
\llbracket
\nu_{d+1}\cdot\beta_{d+1}[v]_{d+1}
\rrbracket_d
\label{eq:directional-deriv}
\end{equation}

\paragraph{Total Derivative.}
The total derivative is the operator
\begin{equation}
D:H^1(\mcO)\to L^2(\mcO)\times L^2(\partial\mcO)
\label{eq:total-derivative-space}
\end{equation}
defined by
\begin{equation}
D v_d
=
(\nabla_d v_d,-\nu_d[v]_d),
\qquad d=0,\dots,n
\label{eq:total-derivative}
\end{equation}
This operator may be interpreted as a mixed-dimensional gradient consisting of the tangential gradient $\nabla_d v_d$ together with the jump term $[v]_d$ that describes exchange across neighboring manifolds.

\paragraph{Domain--Boundary Product Space.}
Here and below, $L^2(\mcO)$ and $L^2(\partial\mcO)$ also denote the
corresponding $\IR^n$-valued spaces when appropriate. We equip the
codomain $L^2(\mcO)\times L^2(\partial\mcO)$ of $D$ with the inner
product
\begin{equation}
\bigl(v,w\bigr)_{\mcO\times\partial\mcO}
=
(v_\mcO,w_\mcO)_\mcO
+
(v_{\partial\mcO},w_{\partial\mcO})_{\partial\mcO}
\label{eq:product-space-inner-product}
\end{equation}
where
$v=(v_\mcO,v_{\partial\mcO})$ and
$w=(w_\mcO,w_{\partial\mcO})$.
The inner products on the right-hand side are taken componentwise
using the Euclidean inner product in $\IR^n$. When a tangential vector
field $\beta$ on $\mcO$ has a well-defined boundary trace, we identify
it with the pair
$(\beta,\beta|_{\partial\mcO})$ in this product space.

\paragraph{Divergence Operator.}
The divergence of a tangential vector field $\beta$ is defined componentwise by
\begin{align}
(\Div\beta)_{d,i}
=
\begin{cases}
\displaystyle
-\llbracket \nu_1\cdot\beta_1\rrbracket_{0,i},
& d=0
\\[6pt]
\displaystyle
\nabla_{d,i}\cdot\beta_{d,i}
-\llbracket
\nu_{d+1}\cdot\beta_{d+1}
\rrbracket_{d,i},
& d=1,\dots,n-1
\\[6pt]
\nabla_{n,i}\cdot\beta_{n,i},
& d=n
\end{cases}
\label{eq:Div-jump}
\end{align}
Using the mixed-dimensional interface operator, this can be written compactly as
\begin{equation}
(\Div\beta)_d
=
\nabla_d\cdot\beta_d
-
\llbracket
\nu_{d+1}\cdot\beta_{d+1}
\rrbracket_d
\end{equation}
This definition yields the mixed-dimensional analogue of the standard divergence--gradient duality identity
\begin{equation}
(-\Div\beta,v)_\mcO
=
(\beta,Dv)_{\mcO\times\partial\mcO}
\label{eq:div-totalderivative}
\end{equation}

\paragraph{Laplacian.}
We define the mixed-dimensional negative Laplace operator
\begin{equation}
A:H^1(\mcO)\to H^1(\mcO)^*
\label{eq:laplace-operator}
\end{equation}
variationally by
\begin{equation}
\langle Av,w\rangle_{H^1(\mcO)^*,H^1(\mcO)}
=
(Dv,Dw)_{\mcO\times\partial\mcO},
\qquad v,w\in H^1(\mcO)
\label{eq:laplace-weak}
\end{equation}
On its natural operator domain, $A=D^*D$. Since $-\Div$ is the
adjoint of $D$ for compatible tangential fields, we write this relation
formally as $A=-\Div D$. Using \eqref{eq:total-derivative}, the
variational definition reads
\begin{align}
\langle Av,w\rangle_{H^1(\mcO)^*,H^1(\mcO)}
&=
(\nabla v,\nabla w)_\mcO
+
([v],[w])_{\partial\mcO}
\end{align}
Thus \eqref{eq:laplace-weak} has the same form as the standard weak
form of the Laplacian, while also incorporating the mixed-dimensional
interface jumps directly through $D$. A mixed-dimensional
integration-by-parts formula is given later in
Lemma~\ref{lem:partial-integration}.

\paragraph{Elliptic Operator.}
More generally, for a symmetric tangential positive definite tensor
$\alpha$ on each $\Omega_{d,i}$ and $\partial\Omega_{d,i}$, we define
the elliptic operator
\begin{equation}
A_\alpha:H^1(\mcO)\to H^1(\mcO)^*
\label{eq:elliptic-operator-definition}
\end{equation}
by
\begin{align}
\langle A_\alpha v,w\rangle_{H^1(\mcO)^*,H^1(\mcO)}
&=
(\alpha Dv,Dw)_{\mcO\times\partial\mcO}
\\
&=
(\alpha\nabla v,\nabla w)_\mcO
+
(\alpha_{\nu\nu}[v],[w])_{\partial\mcO}
\label{eq:elliptic-form}
\end{align}
where $\alpha_{\nu\nu}=\nu\cdot\alpha\cdot\nu$.
Formally, $A_\alpha=-\Div(\alpha D)$. For sufficiently smooth $v$,
the $L^2(\mcO)$-part of this operator can be written componentwise as
\begin{align}
(A_\alpha v)_d
&=
-\nabla_d\cdot(\alpha_d\nabla_d v_d)
-
\llbracket
\alpha_{d+1,\nu\nu}[v]_{d+1}
\rrbracket_d
\label{eq:elliptic-operator}
\end{align}
while the remaining boundary functional is represented by the conormal
interface and boundary conditions in \eqref{eq:problem-b}--\eqref{eq:problem-c}.
In the diffusive case, we assume that $\alpha$ is uniformly positive definite in the sense that there exist constants $\alpha_0,\alpha_1>0$ such that
\begin{align}
\langle A_\alpha v,v\rangle_{H^1(\mcO)^*,H^1(\mcO)}
&=
(\alpha\nabla v,\nabla v)_\mcO
+
(\alpha_{\nu\nu}[v],[v])_{\partial\mcO}
\\
&\ge
\alpha_0\|\nabla v\|_\mcO^2
+
\alpha_1\|[v]\|_{\partial\mcO}^2
\label{eq:elliptic-positivity}
\end{align}
We define
\begin{align}
\epsilon = \min\{\alpha_0,\alpha_1\}
\label{eq:diffusion-scale}
\end{align}
so that the elliptic operator controls the diffusive terms in the sense that
\begin{align}
\langle A_\alpha v,v\rangle_{H^1(\mcO)^*,H^1(\mcO)}
\ge
\epsilon\big(\|\nabla v\|_\mcO^2+\|[v]\|_{\partial\mcO}^2\big)
\end{align}
In this sense, $\epsilon$ represents the strength of the diffusive part of the operator under the global ellipticity assumption. We also consider separately the globally degenerate case $\alpha\equiv0$.

\subsection{Partial Integration Formula}
\label{sec:partial-integration}

To formulate a partial integration formula for $D_\beta$ we split the boundary of the mixed-dimensional domain into the parts that lie on the exterior boundary $\partial\Omega$ and the parts that lie in the interior of $\Omega$. We therefore define
\begin{equation}
\partial \mcO_B = \partial \mcO \cap \partial \Omega = \bigsqcup_{d,i} (\partial \Omega_{d,i} \cap \partial \Omega),
\qquad
\partial \mcO_I = \partial \mcO \setminus \partial \Omega = \bigsqcup_{d,i} (\partial \Omega_{d,i} \setminus \partial \Omega)
\label{eq:boundary-decomposition}
\end{equation}
Thus $\partial\mcO_I \subset \mcO$, while $\partial\mcO_B \not\subset \mcO$.

Given a tangential vector field $\beta$, we introduce the space
\begin{equation}
V_{\beta} = \{ v \in H^1(\mcO) : \|\beta \cdot \nabla v\|_{\mcO} < \infty \}
\label{eq:V-beta}
\end{equation}

\begin{lem}[Partial Integration]
\label{lem:partial-integration}
For a smooth tangential vector field $\beta$ on $\mcO$ and $v\in V_\beta$,
\begin{equation}
\Div(\beta v) = D_\beta v + (\Div\beta) v
\label{eq:Div-formula}
\end{equation}
and for $v,w\in V_\beta$,
\begin{equation}
(D_\beta v, w)_\mcO = -(v,D_\beta w)_\mcO - ((\Div\beta) v,w)_\mcO + (\nu \cdot \beta [v],[w])_{\partial \mcO_I} + (\nu \cdot \beta v, w)_{\partial \mcO_B}
\label{eq:Dbeta-partial-integration}
\end{equation}
where $\nu$ is the exterior unit normal vector field on $\partial\mcO$.
\end{lem}

\begin{proof}
We first prove \eqref{eq:Div-formula}. Using the mixed-dimensional interface operators and the product rule on each component, we obtain
\begin{align}
\Div(\beta v)
&=
\nabla \cdot (\beta v) - \llbracket \nu \cdot \beta v \rrbracket
\\
&=
(\nabla \cdot \beta) v + \beta \cdot \nabla v - \llbracket \nu \cdot \beta [v] \rrbracket - \llbracket \nu \cdot \beta \rrbracket v
\\
&=
(\Div\beta) v + D_\beta v
\end{align}
where we used the identity
\begin{align}
-\llbracket \nu \cdot \beta v \rrbracket_d
&=
-\llbracket \nu_{d+1} \cdot \beta_{d+1} v_{d+1} \rrbracket_d
\\
&=
\llbracket \nu_{d+1} \cdot \beta_{d+1} (v_d - v_{d+1}) \rrbracket_d - \llbracket \nu_{d+1} \cdot \beta_{d+1} \rrbracket_d v_d
\\
&=
-\llbracket \nu_{d+1} \cdot \beta_{d+1} [v]_{d+1} \rrbracket_d - \llbracket \nu_{d+1} \cdot \beta_{d+1} \rrbracket_d v_d
\end{align}
We next prove \eqref{eq:Dbeta-partial-integration}. We apply Green's formula componentwise, split the resulting boundary contributions into their internal and external parts, and use the transfer identity \eqref{eq:boundary-domain} on the internal interfaces. This gives
\begin{align}
(D_\beta v,w)_\mcO
&=
\sum_{d=0}^n (\beta_d \cdot \nabla_d v_d,w_d)_{\mcO_d} - \sum_{d=0}^{n-1} (\llbracket \nu_{d+1}\cdot \beta_{d+1} [v]_{d+1}\rrbracket_d, w_d)_{\mcO_d}
\\
&=
\sum_{d=0}^n \Big(
-(v_d,\beta_d \cdot \nabla_d w_d)_{\mcO_d} - ((\nabla_d \cdot \beta_d) v_d, w_d)_{\mcO_d} + (\nu_d \cdot \beta_d v_d,w_d)_{\partial \mcO_d}
\Big)
\\
&\qquad - \sum_{d=0}^{n-1} (\llbracket \nu_{d+1}\cdot \beta_{d+1} [v]_{d+1}\rrbracket_d, w_d)_{\mcO_d} \notag
\\
&=
\sum_{d=0}^n \Big(
-(v_d,\beta_d \cdot \nabla_d w_d)_{\mcO_d} - ((\nabla_d \cdot \beta_d) v_d, w_d)_{\mcO_d} + (\nu_d \cdot \beta_d v_d,w_d)_{\partial \mcO_d\cap \partial \Omega}
\Big)
\\
&\qquad + \sum_{d=0}^{n-1}\Big(
(\nu_{d+1} \cdot \beta_{d+1} v_{d+1},w_{d+1})_{\partial \mcO_{d+1} \cap \mcO_d}
- (\llbracket \nu_{d+1}\cdot \beta_{d+1} [v]_{d+1}\rrbracket_d, w_d)_{\mcO_d}
\Big) \notag
\\
&=
\sum_{d=0}^n \Big(
-(v_d,\beta_d \cdot \nabla_d w_d)_{\mcO_d} - ((\nabla_d \cdot \beta_d) v_d, w_d)_{\mcO_d} + (\nu_d \cdot \beta_d v_d,w_d)_{\partial \mcO_d\cap \partial \Omega}
\Big)
\\
&\qquad + \sum_{d=0}^{n-1}
(\llbracket \nu_{d+1} \cdot \beta_{d+1} [w]_{d+1}\rrbracket_d, v_d)_{\mcO_d}
+ \sum_{d=0}^{n-1}
(\nu_{d+1}\cdot \beta_{d+1} [v]_{d+1}, [w]_{d+1})_{\partial\mcO_{d+1}\cap\mcO_d} \notag
\\
&\qquad + \sum_{d=0}^{n-1} (\llbracket \nu_{d+1}\cdot \beta_{d+1} \rrbracket_d v_d, w_d)_{\mcO_d} \notag
\\
&=
-(v, D_\beta w)_\mcO - ((\Div\beta) v,w)_\mcO + (\nu \cdot \beta [v],[w])_{\partial \mcO_I} + (\nu \cdot \beta v,w)_{\partial \mcO_B}
\end{align}
Here we used the algebraic identity
\begin{align}
v_{d+1} w_{d+1}
&=
(v_{d+1} - v_d) w_{d+1} + v_d(w_{d+1} - w_d) + v_d w_d
\\
&=
[v]_{d+1} [w]_{d+1} + [v]_{d+1} w_d + v_d [w]_{d+1} + v_d w_d
\end{align}
This completes the proof.
\end{proof}

\subsection{The Mixed-Dimensional Model Problem}
\label{sec:mixed-dimensional-model-problem}

We now formulate the convection-diffusion problem first in componentwise form and then rewrite it in terms of the mixed-dimensional operators introduced above. The abstract formulation makes the coupling structure transparent and prepares for the variational analysis in the next subsection.

\paragraph{Componentwise Formulation.}

Find $u_{d,i}:\Omega_{d,i} \rightarrow \IR$ such that, on each component $\Omega_{d,i}$, the unknown satisfies a tangential convection-diffusion equation coupled to adjacent higher-dimensional components through flux exchange terms,
\begin{alignat}{3}\label{eq:strong-form-bulk}
\nabla_{d,i}\cdot (- \alpha_{d,i} \nabla_{d,i}u_{d,i} + \beta_{d,i} u_{d,i} )  + \kappa_{d,i} u_{d,i} \qquad\quad
\nonumber\\
 - \llbracket \nu_{d+1} \cdot ( - \alpha_{d+1} \nabla_{d+1}u_{d+1} + \beta_{d+1} u_{d+1})\rrbracket_{d,i} 
& = f_{d,i}& \qquad &\text{in $\Omega_{d,i}$}
\\ \label{eq:strong-form-inflow-internal}
\nu_{d,i} \cdot \alpha_{d,i} \nabla_{d,i}u_{d,i} + (\alpha_{d,ii} + |(\nu_{d,i} \cdot \beta_{d,i})_-|) [ u ]_{d,i} & = 0 & \qquad &\text{on $\partial \Omega_{d,i} \setminus \partial \Omega$}
\\ \label{eq:strong-form-inflow-boundary}
\nu_{d,i} \cdot \alpha_{d,i} \nabla_{d,i}u_{d,i} + (\alpha_{d,ii} + |(\nu_{d,i} \cdot \beta_{d,i})_-|)  (u_{d,i} - g_{d,i}) &=0 & \qquad  &\text{on 
$ \partial \Omega_{d,i} \cap \partial \Omega$}
\end{alignat}
where $\alpha_{d, ii} = \nu_{d,i} \cdot \alpha_{d, i} \cdot \nu_{d, i}$ and, for any scalar $v$, we use
$|v_+|=\max(v,0)$ and $|v_-|=-\min(v,0)$ for the absolute values of its positive and negative parts, respectively.

\paragraph{Abstract Formulation.}

Using the definitions of the elliptic operator \eqref{eq:elliptic-operator}, the directional derivative \eqref{eq:directional-deriv}, and the divergence \eqref{eq:Div-jump}, together with the interface condition \eqref{eq:strong-form-inflow-internal}, we can rewrite the left-hand side of \eqref{eq:strong-form-bulk} more compactly as
\begin{align}
&\nabla_{d,i}\cdot ( - \alpha_{d,i} \nabla_{d,i}u_{d,i} + \beta_{d,i} u_{d,i} ) + \kappa_{d,i} u_{d,i} \notag
\\&\qquad \notag
- \llbracket \nu_{d+1} \cdot ( -\alpha_{d+1} \nabla_{d+1}u_{d+1} + \beta_{d+1} u_{d+1})\rrbracket_{d,i}
\\
&\qquad\qquad = (A_\alpha u)_{d,i} + \beta_{d,i}\cdot\nabla_{d,i}u_{d,i} - \llbracket \nu_{d+1}\cdot\beta_{d+1}[u]_{d+1}\rrbracket_{d,i}
\\
&\qquad\qquad \quad + \big( \kappa_{d,i} + \nabla_{d,i}\cdot\beta_{d,i} - \llbracket \nu_{d+1}\cdot\beta_{d+1}\rrbracket_{d,i} \big)u_{d,i} - \llbracket |(\nu_{d+1}\cdot\beta_{d+1})_-| [u]_{d+1} \rrbracket_{d,i}
\notag
\\
&\qquad\qquad = (A_\alpha u + D_\beta u + (\kappa + \Div\beta)u - \llbracket |(\beta_\nu)_-| [u]\rrbracket)_{d,i}
\label{eq:abstract-rewrite}
\end{align}
We introduce the shorthand notations $\beta_\nu := \nu \cdot \beta$ for the normal component of the convection field, 
$B_{\alpha}^{\beta_\pm} := \alpha_{\nu\nu} + |(\beta_\nu)_{\pm}|$ and
$B_{\alpha}^{\beta} := \alpha_{\nu\nu} + |\beta_\nu|$.
The mixed-dimensional model problem corresponding to \eqref{eq:strong-form-bulk}--\eqref{eq:strong-form-inflow-boundary} is: find $u:\mcO \rightarrow \IR$ such that
\begin{alignat}{3}\label{eq:problem-a}
A_\alpha u +
D_\beta u + \gamma u - \llbracket |(\beta_\nu)_-| [ u ]\rrbracket & = f   &   \qquad &\text{in $\mcO$}
\\ \label{eq:problem-b}
\nu \cdot \alpha\nabla u + B_{\alpha}^{\beta_-} [ u ] &= 0 & \qquad &\text{on $\partial \mcO_I$}
\\ \label{eq:problem-c}
\nu \cdot \alpha\nabla u + B_{\alpha}^{\beta_-} ( u - g) &= 0 & \qquad &\text{on $\partial \mcO_B$}
\end{alignat}
where $\gamma = \kappa + \Div \beta$, or equivalently in component form
\begin{equation}\label{eq:gamma-component}
\gamma_d = \kappa_d + \nabla_d \cdot \beta_d -  \llbracket \nu_{d+1} \cdot \beta_{d+1}   \rrbracket_{d}
\end{equation}

\begin{rem}[Darcy's Law and Mass Conservation]
Convection-diffusion equations model conservation of scalar quantities, and in porous-media applications the flux is often interpreted through Darcy's law together with mass conservation. In the present mixed-dimensional setting we define $\sigma_{d,i} = - \alpha_{d,i} \nabla_{d,i}u_{d,i} + \beta_{d,i} u_{d,i}$ and $\kappa_{d,i} = 0$. Then \eqref{eq:strong-form-bulk} reduces to 
\begin{alignat}{2}
     &\sigma_{d,i}  = - \alpha_{d,i} \nabla_{d,i}u_{d,i} + \beta_{d,i} u_{d,i}   & \qquad & \text{ in } \Omega_{d,i} \label{eq:Darcylaw}\\
     &\nabla_{d,i}\cdot \sigma_{d,i} -  \llbracket \nu_{d+1} \cdot \sigma_{d+1}\rrbracket_{d,i}  = f_{d,i}  & \qquad & \text{ in } \Omega_{d,i} \label{eq:conservation_mass}
\end{alignat}
Equation \eqref{eq:Darcylaw} represents Darcy's law on $\Omega_{d,i}$, with diffusive flux $- \alpha_{d,i} \nabla_{d,i}u_{d,i}$ and convective flux $\beta_{d,i} u_{d,i}$, whereas \eqref{eq:conservation_mass} expresses conservation of mass on $\Omega_{d,i}$. In particular, when $d=n$ there are no higher-dimensional neighbors, and
\eqref{eq:conservation_mass} reduces to
\(\nabla_{n,i}\cdot \sigma_{n,i} = f_{n,i}\) in $\Omega_{n,i}$.
\end{rem}

\begin{rem}[Pure Diffusion Interface Conditions]
The interface conditions used for pure diffusive flow in fractured porous media are typically of Robin type, as they allow for a range of transfer regimes across the fracture, cf. \cite{BuHaLa20, MaJaRob05, AnBoHu09}.
To relate our interface condition \eqref{eq:strong-form-inflow-internal}, or equivalently \eqref{eq:problem-b}, to those used in \cite[Eqns.~(23)--(24)]{BuHaLa20}, we consider a simple geometry in which $\Omega=[0,1]^2$ is split by a fracture along $x=1/2$ into two bulk domains $\Omega_{2,1}$ and $\Omega_{2,2}$ and the fracture $\Omega_{1,1}$. The governing bulk and fracture equations then agree with \cite[Eqns.~(1)--(2)]{BuHaLa20}. With $[u]_{2,1} = u_{2,1} - u_{1,1}$ and $[u]_{2,2} = u_{2,2} - u_{1,1}$ on the fracture, the interface conditions take the form
\begin{alignat}{2}
    \nu_{2,1} \cdot \alpha_{2,1} \nabla_{2,1}u_{2,1} + \nu_{2,1} \cdot \alpha_{2,1}\nu_{2,1}( u_{2,1} - u_{1,1}) & = 0 & \qquad &\text{on $\Omega_{1,1}$}\label{eq:bulk1-unitsq}
    \\
    \nu_{2,2} \cdot \alpha_{2,2} \nabla_{2,2}u_{2,2} + \nu_{2,2} \cdot \alpha_{2,2}\nu_{2,2}( u_{2,2} - u_{1,1}) & = 0 & \qquad &\text{on $\Omega_{1,1}$}\label{eq:bulk2-unitsq}
\end{alignat}
Assume $\Theta = \nu_{2,1} \cdot \alpha_{2,1}\nu_{2,1} = \nu_{2,2} \cdot \alpha_{2,2}\nu_{2,2} = \bfc_{\Omega_{1,1}}/t$, where $\bfc_{\Omega_{1,1}}$ is the permeability coefficient across $\Omega_{1,1}$ and $t$ is the thickness of $\Omega_{1,1}$. Define the upward trace-sum of the conormal fluxes, the scalar bulk jump $J_u$, and the averages across $\Omega_{1,1}$ by
\begin{alignat}{2}
    &\llbracket \nu\cdot \alpha\nabla u \rrbracket_{1,1} = \nu_{2,1} \cdot \alpha_{2,1}\nabla_{2,1} u_{2,1} + \nu_{2,2} \cdot \alpha_{2,2}\nabla_{2,2} u_{2,2}, \quad   &&J_u = u_{2,1} - u_{2,2}
    \\
    &\langle \nu\cdot \alpha\nabla u \rangle = \frac{1}{2}( \nu_{2,1} \cdot \alpha_{2,1} \nabla_{2,1}u_{2,1} - \nu_{2,2} \cdot \alpha_{2,2} \nabla_{2,2}u_{2,2} ),  \quad &&\langle u \rangle = \frac{1}{2}( u_{2,1} + u_{2,2})
\end{alignat}
Adding and subtracting \eqref{eq:bulk1-unitsq}--\eqref{eq:bulk2-unitsq} yields
\begin{align}
    &\llbracket \nu\cdot \alpha\nabla u \rrbracket_{1,1} + 2\Theta ( \langle u \rangle - u_{1,1}) = 0 ,
    \qquad
    \langle \nu\cdot \alpha\nabla u \rangle + \frac{\Theta}{2} J_u = 0
\end{align}
These relations are analogous to the interface conditions used in \cite[Eqns.~(5a)--(5b)]{AnFaRuVe19} and \cite[Eqns.~(23)--(24)]{BuHaLa20} for the particular choice $\xi=1$.
\end{rem}

\subsection{Well-posedness}
\label{sec:well-posedness}

In the estimates below, we write $a\lesssim b$ if $a\le Cb$, where $C>0$ is independent of the mesh size $h$ and the diffusion scale $\epsilon$, but may depend on the fixed mixed-dimensional geometry, the fixed coefficient-comparison and stabilization constants, and $c_0$ when applicable. The symbols $\gtrsim$ and $\sim$ are defined analogously, with $a\sim b$ meaning that both $a\lesssim b$ and $a\gtrsim b$ hold.

\paragraph{Weak Formulation.}
We introduce the energy space
\begin{equation}\label{eq:V-alpha-beta}
V_{\alpha,\beta}
=
\{ v \in H^1(\mcO) : \|v\|_{\alpha,\beta} < \infty \}
\end{equation}
equipped with the norm
\begin{align}\label{eq:continuous-norm}
\| v\|_{\alpha,\beta}^2
=
\epsilon ( \|\nabla v\|_{\mcO}^2 + \|[v]\|_{\partial \mcO}^2 )
+ \|v\|_{\mcO}^2
+ \||\beta_\nu|^{1/2}[v]\|_{\partial \mcO_I}^2
+ \||\beta_\nu|^{1/2} v\|_{\partial \mcO_B}^2
\end{align}
where the parameter $\epsilon$ denotes the diffusion scale defined in \eqref{eq:diffusion-scale} in the diffusive case, while $\epsilon=0$ corresponds to the degenerate pure-convection case. The $\epsilon$-weighted terms in the norm provide a uniform scalar control of the gradient and jump contributions, which is convenient in the continuity estimate for the convection operator. Although the elliptic part is naturally expressed in terms of $\alpha$, the above norm is chosen to balance the diffusive and convective contributions in a form suitable for convection-dominated regimes.
This norm combines the diffusive bulk contribution, the interface jumps, and the inflow-weighted boundary terms induced by the convection field.
For $\epsilon>0$, it is equivalent to the product $H^1(\mcO)$-norm,
and $V_{\alpha,\beta}$ is therefore a Hilbert space.

Using \eqref{eq:elliptic-operator}, integration by parts on each component, and the boundary conditions \eqref{eq:problem-b}--\eqref{eq:problem-c}, we arrive at the weak formulation of \eqref{eq:problem-a}--\eqref{eq:problem-c}: find $u \in V_{\alpha,\beta}$ such that
\begin{equation}\label{eq:weak-problem}
a(u,v)  = l(v) \qquad \forall v \in V_{\alpha,\beta}
\end{equation}
where the forms are defined by
\begin{align}
a(v,w) &= (\alpha \nabla v,\nabla w)_\mcO
+ (D_\beta v,w)_\mcO + (\gamma v,w)_\mcO 
\label{eq:continuousbilinear} 
\\&\quad \nonumber
+ ( B_{\alpha}^{\beta_-} [ v ],[w])_{\partial \mcO_I}
+ ( B_{\alpha}^{\beta_-} v , w)_{\partial \mcO_B}
\\
l(w) &= (f,w)_\mcO +  ( B_{\alpha}^{\beta_-} g, w)_{\partial \mcO_B}
\label{eq:continuouslinear}
\end{align}
Using Lemma \ref{lem:partial-integration} we obtain the following stability 
result. 

\begin{lem}\label{lem:coercivity}{\bf (Coercivity)}
If there is a constant $c_0>0$ such that 
\begin{equation}\label{eq:assumption-coeff}
c_0 \leq  \essinf_{x\in \mcO}(2 \kappa + \Div \beta)
\end{equation}
then 
\begin{equation}\label{eq:coercivity}
a(v,v)
\ge
\frac{1}{2}\min\{1,c_0\}\|v\|_{\alpha,\beta}^2
\end{equation}
for all $v\in V_{\alpha,\beta}$.
\end{lem}

\begin{proof}
To prove \eqref{eq:coercivity}, we first note from \eqref{eq:Dbeta-partial-integration} that 
\begin{equation}\label{eq:partial-integration-identity}
2(D_\beta v,v)_\mcO 
= 
-((\Div \beta) v, v)_\mcO 
+(\beta_\nu [v],[v])_{\partial \mcO_I}
+ (\beta_\nu v, v)_{\partial \mcO_B}
\end{equation}
Now, using \eqref{eq:partial-integration-identity}, the definition $\gamma = \kappa + \Div \beta$, which gives $2\gamma-\Div\beta=2\kappa+\Div\beta$, and the decomposition $\beta_\nu = |(\beta_\nu)_+| - |(\beta_\nu)_-|$, we obtain
\begin{align}
2a(v,v) &= 2(\alpha \nabla v,\nabla v)_\mcO+2(D_\beta v,v)_\mcO + 2(\gamma v, v)_\mcO
+ 2(B_{\alpha}^{\beta_-} [ v ],[v])_{\partial \mcO_I}
\\ \nonumber 
&\qquad
+ 2(B_{\alpha}^{\beta_-} v , v)_{\partial \mcO_B}
\\
&= 2 ( (\alpha \nabla v,\nabla v)_\mcO +  (\alpha_{\nu\nu} [ v ],[v])_{\partial \mcO} ) +  ( (2\gamma- \Div \beta) v, v )_\mcO 
\\  \nonumber
&\qquad 
+(\beta_\nu [v],[v])_{\partial \mcO_I} + (\beta_\nu v, v)_{\partial \mcO_B} \\  \nonumber
&\qquad
+ 2(|(\beta_\nu)_-| [ v ],[v])_{\partial \mcO_I}
+  2( |(\beta_\nu)_-|  v , v)_{\partial \mcO_B}
\\
&= 2 ( (\alpha \nabla v,\nabla v)_\mcO +  (\alpha_{\nu\nu} [ v ],[v])_{\partial \mcO} ) + ( (2\kappa  +  \Div \beta) v, v )_\mcO 
\\ \nonumber
&\qquad 
+ ((|(\beta_\nu)_+| - |(\beta_\nu)_-|)[v],[v])_{\partial \mcO_I} + ((|(\beta_\nu)_+| - |(\beta_\nu)_-|) v, v)_{\partial \mcO_B}  
\\ \nonumber
&\qquad 
+ 2( |(\beta_\nu)_-| [ v ],[v])_{\partial \mcO_I}
+  2( |(\beta_\nu)_-|  v , v)_{\partial \mcO_B}
\\
&= 2 ( (\alpha \nabla v,\nabla v)_\mcO +  (\alpha_{\nu\nu} [ v ],[v])_{\partial \mcO} ) + ( (2\kappa  +  \Div \beta) v, v )_\mcO 
\\ \nonumber
&\qquad 
+ ( |\beta_\nu| [ v ],[v])_{\partial \mcO_I} + ( |\beta_\nu|  v , v)_{\partial \mcO_B}
\\
&\ge  
2\epsilon ( \|\nabla v\|_{\mcO}^2 + \|[v]\|_{\partial \mcO}^2 ) + c_0 \| v \|^2_{\mcO} 
+ \| |\beta_\nu|^{1/2} [ v ] \|^2_{\partial \mcO_I} + \|  |\beta_\nu|^{1/2} v  \|^2_{\partial \mcO_B}
\\ 
&\ge \min\{1,c_0\}\|v\|_{\alpha,\beta}^2
\end{align}
where at the second-last step we used the positivity \eqref{eq:elliptic-positivity} and assumption \eqref{eq:assumption-coeff}.
\end{proof}

\begin{rem}[Role of (\ref{eq:assumption-coeff})]
\label{rem:coercivity-poincare}
Condition~\eqref{eq:assumption-coeff} is the sufficient condition used
below to control the $L^2$-term in \eqref{eq:continuous-norm} and hence
to obtain coercivity. Positive diffusion alone does not give the stated
coercivity result by the present argument. In the diffusive case, it may
be possible to replace \eqref{eq:assumption-coeff} by an appropriate
mixed-dimensional Poincaré inequality, using techniques similar to
\cite{MR4682785}, where such an inequality was established for a pure
diffusion problem with Robin-type interface conditions. We do not
establish that alternative for the general setting considered here.
\end{rem}

In the rest of the article, we make the following assumptions on the coefficients
\begin{align}\label{eq:coefficients-behavior}
0\le \epsilon\lesssim 1,
\
\alpha_{\infty} = \|\alpha \|_{L^\infty(\mcO)} \sim \epsilon,
\
\beta_{\infty} = \|\beta \|_{L^\infty(\mcO)} \sim 1,
\
\gamma_{\infty} = \|\gamma \|_{L^\infty(\mcO)} \sim 1
\end{align}
and focus on the convection-dominated regime. For the diffusive analysis,
we assume that the global ellipticity condition
\eqref{eq:elliptic-positivity} holds with $\epsilon>0$ and that the
coefficient scaling above holds on the full mixed-dimensional domain.
The globally degenerate case $\alpha\equiv0$ is treated separately in
Section~\ref{sec:errorbound_pureconvection}, conditionally on the
existence of a sufficiently regular solution. Configurations in which
diffusion vanishes on some components but remains positive on others
are not covered by the present analysis, although the discrete method
is still well defined and such configurations are considered
numerically in Section~\ref{sec:numerics}. Except where explicitly
stated otherwise, the stability and error estimates below also assume
the coercivity condition \eqref{eq:assumption-coeff}.

We use the following trace-sum estimates:
\begin{align}
\|\llbracket \alpha_{\nu\nu}[v]\rrbracket\|_\mcO
&\lesssim \alpha_\infty\|[v]\|_{\partial\mcO_I}
\lesssim \epsilon\|[v]\|_{\partial\mcO_I}
\label{eq:alpha-interfacebnd}
\\
\|\llbracket \theta[v]\rrbracket\|_\mcO
&\lesssim \beta_\infty^{1/2}
\||\beta_\nu|^{1/2}[v]\|_{\partial\mcO_I},
\qquad
\theta\in\{\beta_\nu,|(\beta_\nu)_-|\}
\label{eq:beta-interfacebnd}
\end{align}
These estimates follow from Cauchy--Schwarz over the finitely many incident
traces and the coefficient bounds in \eqref{eq:coefficients-behavior}.

To prove continuity of $a(\cdot,\cdot)$, we apply the definition of the directional derivative $D_\beta$ together with the Cauchy--Schwarz inequality to obtain
\begin{align}
    |(D_\beta v, w)_\mcO| & \le \|\beta \cdot \nabla v - \llbracket \beta_\nu [v]\rrbracket\|_\mcO \| w\|_\mcO
    \\
    &\le ( \beta_{\infty}\|\nabla v\|_\mcO + \|\llbracket \beta_\nu [v]\rrbracket\|_\mcO ) \| w\|_\mcO
    \\
    &\lesssim ( \beta_{\infty}\|\nabla v\|_\mcO + \beta_{\infty}^{1/2} \| |\beta_\nu|^{1/2} [v]\|_{\partial\mcO_I} ) \| w\|_\mcO
    \label{eq:directionalderivbound}
\end{align}
which together with \eqref{eq:alpha-interfacebnd}--\eqref{eq:beta-interfacebnd} and the assumptions in \eqref{eq:coefficients-behavior} for $\epsilon>0$ imply
\begin{align}
    |a(v,w)|&\lesssim ( \epsilon\|\nabla v\|_\mcO\|\nabla w\|_\mcO +  \beta_{\infty}\|\nabla v\|_\mcO \|w\|_\mcO) + \epsilon\|[v]\|_{\partial\mcO_I}\|[w]\|_{\partial\mcO_I}
    \\
   & \qquad + \epsilon\|v\|_{\partial\mcO_B}\|w\|_{\partial\mcO_B} +  \beta_{\infty}^{1/2} \| |\beta_\nu|^{1/2} [v]\|_{\partial\mcO_I}  \| w\|_\mcO + \gamma_{\infty}\| v\|_\mcO\| w\|_\mcO  \nonumber
   \\
   & \qquad + \| |\beta_\nu|^{1/2} [v] \|_{\partial\mcO_I}\| |\beta_\nu|^{1/2} [w]\|_{\partial\mcO_I} + \| |\beta_\nu|^{1/2} v \|_{\partial\mcO_B}\| |\beta_\nu|^{1/2} w\|_{\partial\mcO_B}     \nonumber
   \\
   & \lesssim \epsilon^{-1/2}\|v\|_{\alpha,\beta}\|w\|_{\alpha, \beta}   \label{eq:a_bounded}
\end{align}

\begin{thm}[Well-posedness under the Coercivity Condition]
\label{thm:well-posedness}
Assume that the global ellipticity condition
\eqref{eq:elliptic-positivity} and the coefficient conditions
\eqref{eq:coefficients-behavior} hold with $\epsilon>0$, and that the
coercivity condition \eqref{eq:assumption-coeff} holds for some
$c_0>0$. If $f\in L^2(\mcO)$ and $g\in L^2(\partial \mcO_B)$, then the
problem \eqref{eq:weak-problem} admits a unique weak solution
$u\in V_{\alpha,\beta}$ such that
\begin{align}\label{eq:stabilt_ctssol}
    \|u\|_{\alpha,\beta}
    \lesssim \min\{1,c_0\}^{-1}
    \Big(\|f\|_\mcO
    + \|(B_{\alpha}^{\beta})^{1/2}g\|_{\partial\mcO_B}\Big)
\end{align}
\end{thm}
\begin{proof}
The well-posedness result is an immediate consequence of Lemma~\ref{lem:coercivity}, continuity \eqref{eq:a_bounded}, and the Lax--Milgram lemma. Since \eqref{eq:weak-problem} holds for all $v\in V_{\alpha,\beta}$, choose $v=u$ in \eqref{eq:weak-problem} and invoke the coercivity \eqref{eq:coercivity} to obtain
\begin{equation}\label{eq:stability-proof-estimate}
\frac{1}{2}\min\{1,c_0\}\| u\|_{\alpha,\beta}^2
\le a(u,u)
\lesssim \Big( \|f\|_\mcO
+ \|(B_{\alpha}^{\beta})^{1/2}g\|_{\partial\mcO_B} \Big)
\|u\|_{\alpha,\beta}
\end{equation}
which completes the proof of \eqref{eq:stabilt_ctssol}.
\end{proof}

\begin{rem}[Pure-Convection Limit]
When $\epsilon=0$, the diffusive control of tangential derivatives is lost, and the continuity argument above no longer applies in the stated energy space. A related well-posedness result for an embedded transport problem is given in \cite[Proposition~2.1]{BuHaLaZa19}. Extending that argument to the general mixed-dimensional setting considered here would require an appropriate transport graph space and additional assumptions on the flow, inflow boundaries, and coupling at junctions. We do not pursue such an analysis here. The pure-convection error estimates below are therefore conditional on the existence of a sufficiently regular solution.
\end{rem}

\begin{rem}[Junctions without Lower-Dimensional Flow]
\label{rem:continuity-sol}
Some discrete fracture-network flow models treat a lower-dimensional
intersection only as a junction and impose continuity of the
pressure across it; see, e.g., \cite{MR3038028,MR3595292}.
Such a model is not obtained from the present formulation merely by setting
the local tangential transport and reaction coefficients, together with the source,
to zero on the junction component. This removes transport on the junction
and reduces its equation to a balance of the incident fluxes, but the
Robin-type coupling remains and still permits differences between the
incident traces and the junction variable. Continuity must therefore
additionally be imposed by modifying the coupling formulation or the
underlying function space.

Related CutFEM constructions illustrate possible alternatives. The Nitsche
formulation in \cite{BuHaLaSa19} weakly enforces continuity between fracture
branches at junctions while incorporating the corresponding Kirchhoff flux
balance. Hybridized CutFEM formulations instead couple independent subdomain
spaces through a skeletal variable using Nitsche terms
\cite{BuElHaLaLa19}. Adapting either mechanism to the present
convection--diffusion problem would require modified coupling conditions and
a separate analysis.
\end{rem}

\section{The Cut Finite Element Method}
\subsection{Mesh and Finite Element Spaces}
Let $\Omega^0$ be a polygonal domain such that $\Omega\subset \Omega^0$, and let $\mcT_h^0$ be a family of quasi-uniform meshes on $\Omega^0$ with elements $T$ and mesh parameter $h\in(0,h_0]$. Let $V_h^0$ denote a finite element space defined on $\mcT_h^0$.
For each manifold $\Omega_{d,i}\in\mcO_d$ we define the active mesh
\begin{equation}\label{eq:mesh-active}
\mcT_{h,d,i} = \{ T \in \mcT_h^0 : T\cap \Omega_{d,i} \neq \emptyset \}
\end{equation}
and define the corresponding finite element space $V_{h,d,i}=V_h^0|_{\mcT_{h,d,i}}$, see Figure~\ref{fig:meshes}.

The global finite element space on $\mcO$ is defined as the product space
\begin{equation}
V_h = \prod_{d=0}^n V_{h,d}, \qquad
V_{h,d} = \prod_{i=1}^{n_d} V_{h,d,i}
\end{equation}

\paragraph{Locally Refined Meshes.}
We assume global quasi-uniformity below so that the stabilization parameters and estimates can be expressed using a single mesh size $h$. The CutFEM construction can be adapted to a shape-regular, locally graded background mesh by replacing $h$ with the local element diameters $h_T$. Since all active meshes are restrictions of the same background mesh, refinement near a lower-dimensional component automatically refines the corresponding portions of the adjacent higher-dimensional active meshes, while elements farther away may remain coarser. Establishing the resulting locally weighted stability and error estimates is beyond the scope of this work.

\begin{figure}
\centering
\begin{subfigure}[t]{.32\linewidth}\centering
\includegraphics[width=0.9\linewidth,trim={295 107 266 88},clip]{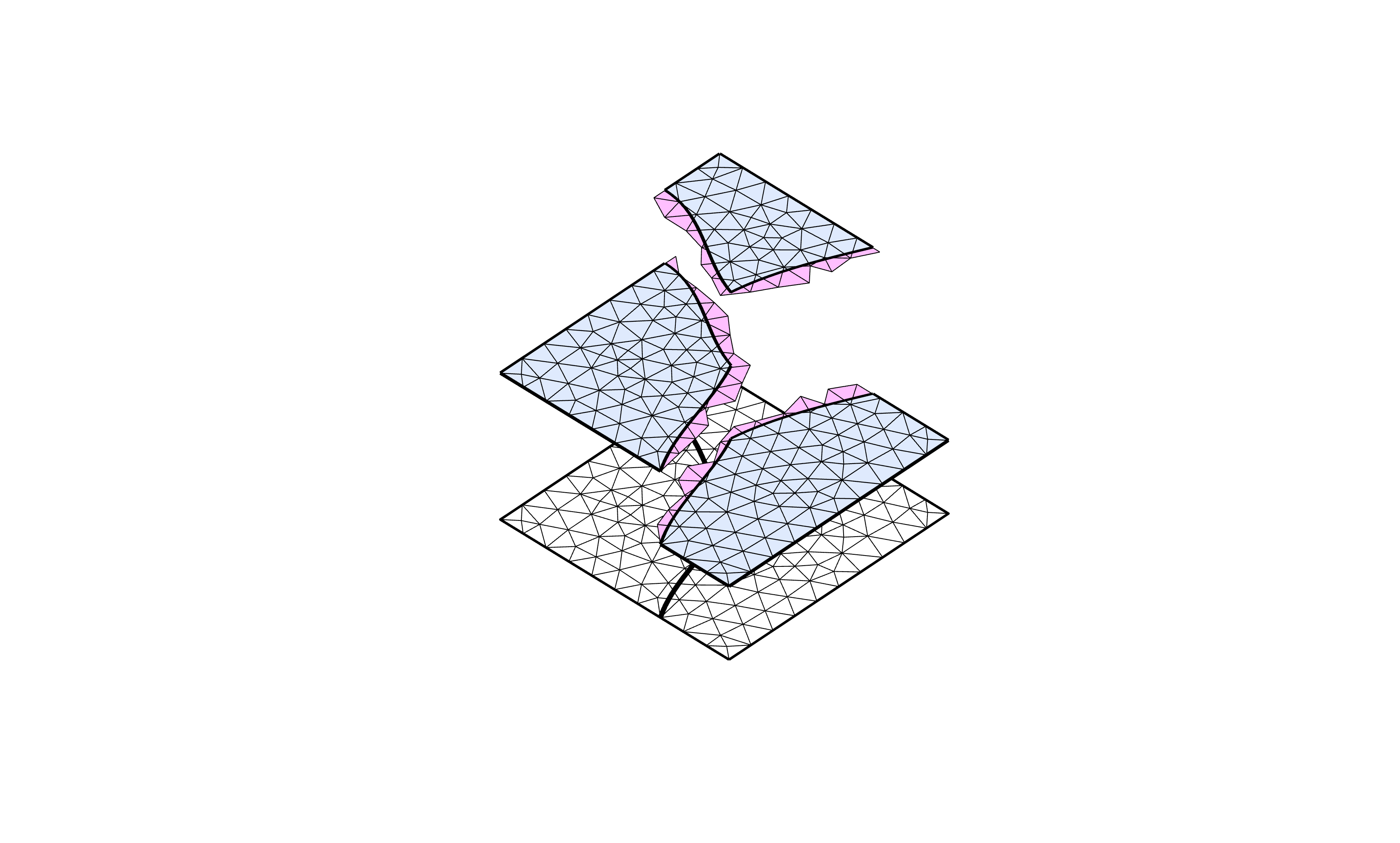}
\caption{$d=2$} \label{fig:meshes-codim0}
\end{subfigure}
\begin{subfigure}[t]{.32\linewidth}\centering
\includegraphics[width=0.9\linewidth,trim={295 107 266 88},clip]{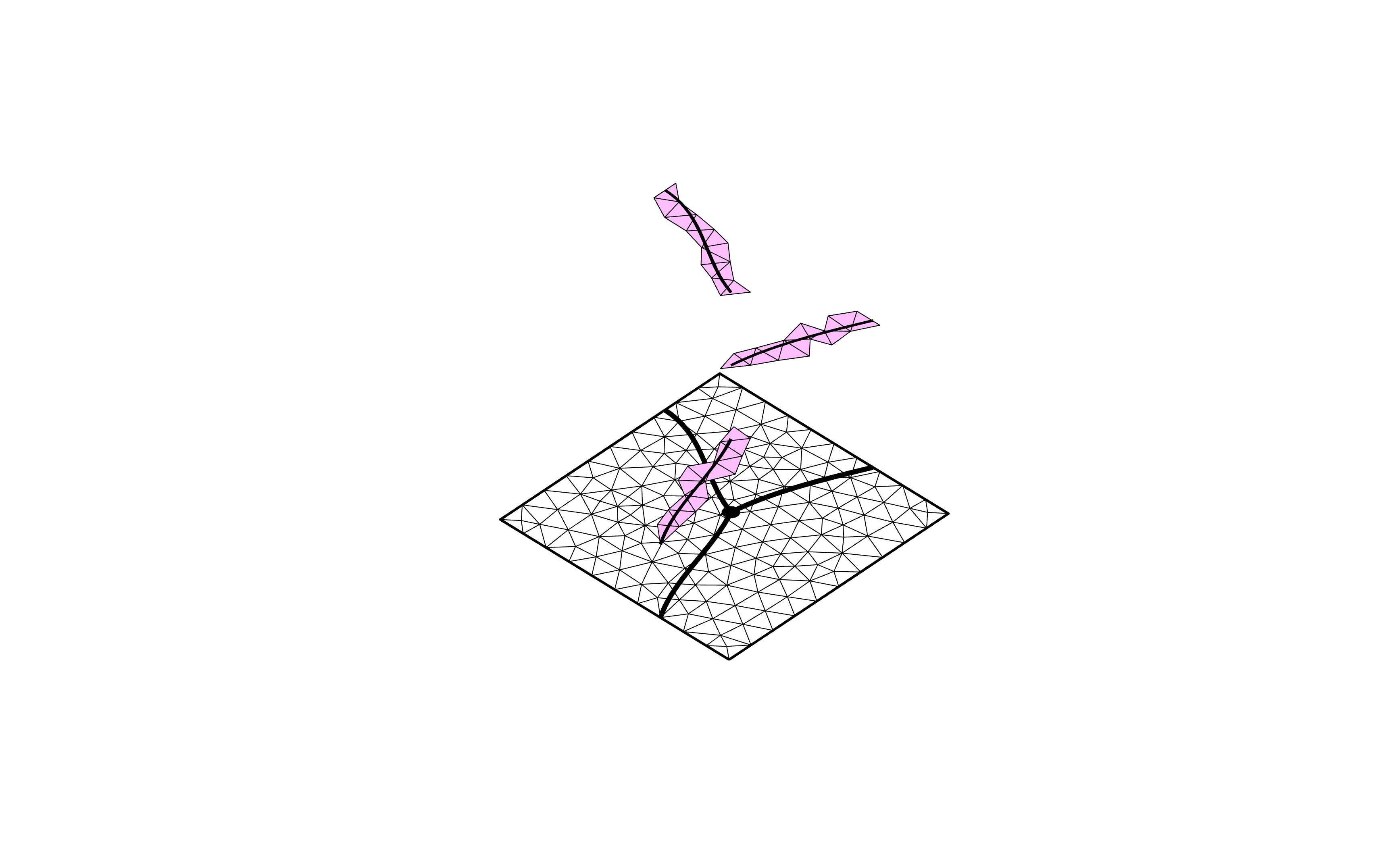}
\caption{$d=1$}\label{fig:meshes-codim1}
\end{subfigure}
\begin{subfigure}[t]{.32\linewidth}\centering
\includegraphics[width=0.9\linewidth,trim={295 107 266 88},clip]{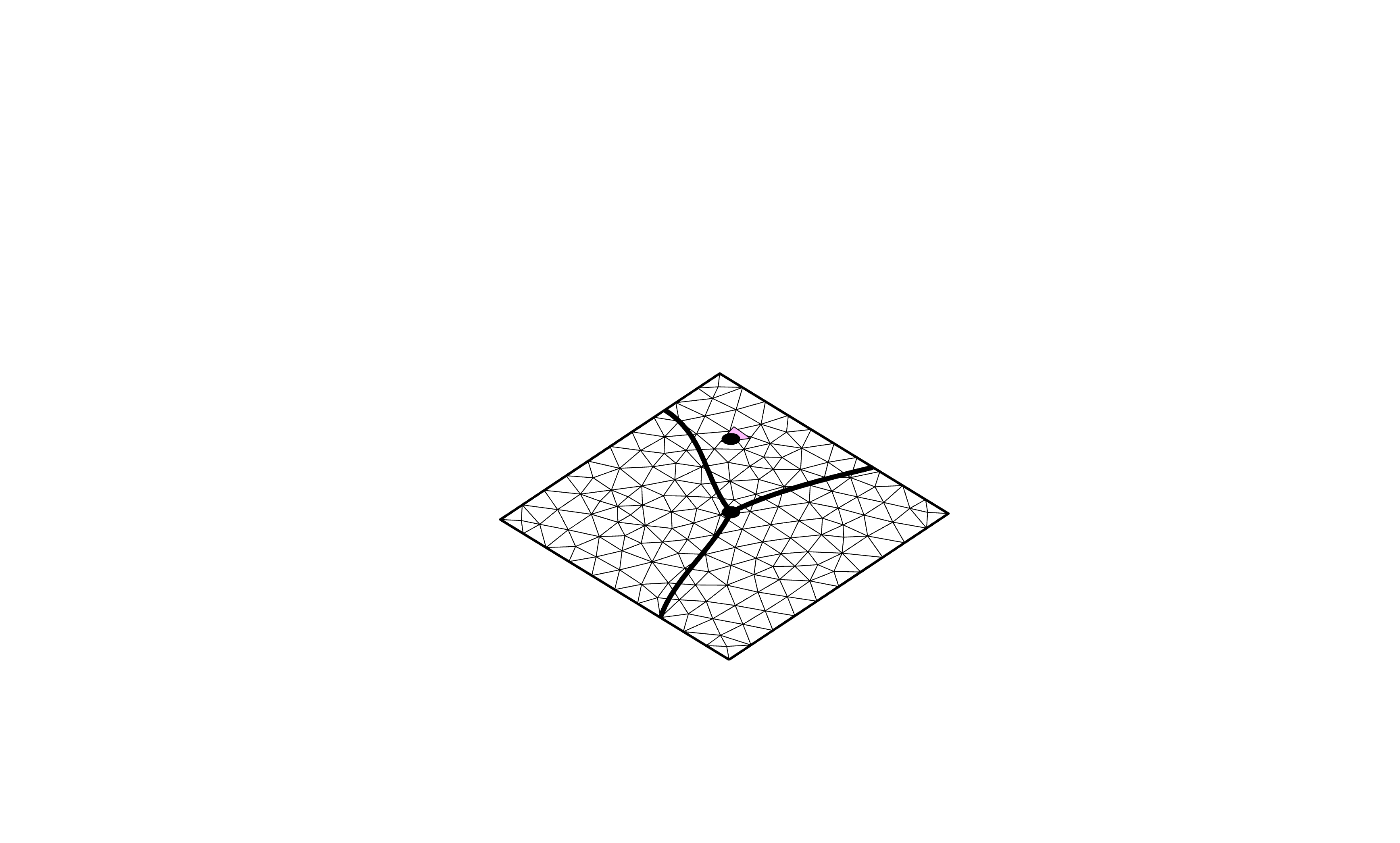}
\caption{$d=0$}\label{fig:meshes-codim2}
\end{subfigure}
\caption{Example of active meshes for a mixed-dimensional geometry in 2D consisting of three bulk domains ($d=2$), three cracks ($d=1$), and one bifurcation point ($d=0$). The colored parts are the active meshes $\{\mcT_{h,d,i}\}$.}
\label{fig:meshes}
\end{figure}

\subsection{Finite Element Method} 
We consider a finite element method based on the weak formulation \eqref{eq:weak-problem}, which naturally incorporates the coupling between the different manifolds. Using a conforming finite element space requires stabilization of the convection term. In addition, since the manifolds cut through the background mesh, extra stabilization is needed to control the variation of the discrete solution in directions orthogonal to the manifolds $\Omega_{d,i}$.

For simplicity we consider piecewise linear elements and employ the standard Galerkin least-squares (GLS) method together with the full gradient stabilization for cut elements developed in \cite{MR3548562}. The full gradient stabilization controls the variation of the discrete solution in directions orthogonal to the manifolds and also improves the conditioning of the resulting linear system. Since the stabilization is not consistent, it is scaled so that its consistency contribution is compatible with the energy-error bounds derived below. For linear elements, its artificial tangential diffusion coefficient is of order $h^3$, which gives a consistency contribution of order $h^{3/2}$ in the energy norm.

For higher order elements one may instead use a weaker full-gradient stabilization or a consistent alternative, such as the normal stabilization proposed in \cite{BuHaLaMa16,GraLehReu16} or the combined normal-face stabilization of \cite{LaZa17}.

\paragraph{Galerkin Least Squares (GLS).}
The GLS method reads: find $u_h\in V_h$ such that 
\begin{equation}\label{eq:fem}
a_h(u_h,v) = l_h(v)\qquad \forall v \in V_h
\end{equation}
where 
\begin{align}
    a_h(v,w) & = a(v,w) + (\tau_1 h ( L v - \llbracket  B_{\alpha}^{\beta_-}   [ v ]\rrbracket ), L w - \llbracket B_{\alpha}^{\beta_-} [ w ]\rrbracket )_{\mcO} + s_h(v,w)
    \\
    l_h(v) & = l(v) + (\tau_1 h f, L v - \llbracket B_{\alpha}^{\beta_-} [ v ]\rrbracket)_{\mcO}
\end{align}
In the above the linear forms $a(\cdot, \cdot)$ and $l(\cdot)$ are defined in \eqref{eq:continuousbilinear}--\eqref{eq:continuouslinear} and $\tau_1>0$ is a parameter defined by
\begin{align}\label{eq:weighting_paramtr}
\tau_1 &=
\begin{cases}
c_\tau\beta_{\infty}^{-1}
& \text{if } \beta_{\infty}h\ge \epsilon
\quad\text{(convection-dominated regime)}
\\
c_\tau h\epsilon^{-1}
& \text{if } \beta_{\infty}h\le \epsilon
\quad\text{(diffusion-dominated regime)}
\end{cases}
\\
&=c_\tau\min\{\beta_{\infty}^{-1},h\epsilon^{-1}\},
\qquad \epsilon>0,\quad\beta_\infty>0
\end{align}
where $c_\tau>0$ is a user-defined parameter. At the degenerate
endpoints, the piecewise definition is used directly: if $\epsilon=0$
and $\beta_\infty>0$, then $\tau_1=c_\tau\beta_\infty^{-1}$, while if
$\beta_\infty=0$ and $\epsilon>0$, then
$\tau_1=c_\tau h\epsilon^{-1}$. The case
$\epsilon=\beta_\infty=0$ is not considered here. For all choices above, we have
\begin{align}\label{eq:weight-diff-conv-bnd}
    \tau_1\epsilon\le c_\tau h \qquad \text{ and } \qquad \tau_1\beta_{\infty}\le c_\tau
\end{align}
Furthermore, the operator $L_{d,i}$ is defined by
\begin{align}
L_{d,i} v &= (D_\beta v + \gamma v)|_{d,i}
\\
&=\beta_{d,i}\cdot \nabla_{d,i} v_{d,i}  
+ ( \nabla_{d,i} \cdot \beta_{d,i}  + \kappa_{d,i}) v_{d,i}
- \llbracket \nu_{d+1} \cdot \beta_{d+1} v_{d+1}\rrbracket_{d,i}
\end{align}
and $s_h$ denotes the stabilization form
\begin{equation}
s_{h,d,i}(v,w) = \tau_2 h^{3 - (n-d)} (\nabla_{\IR^n} v, \nabla_{\IR^n} w )_{\mcT_{h,d,i}}
\end{equation}
where $\tau_2>0$ is a stabilization parameter and $\nabla_{\IR^n}$ denotes the gradient in $\IR^n$. 
The global stabilization form is
\begin{equation}
s_h(v,w)
=
\sum_{d=0}^n\sum_{i=1}^{n_d}
s_{h,d,i}(v_{d,i},w_{d,i})
\end{equation}
Note that $n-d$ is the codimension of $\Omega_{d,i}$; the factor $h^{-(n-d)}$ compensates for the fact that the integral is taken over the $n$-dimensional set $\mcT_{h,d,i}$. The additional factor $h^3$ ensures that the stabilization does not reduce the order of convergence.
After rearranging the terms, the method can be written in componentwise form as 
\begin{align}
 a_h(v,w) &= \sum_{d=0}^n \sum_{i=1}^{n_d} \Big( a_{h,d,i}(v,w) + s_{h,d,i} (v_{d,i},w_{d,i})  \Big)
\\&\qquad\notag
+ ( B_{\alpha}^{\beta_-} [v],[w])_{\partial \mcO_I} +  ( B_{\alpha}^{\beta_-} v,w)_{\partial \mcO_B}
\\
l_h(v) &= \sum_{d=0}^n  \sum_{i=1}^{n_d} l_{h,d,i}(v) + ( B_{\alpha}^{\beta_-} g,v)_{\partial \mcO_B}
\end{align}
where $a_{h,d,i}:V_h\times V_h\to\IR$ and
$l_{h,d,i}:V_h\to\IR$ are, respectively, the local bilinear and linear
contributions defined by
\begin{align}
a_{h,d,i}(v,w) 
&=
(\alpha_{d,i} \nabla_{d,i} v_{d,i}, \nabla_{d,i} w_{d,i})_{\Omega_{d,i}}
+ (L_{d,i} v, w_{d,i})_{\Omega_{d,i}}  \label{eq:method-a}
\\
&\qquad + (\tau_1 h ( L_{d,i} v - \llbracket  B_{\alpha}^{\beta_-}   [ v ]\rrbracket_{d,i} ), L_{d,i} w - \llbracket B_{\alpha}^{\beta_-} [ w ]\rrbracket_{d,i} )_{\Omega_{d,i}} \nonumber
\\
l_{h,d,i}(v) &= (f_{d,i}, v_{d,i})_{\Omega_{d,i}} 
+ (\tau_1 h f_{d,i}, L_{d,i} v - \llbracket B_{\alpha}^{\beta_-} [ v ]\rrbracket_{d,i})_{\Omega_{d,i}} \label{eq:method-b}
\end{align}

\begin{rem}[Pure Convection]
In the pure-convection case ($\alpha=0$), our method differs slightly
from \cite[Eqns.~(3.6)--(3.7)]{BuHaLaLa19}. In that work, the equation
$L u=f$ is stabilized directly, while the inflow coupling is included
separately in the weak form. Here, the GLS residual is inherited from
the convection--diffusion formulation, where the Robin condition is
used to eliminate the conormal diffusive flux. Consequently, the
pure-convection limit contains the additional term
\begin{equation*}
-\llbracket |(\beta_\nu)_-|[v]\rrbracket_d
\end{equation*}
This term is consistent, since the inflow condition implies
$|(\beta_\nu)_-|[u]=0$ for the exact solution, but it results in a
slightly different discrete stabilization.
\end{rem}

\begin{rem}[Classical Interface Problems]
The component-space construction also provides a natural starting point for
classical interface problems on same-dimensional subdomains, such as diffusion
with piecewise constant coefficients. In this setting, the common
lower-dimensional interface carries no separate transport equation;
continuity and flux balance may instead be imposed by replacing the present
Robin coupling, for example using a hybridized CutFEM formulation
\cite{BuElHaLaLa19}. Such a variant is not covered by the present
analysis.
\end{rem}

\section{Error Estimates}\label{sec:errorbounds} 

We prove a basic error estimate in the natural energy norm associated with the GLS method. We assume that the geometry is represented exactly and that all integrals are computed exactly. Under these assumptions, the proof follows standard arguments combined with interpolation error estimates for manifolds of arbitrary codimension. Geometric error estimates can be derived using a generalization of the approach developed in \cite{BuHaLaMa16}.

\subsection{Continuity and Coercivity}
Let 
\begin{equation}\label{eq:extended-space}
V^e = \{v^e = Ev: v \in V_{\alpha,\beta} \}, \qquad W = V^e + V_h
\end{equation}
where $E$ is the extension operator introduced in Section~\ref{sec:defs}. Define the energy norm associated with the bilinear form $a_h(\cdot,\cdot)$ by
\begin{equation}\label{eq:energy-norm}
\tn v \tn^2_h = \| v \|_{\alpha,\beta}^2 + \tau_1 h \| L v \|^2_\mcO + \| v \|^2_{s_h}, 
\qquad v \in W
\end{equation}
and the weighted energy norm by 
\begin{equation}\label{eq:energy-norm-cont}
\tn v \tn^2_{h,*} = \tau_1^{-1}h^{-1} \| v \|^2_\mcO+\tn v \tn^2_h, 
\qquad v \in W
\end{equation}
which we will need in the statement of continuity of the discrete bilinear form $a_h(\cdot,\cdot)$ and error analysis.

\begin{lem}[Continuity and Coercivity]
The form $a_h(\cdot,\cdot)$ is continuous
\begin{equation}\label{eq:Ah-continuity}
|a_h(v, w)| \lesssim \tn v \tn_{h} \tn w \tn_{h,*}, \qquad v,w \in W
\end{equation}
and, for sufficiently small $h$, it is coercive provided \eqref{eq:assumption-coeff} holds,
\begin{equation}\label{eq:Ah-coercivity}
\tn v \tn_h^2 \lesssim a_h(v,v)\qquad v \in W
\end{equation}
\end{lem}
\begin{proof}
\proofparagraph{Continuity.}  We utilize \eqref{eq:coefficients-behavior}--\eqref{eq:beta-interfacebnd} and \eqref{eq:weight-diff-conv-bnd} to see that 
\begin{align}
   (\tau_1 h)^{1/2} \|\llbracket B_{\alpha}^{\beta_-} [v] \rrbracket \|_\mcO & = (\tau_1 h)^{1/2} \|\llbracket ( \alpha_{\nu\nu} + |(\beta_\nu)_-| ) [v] \rrbracket \|_\mcO 
   \\
   &\lesssim  (\tau_1 h)^{1/2} ( \epsilon \|[v]\|_{\partial \mcO_I}  + \beta_{\infty}^{1/2}\| |\beta_\nu|^{1/2} [v] \|_{\partial \mcO_I} )
   \label{eq:alpha-beta-interface-bnd1}
   \\
   &\lesssim h^{1/2} ( h^{1/2} \epsilon^{1/2} \|[v]\|_{\partial \mcO_I}  + \| |\beta_\nu|^{1/2} [v] \|_{\partial \mcO_I} )
   \\
   &\lesssim h^{1/2} \tn v \tn_h   
   \label{eq:alpha-beta-interface-bnd2}
\end{align}
Apply the Cauchy--Schwarz inequality in all terms of $a_h(\cdot,\cdot)$ to obtain
\begin{align}
    |a_h(v, w)| &\le  \| \alpha \nabla v\|_\mcO \| \nabla w\|_\mcO + (\tau_1 h)^{1/2} \| Lv\|_\mcO (\tau_1 h)^{-1/2} \| w\|_\mcO
    \\ \nonumber
    &\quad + \| ( B_{\alpha}^{\beta_-} )^{1/2} [v]  \|_{\partial \mcO_I} \| ( B_{\alpha}^{\beta_-} )^{1/2} [w]  \|_{\partial \mcO_I} 
    \\ \nonumber
    &\quad + \| ( B_{\alpha}^{\beta_-} )^{1/2} v  \|_{\partial \mcO_B} \| ( B_{\alpha}^{\beta_-} )^{1/2} w  \|_{\partial \mcO_B}
    \\ \nonumber
    &\quad + \tau_1 h \| Lv\|_\mcO \| Lw\|_\mcO + \tau_1 h \| Lv\|_\mcO \|\llbracket  B_{\alpha}^{\beta_-}  [w] \rrbracket \|_\mcO  
    \\ \nonumber
    &\quad + \tau_1 h \|\llbracket  B_{\alpha}^{\beta_-}  [v] \rrbracket \|_\mcO \| Lw\|_\mcO 
    \\ \nonumber
    &\quad + \tau_1 h \|\llbracket  B_{\alpha}^{\beta_-}  [v] \rrbracket \|_\mcO  \|\llbracket  B_{\alpha}^{\beta_-}  [w] \rrbracket \|_\mcO + \|v\|_{s_h}\|w\|_{s_h}
\end{align}
which along with \eqref{eq:alpha-beta-interface-bnd2} completes the proof of \eqref{eq:Ah-continuity}.

\proofparagraph{Coercivity.} Proceed as in \eqref{eq:alpha-beta-interface-bnd1} to show that 
\begin{align}
  - \tau_1 h \|\llbracket ( B_{\alpha}^{\beta_-} ) [v] \rrbracket \|_\mcO^2 & = - \tau_1 h \|\llbracket ( \alpha_{\nu\nu} + |(\beta_\nu)_-| ) [v] \rrbracket \|_\mcO^2 
  \\
  & \ge - C_J \tau_1 h \epsilon^2 \|[v]\|_{\partial \mcO_I}^2 - C_J \tau_1 h\beta_\infty \||\beta_\nu|^{1/2}[v]\|_{\partial \mcO_I}^2
  \\
  & \ge - C_J c_\tau h^2 \epsilon \|[v]\|_{\partial \mcO_I}^2 - C_J c_\tau h\||\beta_\nu|^{1/2}[v]\|_{\partial \mcO_I}^2 \label{eq:alpha-beta-interface-bnd-coercivity}
\end{align}
where $C_J>0$ depends only on the fixed geometry and coefficient-comparison constants.

The coercivity \eqref{eq:Ah-coercivity} follows by observing that 
\begin{align}\label{eq:discretebilinear}
a_h(v,v) & =  a(v,v) + (\tau_1 h L v, L v)_{\mcO} + \tau_1 h \|\llbracket B_{\alpha}^{\beta_-} [v] \rrbracket \|_\mcO^2 
\\
& \qquad - 2 \tau_1 h (Lv, \llbracket B_{\alpha}^{\beta_-} [v] \rrbracket)_{\mcO} + s_h(v,v) \nonumber
\end{align} 
Set $c_a=\frac12\min\{1,c_0\}$. Lemma~\ref{lem:coercivity}, Young's inequality with $0<\delta<1$, and inequality \eqref{eq:alpha-beta-interface-bnd-coercivity} then give
\begin{align}
a_h(v,v)
&\ge c_a\| v \|^2_{\alpha,\beta}
+ (1-\delta)\tau_1 h \| L v \|^2_\mcO + \| v \|^2_{s_h}
\\
&\qquad - (\delta^{-1}-1)\tau_1 h
\|\llbracket B_{\alpha}^{\beta_-} [v] \rrbracket \|_\mcO^2 \nonumber
\\
&\ge c_a\| v \|^2_{\alpha,\beta}
+ \delta_1\tau_1 h \| L v \|^2_\mcO + \| v \|^2_{s_h}
\\
&\qquad - C_J\delta_2 c_\tau h^2 \epsilon \|[v]\|_{\partial \mcO_I}^2
- C_J\delta_2 c_\tau h\||\beta_\nu|^{1/2}[v]\|_{\partial \mcO_I}^2 \nonumber
\\
&\ge \frac{c_a}{2}\|v\|_{\alpha,\beta}^2
+ \delta_1\tau_1 h \| L v \|^2_\mcO + \| v \|^2_{s_h}
\\
&\gtrsim \tn v \tn^2_h
\end{align}
where $\delta_1=1-\delta>0$ and $\delta_2=\delta^{-1}-1>0$. For sufficiently small $h$, we have
$C_J\delta_2c_\tau h^2\le c_a/2$ and
$C_J\delta_2c_\tau h\le c_a/2$, which proves the result.
\end{proof}

\subsection{Interpolation Error Estimates}
We define an interpolation operator $\pi_h:H^s(\mcO)\rightarrow V_h$ componentwise by the Cl\'ement interpolant on each active mesh. It satisfies the estimate
\begin{equation}\label{eq:interpol-energy}
\tn u - \pi_h u \tn^2_{h,*}\lesssim \Lambda_I^2 h^{2} \| u \|^2_{H^{2}(\mcO)}
\end{equation}
where 
\begin{equation}\label{eq:energy-norm_paramtr}
\Lambda_I^2 :=  \max\{\epsilon,(\beta_\infty + 1)h\} + \max\{\beta_\infty^{-1}\gamma_\infty^2h^3,  \beta_\infty h\}  + h  
\end{equation} 
More precisely, we define $\pi_h$ by
\begin{equation}
\pi_h u = \pi_{h,Cl} u^e
\end{equation} 
where $\pi_{h,Cl}: L^2(\mcT_{h,d,i}) \rightarrow V_{h,d,i}$ is the usual Clement interpolator.
Here $u^e$ is the prescribed normal extension
\eqref{eq:extension-operator}. In particular, this construction does
not require taking the trace of an arbitrary ambient $H^s$ function
on $\Omega_{d,i}$.
We refer to \cite{BuHaLaMa16} for further details including a proof of the basic interpolation 
estimate 
\begin{equation}\label{eq:interpol-energy_basic}
\| u - \pi_h u \|^2_{\Omega_{d,i}} + h^2 \| \nabla_d ( u - \pi_h u )\|^2_{\Omega_{d,i}}
\lesssim h^4 \| u \|^2_{H^2(\Omega_{d,i})}
\end{equation}
which is then used to derive \eqref{eq:interpol-energy}.

In addition, for lower regularity functions we have the approximation property (cf. \cite[Subsection 3.6]{BuHaLa24}): for $u\in H^{s}(\mcO)$
\begin{equation}\label{eq:interpol-energy_basic_lowregulart}
\| u - \pi_h u \|^2_{H^{m}(\mcO)}\lesssim h^{2(s-m)}\| u \|^2_{H^{s}(\mcO)} \qquad 0\le m\le 1\le s\le 2
\end{equation}
which is employed to establish 
\begin{equation}\label{eq:interpol-energy_lowregulart}
\tn u - \pi_h u \tn^2_{h,*}\lesssim \Lambda_I^2 h^{2(s-1)} \| u \|^2_{H^{s}(\mcO)} \qquad 1\le s\le 2
\end{equation}

\subsection{Error Estimates}
\begin{thm}[Energy Error Bound]\label{thm:energy-error}
Assume that the global ellipticity condition
\eqref{eq:elliptic-positivity} and the coefficient conditions
\eqref{eq:coefficients-behavior} hold with $\epsilon>0$, and that the
coercivity condition \eqref{eq:assumption-coeff} holds. Assume also that
$\alpha$ is componentwise $W^{1,\infty}$ and satisfies
\begin{equation}
\max_{d,i}\|\nabla_{d,i}\alpha_{d,i}\|_{L^\infty(\Omega_{d,i})}
\lesssim\epsilon
\end{equation}
If the continuous solution $u$ to \eqref{eq:weak-problem} belongs to
$H^2(\mcO)$ and $u_h$ is the finite element approximation defined by
\eqref{eq:fem}, with $\tau_1$ chosen as in
\eqref{eq:weighting_paramtr}, then for sufficiently small $h$,
\begin{equation}\label{eq:energy-error-bound}
\tn u - u_h \tn_h \lesssim \Lambda_e h \|u\|_{H^2(\mcO)}
\end{equation}
where $\Lambda_e = \Lambda_I$ and $\Lambda_I$ is defined in \eqref{eq:energy-norm_paramtr}.
\end{thm}
\begin{proof}
By virtue of coercivity \eqref{eq:Ah-coercivity}, continuity \eqref{eq:Ah-continuity}, and discrete scheme \eqref{eq:fem} we obtain 
\begin{align}
\tn u - u_h \tn_h^2 &\lesssim a_h(u - u_h, u - u_h) 
\\
&= a_h(u - u_h, u - \pi_h u) + a_h(u - u_h, \underbrace{\pi_hu - u_h}_{= \rho_h}) 
\\
&\lesssim \tn u - u_h \tn_h \tn u - \pi_h u\tn_{h,*} + a_h(u, \rho_h) - l_h(\rho_h) 
\\
&\lesssim \delta \tn u - u_h \tn^2_h + \delta^{-1} \tn u - \pi_h u\tn^2_{h,*} 
\\
& \qquad+ \underbrace{(\tau_1h (Lu - \llbracket B_{\alpha}^{\beta_-} [ u ]\rrbracket - f), L\rho_h 
 - \llbracket B_{\alpha}^{\beta_-} [ \rho_h ]\rrbracket)_\mcO}_{I}  + \underbrace{s_h(u^e, \rho_h)}_{II} \nonumber
\end{align}
for $\delta>0$.

\proofparagraph{Term $I$.}  
Using \eqref{eq:problem-a}, the definition of $A_\alpha$ in \eqref{eq:elliptic-operator}, the Cauchy--Schwarz inequality, Young's inequality, and the energy norm \eqref{eq:energy-norm}, we obtain
\begin{align}
I & = (\tau_1h \nabla\cdot(\alpha\nabla u), L\rho_h - \llbracket B_{\alpha}^{\beta_-} [ \rho_h ]\rrbracket )_\mcO
\\
&\le (\tau_1h)^{1/2} \epsilon \| u\|_{H^2(\mcO)} \Big( (\tau_1h)^{1/2} \|L\rho_h\|_\mcO + (\tau_1 h)^{1/2} \| \llbracket B_{\alpha}^{\beta_-} [ \rho_h ]\rrbracket \|_\mcO   \Big)
\\
&\lesssim \epsilon^{1/2} h \| u \|_{H^{2}(\mcO)} \Big( (1 + h^{1/2})\tn \rho_h \tn_h     \Big)
\\ 
&\le 2\epsilon^{1/2}h  \| u \|_{H^{2}(\mcO)}\underbrace{\tn \pi_h u - u \tn_h}_{\le \tn u - \pi_h u \tn_{h,*}} + \big( 2\epsilon^{1/2}h  \| u \|_{H^{2}(\mcO)} \big)\tn  u - u_h \tn_h
\\ 
&\le ( 1 + \delta^{-1})4 \epsilon h^2 \| u \|^2_{H^{2}(\mcO)}+ \tn u - \pi_h u \tn_{h,*}^2  + \delta\tn u - u_h \tn^2_h
\end{align}
Here the second line uses
$\|\nabla\cdot(\alpha\nabla u)\|_\mcO
\lesssim (\|\alpha\|_{L^\infty(\mcO)}
+\|\nabla\alpha\|_{L^\infty(\mcO)})\|u\|_{H^2(\mcO)}
\lesssim\epsilon\|u\|_{H^2(\mcO)}$.
We also used the bounds in \eqref{eq:weight-diff-conv-bnd} and
\eqref{eq:alpha-beta-interface-bnd2}.

\proofparagraph{Term $II$.} We have
\begin{align}
II &= s_h (u, \pi_h u - u ) +  s_h (u, u - u_h ) 
\\
&\leq \| u \|_{s_h} \| \pi_h u - u \|_{s_h} + \|u\|_{s_h} \| u - u_h \|_{s_h}
\\
&\leq  \| u \|^2_{s_h} 
+ \underbrace{\| \pi_h u - u \|^2_{s_h}}_{\leq \tn u - \pi_h u \tn^2_{h,*}} 
+ \delta^{-1} \|u\|^2_{s_h} 
+ \delta\| u - u_h \|^2_{s_h}
\end{align}
Using kickback with $\delta>0$ small enough and the interpolation error bound \eqref{eq:interpol-energy}, we arrive at
\begin{align}
\tn u - u_h \tn_h^2 &\lesssim  \epsilon h^2\| u \|^2_{H^{2}(\mcO)} + \tn u - \pi_h u\tn^2_{h,*}  + \| u \|^2_{s_h} \\
&\lesssim \Lambda_e^2 h^2 \| u \|^2_{H^2(\mcO)} + h^3 \| u \|^2_{H^1(\mcO)} \lesssim \Lambda_e^2 h^2 \| u \|^2_{H^2(\mcO)}
\end{align}
where we used $\epsilon \le \Lambda_e^2$ in the last displayed inequality. To estimate the third term, note that the union of the elements in $\mcT_{h,d,i}$ is contained in an $O(h)$ tubular neighborhood of $\Omega_{d,i}$. The tubular-neighborhood stability estimate for normal extensions
\cite[Lemma~9.1]{BuHaLaMa16}, together with the boundedness of the
intrinsic prolongation, gives, for $d\ge1$,
\begin{equation}
\| \nabla_{\IR^n} u^e \|^2_{\mcT_{h,d,i}}
\lesssim
h^{n-d}\|u\|^2_{H^1(\Omega_{d,i})}
\end{equation}
For $d=0$, the left-hand side vanishes because the extension is constant. Hence
\begin{align}
\|u\|^2_{s_h}
&= \sum_{d=0}^n \sum_{i=1}^{n_d} \tau_2 h^{3 - (n-d)} \| \nabla_{\IR^n} u^e \|^2_{\mcT_{h,d,i}}
\\
&\lesssim \sum_{d=0}^n \sum_{i=1}^{n_d} h^3 \|u\|^2_{H^1(\Omega_{d,i})}
= h^3\|u\|^2_{H^1(\mcO)}
\end{align}
which completes the proof.
\end{proof}

\begin{rem}[Elementwise P\'eclet-Number Analysis]
\label{rem:local-peclet}
The estimates derived above use a quasi-uniform background mesh and
global coefficient bounds. A sharper analysis could instead use
genuinely local P\'eclet numbers, as in
\cite[Theorem~10 and Remark~11]{PiDrEr15}. Such estimates distinguish
regions dominated by diffusion, transition regions, and regions
dominated by convection element by element.

For an active element $T\in\mcT_{h,d,i}$ associated with a
positive-dimensional component, one could define
\begin{equation}
\label{eq:local-peclet}
Pe_{T,d,i}
:=
\frac{h_T\beta_{T,d,i}}{\epsilon_{T,d,i}},
\qquad
\beta_{T,d,i}
:=
\|\beta_{d,i}\|_{L^\infty(T\cap\Omega_{d,i})}
\end{equation}
where $\epsilon_{T,d,i}$ is a local lower bound for the tangential
diffusion tensor, with $Pe_{T,d,i}=\infty$ when the local diffusion
vanishes. The global GLS parameter would then be replaced by
\begin{equation}
\label{eq:local-gls-parameter}
\tau_{T,d,i}
\simeq
\min\{
\beta_{T,d,i}^{-1},
h_T\epsilon_{T,d,i}^{-1}
\}
\end{equation}
and the approximation and consistency estimates would be retained as
sums of locally weighted element contributions.

The local trace, inverse, and extension estimates used in the CutFEM
analysis suggest that such a localization is possible. A complete analysis would, however, require interface-compatible
choices of the local GLS weights $\tau_{T,d,i}$ on incident manifold
components, together with local estimates for the transfer terms
involving $\beta_\nu$ and $\alpha_{\nu\nu}$. Zero-dimensional components would
have to be controlled through their incident higher-dimensional
components. We therefore expect a local P\'eclet analysis to be
applicable, but do not pursue it here.
\end{rem}
\subsection{Error Bounds for Pure Convection Problems}\label{sec:errorbound_pureconvection}
Assume that the coercivity condition \eqref{eq:assumption-coeff} holds
and that the pure-convection problem admits a solution
$u\in H^2(\mcO)$. Let $u_h$ be the finite element solution defined by
\eqref{eq:fem} in the pure convection case
(cf.~\cite{BuHaLaLa19}), i.e., when $\alpha\equiv 0$. In this case, the modification in the
interpolation error bound \eqref{eq:interpol-energy} gives
\begin{equation}\label{eq:interpolbnd-pc}
\tn u - \pi_h u \tn^2_{h,*}\lesssim \Lambda_{I}^{pc} h^{3} \| u \|^2_{H^{2}(\mcO)}
\end{equation}
where
\begin{align}\label{eq:pureconv_lambda1}
    \Lambda_{I}^{pc} = \beta_\infty + \beta_\infty^{-1}\gamma_\infty^2 h^2 + 1
\end{align}
Further, we see that the term $I$ in the proof of Theorem \ref{thm:energy-error} is zero (as $Lu - \llbracket B_{\alpha}^{\beta_-} [ u ]\rrbracket - f = Lu - \llbracket |(\beta_\nu)_-| [ u ]\rrbracket - f = 0$), which leads to
\begin{equation}\label{eq:energyerrorbnd_pureconvection}
\tn u - u_h\tn_{h}\lesssim \Lambda_{e}^{pc}  h^{3/2} \| u \|_{H^{2}(\mcO)}  \qquad \text { where } \Lambda_{e}^{pc} =  (\Lambda_{I}^{pc} + 1)^{1/2}
\end{equation}

\subsection{Error Bounds with Low Regularity Solutions}\label{subsec:lowregularity}
The energy bound in Theorem~\ref{thm:energy-error} requires the continuous
solution $u$ to belong to $H^2(\mcO)$. In view of the lower-regularity
interpolation estimate \eqref{eq:interpol-energy_lowregulart}, we can
extend the energy analysis to $u\in H^s(\mcO)$ with $1\le s<2$.

\begin{thm}[Low-Regularity Energy-Error Bound]\label{thm:low-regularity}
Assume that the global ellipticity condition
\eqref{eq:elliptic-positivity} and the coefficient conditions
\eqref{eq:coefficients-behavior} hold with $\epsilon>0$, and that the
coercivity condition \eqref{eq:assumption-coeff} holds. Let
$u\in H^s(\mcO)$, $1\le s<2$, be the continuous solution to
\eqref{eq:weak-problem}, and let $u_h$ be the finite element solution
defined by \eqref{eq:fem}, with $\tau_1$ defined as in
\eqref{eq:weighting_paramtr}. Then, for sufficiently small $h$,
\begin{equation}
   \tn u - u_h \tn_h
   \lesssim \Lambda_{e}^{lr} h^{\sigma_e}
   \label{eq:lowregulrt_enrgythm}
\end{equation}
where $\sigma_e = \min\{1/2, s -1 \}$ and 
\begin{align} 
    \Lambda_{e}^{lr} & = h^{1/2-\sigma_e}\Lambda_1  + \Lambda_I h^{s-1-\sigma_e}\|u\|_{H^s(\mcO)} + h^{3/2-\sigma_e}\|u\|_{H^1(\mcO)}  \label{eq:low-regularity-bnd1}
    \\
    \Lambda_1 & =  \big(  \beta_\infty^{1/2}\| \nabla u\|_{\mcO} + ( 1 + \tau_1^{1/2}\gamma_\infty )\|u\|_{\alpha, \beta} + \tau_1^{1/2}\|f\|_\mcO  \big) \label{eq:low-regularity-bnd2}
\end{align}
with $\Lambda_I$ defined in \eqref{eq:energy-norm_paramtr}.
\end{thm}
\begin{proof}
We bound the term $I$ in the same way as in the proof of Theorem \ref{thm:energy-error}, but without using the $\|\cdot\|_{H^{2}(\mcO)}$-norm, i.e., without involving the second-order term in $A_\alpha u$, as follows:
\begin{align}
    I & = (\tau_1h(Lu - \llbracket B_{\alpha}^{\beta_-} [ u ]\rrbracket -f), L(\pi_h u - u_h) - \llbracket B_{\alpha}^{\beta_-} [ \pi_h u - u_h ]\rrbracket )_\mcO
    \\
    &\le \Big(  (\tau_1h)^{1/2}\|Lu - \llbracket B_{\alpha}^{\beta_-} [ u ]\rrbracket -f \|_\mcO \Big)
    \\ \nonumber
    & \qquad \Big( (\tau_1h)^{1/2} \|L(\pi_h u - u_h) - \llbracket B_{\alpha}^{\beta_-} [ \pi_h u - u_h ]\rrbracket\|_\mcO \Big)
    \\
    &\lesssim h^{1/2}\Big( \beta_\infty^{1/2}\| \nabla u\|_{\mcO} + ( 1 + \tau_1^{1/2}\gamma_\infty )\|u\|_{\alpha, \beta} + \tau_1^{1/2}\|f\|_\mcO \Big)
    \\
    & \qquad \Big(   (1 + h^{1/2} ) \tn \pi_h u - u_h \tn_h \Big) \nonumber
    \\
    &\le (1+\delta^{-1}) \Lambda_1^2 h + \tn u - \pi_h u \tn_{h,*}^2 + \delta\tn u - u_h \tn^2_h
\end{align}
where we have used the directional derivative bound in \eqref{eq:directionalderivbound} and the inequalities \eqref{eq:alpha-interfacebnd}--\eqref{eq:beta-interfacebnd} and \eqref{eq:weight-diff-conv-bnd}.
The bound for the term $II$ is the same as in the proof of Theorem \ref{thm:energy-error}. Collecting these bounds together with the interpolation estimate \eqref{eq:interpol-energy_lowregulart} and applying kickback for sufficiently small positive $\delta$ finally yields
\begin{equation}\label{eq:low-regularity-energy-proof}
\tn u - u_h \tn^2_h\lesssim \Lambda_1^2 h + \Lambda_I^2 h^{2(s-1)} \| u \|^2_{H^{s}(\mcO)} +  h^{3} \|u\|^2_{H^{1}(\mcO)}
\end{equation}
which completes the proof of \eqref{eq:lowregulrt_enrgythm}.
\end{proof}

\begin{rem}\label{rem:low-regularity}
\proofparagraph{(i).}
If $s\in [1,3/2]$, so that $\sigma_e = \min\{1/2,s-1\} = s-1$, then the energy error bound in \eqref{eq:lowregulrt_enrgythm} is of order $O(h^{s-1})$. This agrees with \cite[Theorem~3.1]{BuHaLa24}, where low-regularity CutFEM approximations of an elliptic problem with mixed boundary conditions are analyzed.

\proofparagraph{(ii).} For the pure diffusion problem, i.e., when
$\beta\equiv\kappa\equiv0$, neither the coefficient scaling
\eqref{eq:coefficients-behavior} nor the coercivity condition
\eqref{eq:assumption-coeff} is satisfied, and
Theorem~\ref{thm:low-regularity} therefore does not apply as stated.
An analogous estimate would require coercivity from an appropriate
mixed-dimensional Poincaré inequality, as discussed in
Remark~\ref{rem:coercivity-poincare}; that extension is not established
here.

\proofparagraph{(iii).} Assume that \eqref{eq:assumption-coeff} holds
and that the pure-convection problem admits a solution
$u\in H^s(\mcO)$, $1\le s\le2$. The modification of the interpolation
error bound \eqref{eq:interpol-energy_lowregulart} then results in
\begin{align}\label{eq:interpolenergy-lowreg-pc}
  \tn u - \pi_h u \tn^2_{h,*}\lesssim \Lambda_I^{pc} h^{2(s-1/2)} \| u \|^2_{H^{s}(\mcO)} \qquad 1\le s\le 2
\end{align}
with $\Lambda_I^{pc}$ defined in \eqref{eq:pureconv_lambda1}. 
Proceeding as in the proof of
Theorem~\ref{thm:low-regularity}, but using
\eqref{eq:interpolenergy-lowreg-pc}, leads to
\begin{equation}
   \tn u - u_h \tn_h
   \lesssim \Lambda_{e}^{pc} h^{(s-1/2)}\|u\|_{H^s(\mcO)}
   \label{eq:energy-lowreg-pc}
\end{equation}

\end{rem}

\section{Numerical Results}
\label{sec:numerics}

This section presents numerical experiments that assess the energy-error rates proved in Section \ref{sec:errorbounds} and illustrate the behavior of the method. Cases~I and III represent, respectively, the globally diffusive and globally pure-convection regimes covered by the analysis (the latter conditionally on the assumed regular solution). Case~II and the low-regularity example have diffusion in the bulk but not in the fractures and therefore lie outside the two global regimes analyzed here; their convergence rates are reported as empirical observations. Case~IV lies outside the coercivity assumption \eqref{eq:assumption-coeff}. All $L^2$-rates in Sections \ref{sec:convergence}--\ref{sec:convergence-lowregularity} are empirical, since no $L^2$-error estimate is established here. Section \ref{sec:illustration} contains additional numerical illustrations.

\subsection{Implementation}

The CutFEM is implemented in MATLAB in two space dimensions and the linear system of equations is solved using a direct solver (MATLAB's \verb+\+ operator). We refer to \cite{MR3682761} for implementational details.

\paragraph{Approximation Spaces.} We generate a background triangular quasi-uniform mesh $\mcT_{h}^0$ of $\Omega^0$ with mesh size $h$ and construct a continuous piecewise linear finite element space on $\mcT_{h}^0$. Since the background mesh is generated independently of the geometry, each manifold component $\Omega_{d,i}$ may intersect the mesh in an arbitrary way. Restricting the basis functions to the active mesh \eqref{eq:mesh-active} of each manifold component, i.e., the elements of $\mcT_h^0$ that intersect the manifold component, gives the finite element space for that component. Note that in this construction the approximation spaces for all components are in $\IR^n$, regardless of the dimension of the manifold components, see Figure~\ref{fig:meshes}. Also, the approximation spaces for the components are not coupled, since all coupling is naturally enforced weakly in our formulation.

\paragraph{Parameter Values.} The least-squares stabilization constant is
chosen as $c_\tau=1$. For $\epsilon>0$ and $\beta_\infty>0$, we take
$\tau_1=c_\tau\min\{\beta_\infty^{-1},h\epsilon^{-1}\}$. When the global
diffusion scale is zero, including Case~II, Case~III, and the low-regularity
example, we instead take $\tau_1=c_\tau\beta_\infty^{-1}$. In the
low-regularity case, i.e., for
$u\in H^s(\mcO)$, we take $c_\tau=25$. We observe no particular
sensitivity with respect to $c_\tau$ for the case $u\in H^2(\mcO)$.
Throughout this section, we fix the full gradient stabilization parameter
at $\tau_2=10^{-3}$.

\paragraph{Reported Norms.}
Throughout this section, the reported energy errors are evaluated using the
actual componentwise diffusion coefficients. Thus, for $e=u-u_h$, the terms
$\epsilon\|\nabla e\|_{\mcO}^2$ and
$\epsilon\|[e]\|_{\partial\mcO}^2$ in
\eqref{eq:continuous-norm} are replaced by
$(\alpha\nabla e,\nabla e)_{\mcO}$ and
$\|\alpha_{\nu\nu}^{1/2}[e]\|_{\partial\mcO}^2$, respectively, while the
remaining terms are unchanged. Under the global
uniform-diffusion assumptions, this coefficient-weighted norm is equivalent
to \eqref{eq:continuous-norm}; in the partially degenerate examples, the
reported rates remain empirical.

\subsection{Convergence}\label{sec:convergence}
All experiments in this subsection are performed on the two-dimensional domain $\Omega=[0,1]^2$. Thus, $\Omega$ may consist of bifurcation points ($d=0$), fractures ($d=1$), and bulk domains ($d=2$). With the exception of Case IV, the examples are constructed so that the exact solutions satisfy condition \eqref{eq:assumption-coeff} for suitable choices of the reaction coefficient and tangential vector field.

The flow regimes and energy-error rates used for theoretical comparison
are summarized in Table~\ref{table:overview}.

\begin{table}
    \scriptsize
    \caption{\emph{Rates used for theoretical comparison.} Flow regimes
    on the manifolds and energy-error rates in the analyzed cases.}
    \label{table:overview}
    \centering
    \begin{tabular}{ c  c  c  c  c }
    \toprule 
    Cases & \multicolumn{3}{c}{Manifolds} & Analysis rate
    \\\cmidrule(lr){2-4}\cmidrule(lr){5-5}
    & $d=2$ &  $d=1$ & $d=0$ & Energy-OC \\
    \midrule
     I &  convection-diffusion  &  convection-diffusion   &  $-$  & $1.5$ \\
    II &  convection-diffusion  &  convection     &  $-$   &  $-$ \\
    III & convection   &   convection   &  coupling terms   & $1.5$ \\
    \bottomrule
    \end{tabular}
    \end{table}

The dash in the energy-bound column for Case~II indicates that its
partially degenerate diffusion configuration is not covered by the
present analysis.
For Case~I, all tabulated meshes satisfy
$\beta_\infty h\ge\epsilon$. In this convection-dominated range,
\eqref{eq:energy-norm_paramtr} gives $\Lambda_I=O(h^{1/2})$, and
Theorem~\ref{thm:energy-error} therefore gives an $O(h^{3/2})$ energy
bound. For fixed positive $\epsilon$, the asymptotic bound becomes
$O(h)$ once $\beta_\infty h\le\epsilon$. No optimality claim is made
for either regime.

\paragraph{Case I: Convection-Diffusion in Bulk and Fracture.} In this case, a single fracture divides the unit square into two parts so that we have one fracture and two bulk domains, see Figure~\ref{pic:cd_bulkcrack}. We consider $\alpha_{2, 1} = \alpha_{2, 2} = \epsilon \bfI_{2\times 2}$ and $\alpha_{1, 1} = \epsilon$,  where $\epsilon = 10^{-5}$ and $\bfI_{2\times 2}$ is the $2\times 2$ identity matrix. The bulk vector fields $\{\beta_{2,i}\}$ are directed into the fracture so that the fracture solution is influenced by the bulk solutions but not vice versa.

The bulk solutions are constructed so as to satisfy the Robin interface conditions on the fracture. For appropriate choices of the source functions $f_{d,i}$, boundary data $g_{d,i}$, and reaction coefficients $\kappa_{d,i}$, the manufactured exponential solutions are given by
\begin{align}\label{numer_exactsolns}
    u_{2,1} = e^{ (x-1/2) + y}, \quad u_{2,2} = e^{ -(x-1/2) + y}, \quad \text{and} \quad  u_{1,1} = 2e^{y}
\end{align}
The exact solution and the numerical solution for $h=0.2$ are shown in Figures~\ref{sol:exact}--\ref{sol:numer}.
\begin{figure}
\centering
\begin{subfigure}[t]{.30\linewidth}\centering
\includegraphics[width=0.9\linewidth]{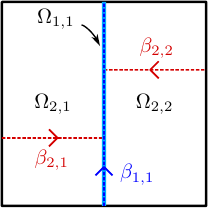}
\caption{Set-up} \label{pic:cd_bulkcrack}
\end{subfigure}
\begin{subfigure}[t]{.30\linewidth}\centering
\includegraphics[width=0.9\linewidth,trim=0 0 55 0,clip]{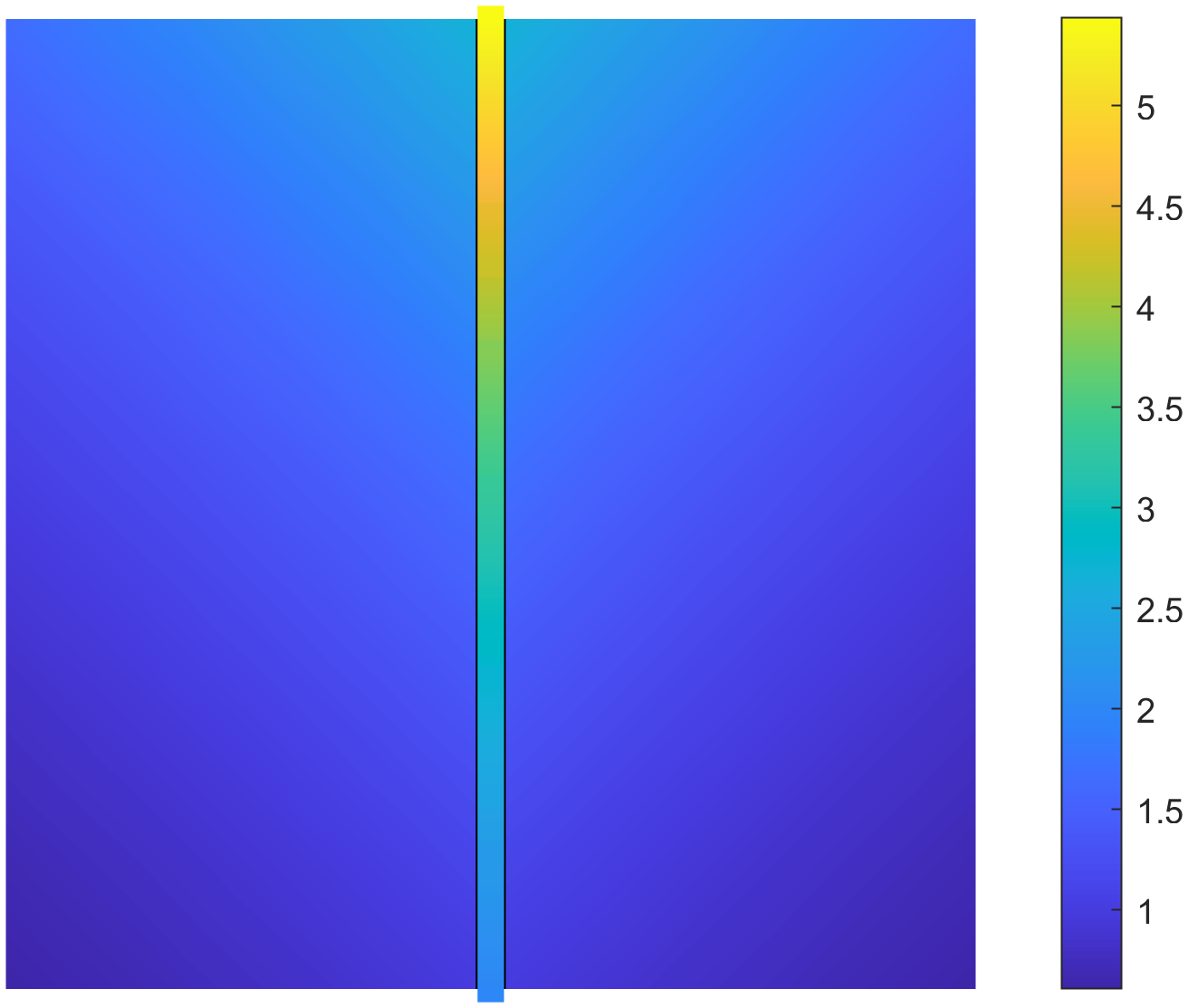}
\caption{Exact solution}
\label{sol:exact}
\end{subfigure}
\begin{subfigure}[t]{.30\linewidth}\centering
\includegraphics[width=0.9\linewidth,trim=0 0 55 0,clip]{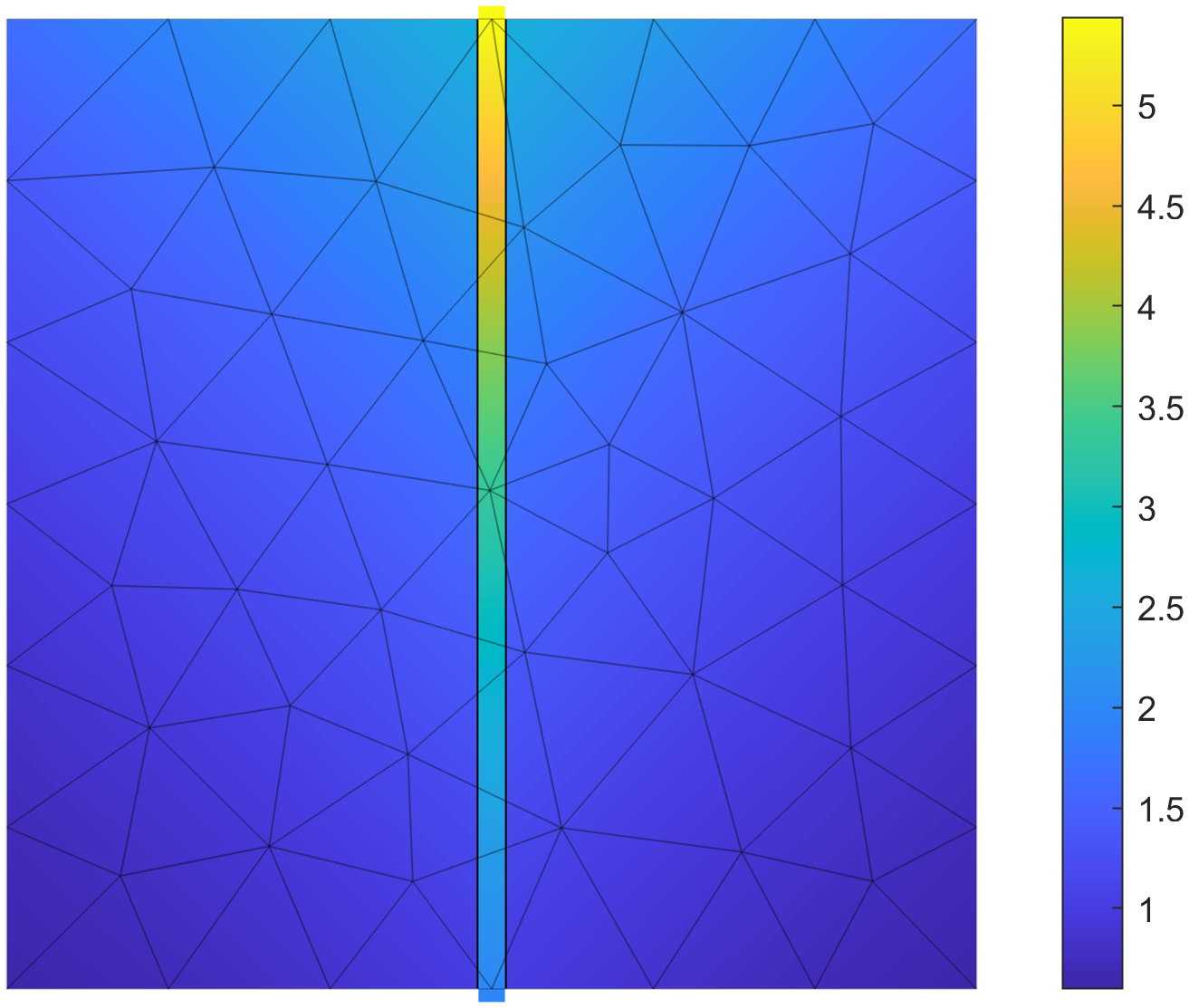}
\caption{Numerical solution} 
\label{sol:numer}
\end{subfigure}
\caption{\emph{Case I: Convection-diffusion in bulk and fracture.} (a) In this set-up: bulk diffusion coefficients $\alpha_{2, 1} = \alpha_{2, 2} = \epsilon \bfI_{2\times 2}$ and fracture diffusion coefficient $\alpha_{1, 1} = \epsilon$,  where $\epsilon = 10^{-5}$; bulk vector fields $\beta_{2,1}=[1,0]$, $\beta_{2,2}=[-1,0]$; fracture vector fields $\beta_{1,1}=[0,1]$. (b) Exact solution. (c) Numerical solution with $h=0.2$.}

\end{figure}
The empirical findings in Table~\ref{table:cd_bulkcrack} demonstrate
approximately $O(h^{3/2})$ and quadratic convergence in the energy and
$L^2$-norms, respectively. The observed energy rate agrees with the
bound for the convection-dominated range used here, while the $L^2$
rate is reported empirically.

\begin{table}
\scriptsize
\caption{\emph{Case I: Convection-diffusion in bulk and fracture.} Experimental errors and convergence rates for the convection-dominated regime.}
\centering
\label{table:cd_bulkcrack}
\begin{tabular}{ c    c c c c } 
\toprule
$h$ &  Errors in energy norm  & Energy-OC & Errors in $L^2$-norm  &  $L^2$-OC  \\
\midrule
1/5  & \na{7.45212e-2}   &  $-$  &  \na{7.24865e-3}  & $-$     \\
1/10 & \na{2.22754e-2}   &  1.74220     &  \na{1.16855e-3}  &  2.63298 \\ 
1/20 & \na{8.13849e-3}    &   1.45262    &  \na{3.03294e-4}  &  1.94594 \\ 
1/40 & \na{2.72321e-3}    &   1.57945    &  \na{7.48028e-5}  &  2.01955 \\
1/80 & \na{8.95915e-4}     &  1.60388     &  \na{1.69896e-5}  &  2.13844  \\ 
\midrule
Analysis rate &  &  1.5    & &   $-$ \\
\bottomrule
\end{tabular}
\end{table}

\paragraph{Case II: Convection-Diffusion in Bulk and Pure Convection in Fracture.}
Here we consider the same domain and convection parameters as in Case~I,
but take
$\alpha_{2,1}=\alpha_{2,2}=\epsilon_{\mathrm{bulk}}\bfI_{2\times2}$,
with $\epsilon_{\mathrm{bulk}}=10^{-5}$, and $\alpha_{1,1}=0$. Thus, we
have convection--diffusion in the bulk domains and pure convection in
the fracture, while the flow in the fracture is still affected by
diffusion in the bulk domains. The manufactured exact
solutions are the same as in \eqref{numer_exactsolns}. Since diffusion is
positive only on some components, this case satisfies neither the global
ellipticity assumptions of the diffusive analysis nor $\alpha\equiv0$.
The findings in Table~\ref{table:cdbulk_ccrack} are therefore entirely
empirical, although they are very similar to those in
Table~\ref{table:cd_bulkcrack}.

\begin{table}
\scriptsize
\caption{\emph{Case II: Convection-diffusion in bulk and pure convection in fracture.} Experimental errors and convergence rates for the convection-dominated regime.}
\centering
\label{table:cdbulk_ccrack}
\begin{tabular}{ c   c c c c } 
\toprule
\rule{0pt}{3ex} 
$h$  & Errors in energy norm  & Energy-OC & Errors in $L^2$-norm  &  $L^2$-OC  \\
\midrule
\multirow{4}{*}{}
1/5    & \na{7.45205e-2}  & $-$ & \na{7.24929e-3}  & $-$ \\
1/10   & \na{2.22748e-2}  & 1.74222    & \na{1.16867e-3}  &  2.63297  \\ 
1/20   & \na{8.13796e-3}  & 1.45268   & \na{3.03284e-4}  &  1.94613  \\ 
1/40   & \na{2.72287e-3}  & 1.57954    & \na{7.48169e-5}  &  2.01924  \\
1/80   & \na{8.95660e-4}  & 1.60411    & \na{1.69936e-5}  &  2.13837  \\ 
\bottomrule
\end{tabular}
\end{table}

\paragraph{Case III: Pure Convection in Bulk and Fractures.} In contrast to the previous cases, for the set-up in Figure~\ref{purely-convection} we consider the purely convective case ($\alpha= 0$) in four bulk domains, four fractures, and one bifurcation point at $(1/2,1/2)$ where the fractures split. We have three fractures with in-flow and the rest with out-flow. With the appropriate choices of the problem data the algebraic solutions are constructed as
\begin{alignat}{2}\label{pureconv_exactsolns}
    & u_{2,1} = x^2 + y^2,\qquad &&u_{2,2} = 1 + (x - 1/2)^2 + (y - 1/2)^2
    \\
    &u_{2,3} = -(x^2 + y^2), \qquad &&u_{2,4} = x^2 - y^2
    \\
    & u_{1,1} = 3/4 + (y-1)^2, \qquad &&u_{1,2} = 1 + (x - 1/2)^2
    \\
    & u_{1,3} = 1 + (y - 1/2)^2, \qquad &&u_{1,4} = 1 - (x - 1/2)^2
    \\&
    u_{0,1} = 1
\end{alignat}
The exact solution and the numerical solution for $h=0.2$ are shown in
Figures~\ref{pc_exactsol}--\ref{pc_numersol}. The numerical results with
$\tau_1=c_\tau\beta_{\infty}^{-1}=1/\sqrt2$ are reported in
Table~\ref{table:purelyconvection}. They complement the numerical
illustrations in \cite[Section~5]{BuHaLaLa19} and are consistent with
the energy estimate in Subsection~\ref{sec:errorbound_pureconvection}.
The $L^2$ rate is reported empirically.

\begin{figure}
\centering
\begin{subfigure}[t]{.30\linewidth}\centering
\includegraphics[width=0.9\linewidth]{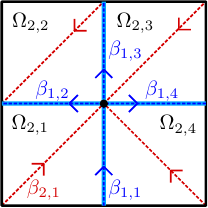}
\caption{Set-up}\label{purely-convection}
\end{subfigure}
\begin{subfigure}[t]{.30\linewidth}\centering
\includegraphics[width=0.9\linewidth,trim=0 0 55 0,clip]{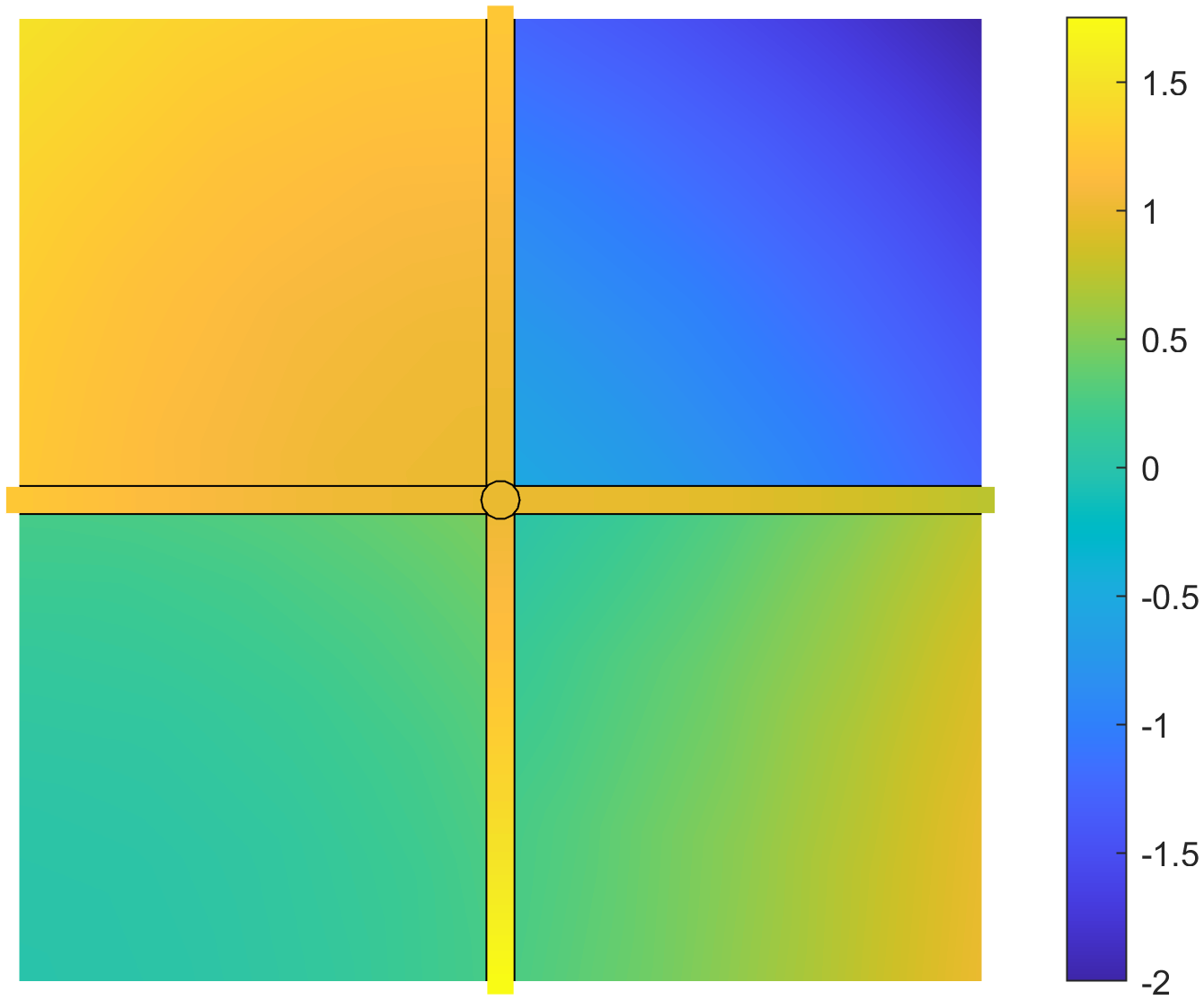}
\caption{Exact solution}\label{pc_exactsol}
\end{subfigure}
\begin{subfigure}[t]{.30\linewidth}\centering
\includegraphics[width=0.9\linewidth,trim=0 0 55 0,clip]{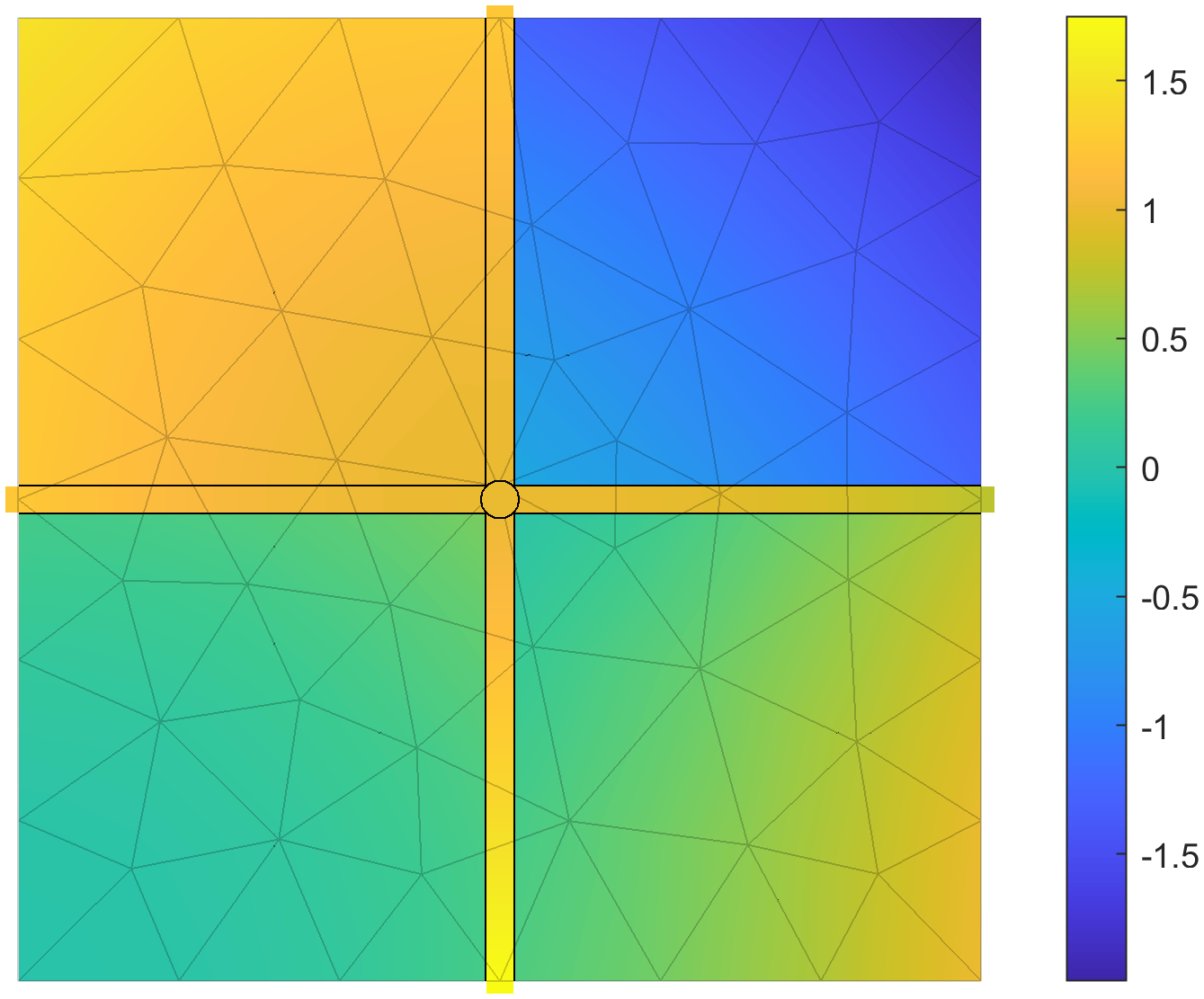}
\caption{Numerical solution}\label{pc_numersol}
\end{subfigure}
\caption{\emph{Case III: Pure convection in bulk and fractures.} (a) In this set-up: bulk vector fields $\beta_{2,i}$ are $[1,1]$, $[-1,-1]$, $[-1,-1]$, $[-1,1]$; fracture vector fields $\beta_{1,i}$ are $[0,1]$, $[-1,0]$, $[0,1]$, $[1,0]$. (b) Exact solution. (c) Numerical solution when $h=0.2$.}
\end{figure}

\begin{table}
\scriptsize
\caption{\emph{Case III: Pure convection in bulk and fractures.} Experimental errors and convergence rates for pure convection.}
\centering
\label{table:purelyconvection}
\begin{tabular}{ c   c c c c } 
\toprule
\rule{0pt}{3ex} 
$h$  & Errors in energy norm  & Energy-OC & Errors in $L^2$-norm  &  $L^2$-OC  \\
\midrule
\multirow{4}{*}{}
1/5    & \na{4.80419e-2}   & $-$   & \na{3.51324e-3}    & $-$  \\
1/10   & \na{1.73241e-2}   &  1.47151     & \na{7.47837e-4}    &  2.23201  \\ 
1/20   & \na{6.15353e-3}   &  1.49329     & \na{2.13119e-4}    &  1.81106 \\ 
1/40   & \na{2.07016e-3}   &  1.57167     & \na{5.01863e-5}    &  2.08630  \\
1/80   & \na{6.97887e-4}   &  1.56868     & \na{1.23101e-5}    &  2.02745 \\ 
\midrule
Analysis rate & & 1.5 & & $-$ \\
\bottomrule
\end{tabular}
\end{table}

\paragraph{Case IV: Example Violating (\ref{eq:assumption-coeff}).}
We also present a test outside assumption \eqref{eq:assumption-coeff}.
For small $\epsilon$, condition \eqref{eq:assumption-coeff} can be
relaxed in standard settings (cf.~\cite{navert1982}). Motivated by this,
we take $\epsilon=10^{-10}$ and
$\kappa_{2,1}=\kappa_{2,2}=\kappa_{1,1}=0$ in the geometry of Case~I.
Thus $2\kappa+\Div\beta=0$ in $\Omega_{2,1}$ and $\Omega_{2,2}$,
whereas $2\kappa+\Div\beta=-2$ in $\Omega_{1,1}$. The empirical results
in Table~\ref{table:zeroreaction} show strong observed rates in both
norms. Since the example lies outside the assumptions of the analysis,
these results are numerical evidence rather than a verification of the
theoretical bounds.

\begin{table}
\scriptsize
\caption{\emph{Case IV: Example violating (\ref{eq:assumption-coeff}).}
Experimental errors and convergence rates for the convection-dominated regime with $\epsilon = 10^{-10}$ and zero reaction terms.}
\centering
\label{table:zeroreaction}
\begin{tabular}{ c   c c c c } 
\toprule
\rule{0pt}{3ex} 
$h$  & Errors in energy norm  & Energy-OC & Errors in $L^2$-norm  &  $L^2$-OC  \\
\midrule
\multirow{4}{*}{}
1/5    & \na{7.50180e-2}   & $-$   &  \na{8.50034e-3}   & $-$ \\
1/10   & \na{2.23190e-2}   &  1.74896   &  \na{1.29867e-3}   &  2.71048   \\ 
1/20   & \na{8.14348e-3}   &  1.45456   &  \na{3.25722e-4}   &  1.99533  \\ 
1/40   & \na{2.72315e-3}   &  1.58037   & \na{7.88389e-5}    &  2.04666  \\
1/80   &  \na{8.95429e-4}  &  1.60463   &  \na{1.78599e-5}   &  2.14219 \\ 
\bottomrule
\end{tabular}
\end{table}

\subsection{Convergence Test with Low Regularity Solution}\label{sec:convergence-lowregularity}
In this subsection, we consider $\Omega = [-1,1]^2$ and the set-up illustrated in Figure~\ref{pic:pc}. Similar to Case III, the geometry consists of four bulk domains, four fractures, and one bifurcation point at the center. Here, however, we consider convection-diffusion in the bulk domains and pure convection in the fractures. We take $\alpha_{2,i} = \epsilon_{\mathrm{bulk}} \bfI_{2\times 2}$ and $\alpha_{1,i} = 0$, where $\epsilon_{\mathrm{bulk}} = 10^{-5}$ and $\bfI_{2\times 2}$ is the $2\times 2$ identity matrix. For simplicity, all fractures are taken with inflow, so that all fracture solutions are affected by the bulk solutions but not vice versa. For suitable choices of the problem data, the solutions in polar coordinates are given by
\begin{align}\label{lowregulrt_exactsolns}
    & u_{2,1} = r^{5/3}\sin(2\theta),\quad u_{2,2} = - r^{5/3}\sin(2\theta), \quad u_{2,3} = u_{2,1}, \quad u_{2,4} =  u_{2,2} \\
    & u_{1,1} =  u_{1,2} =  u_{1,3} = u_{1,4} = - 2r^{2/3} \\
    & u_{0,1} = 0
\end{align}
Due to the Robin-type interface conditions and the singularity at the bifurcation point $(0,0)$, we cannot consider convection-diffusion in the fractures and therefore restrict attention to the purely convective case there. If diffusion were included in the fractures, then the source terms would become $f_{1,i} = 4r^{-4/3}/9$, which do not belong to $L^2(\Omega_{1,i})$. Although there is no diffusion in the fractures, the coupling terms involving the diffusion coefficients of the bulk domains are still present.

This partially degenerate diffusion configuration is outside the two
global coefficient regimes analyzed in Section~\ref{sec:errorbounds}.
The bulk solutions satisfy
$u_{2,i}\in H^{2+2/3-\delta}(\Omega_{2,i})$, whereas the fracture
solutions satisfy
$u_{1,i}\in H^{1+1/6-\delta}(\Omega_{1,i})$ for every $\delta>0$,
cf.~\cite[Example~2.16]{ErnGuermond-I}.
For comparison, a heuristic application of Remark~4.2(iii) to an isolated pure-convection fracture suggests the reference rate \(s-\frac12=\frac23-\delta\), arbitrarily close to \(2/3\).

\begin{figure}
\centering
\begin{subfigure}[t]{.30\linewidth}\centering
\includegraphics[width=0.9\linewidth]{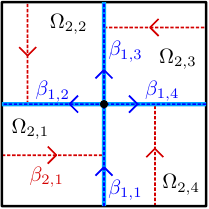}
\caption{Set-up} \label{pic:pc}
\end{subfigure}
\begin{subfigure}[t]{.30\linewidth}\centering
\includegraphics[width=0.9\linewidth,trim=0 0 55 0,clip]{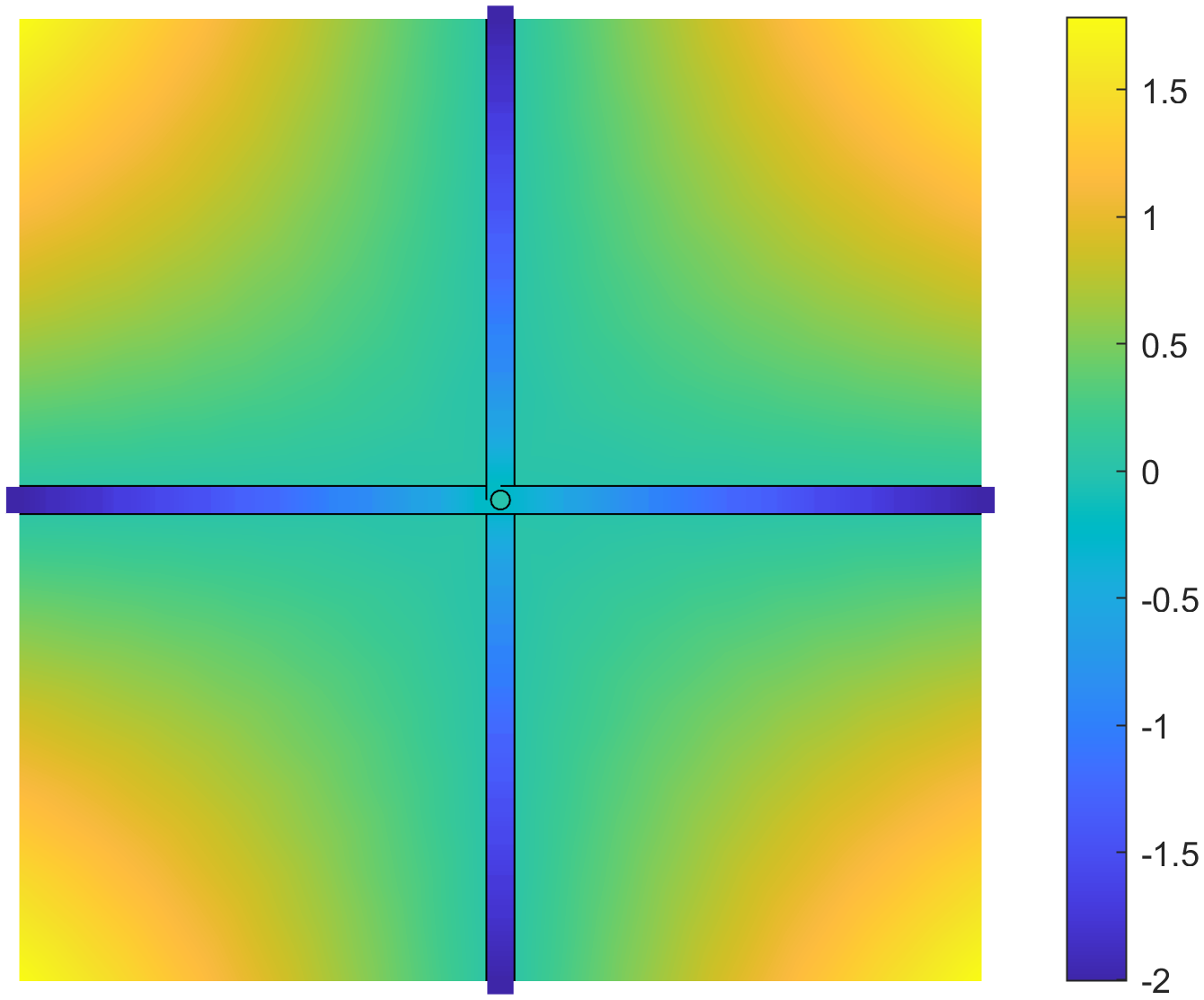}
\caption{Exact solution}\label{pc:lr_exactsol}
\end{subfigure}
\begin{subfigure}[t]{.30\linewidth}\centering
\includegraphics[width=0.9\linewidth,trim=0 0 55 0,clip]{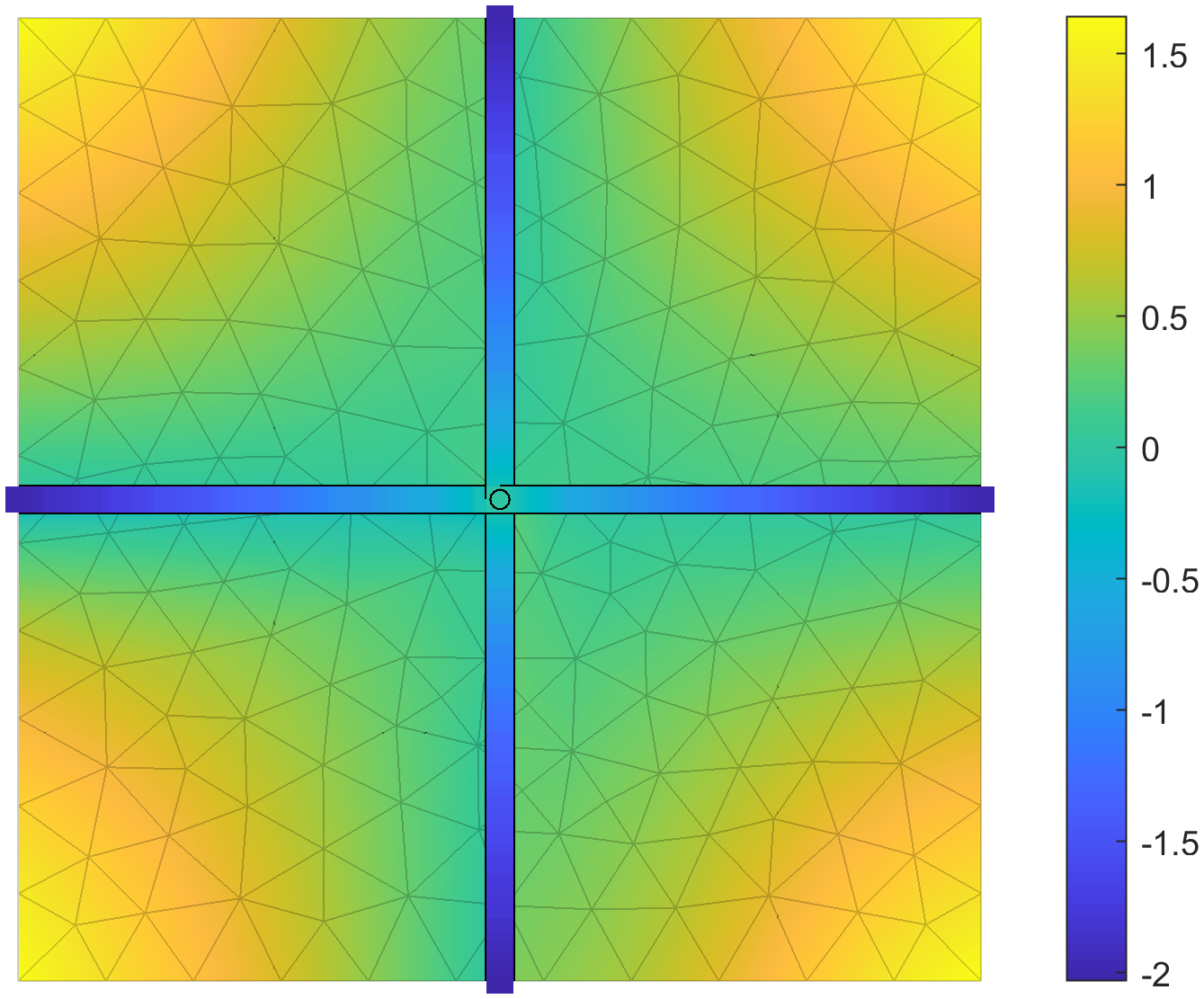}
\caption{Numerical solution}\label{pc:lr_numersol}
\end{subfigure}
\caption{\emph{Low regularity example.} (a) In this set-up: bulk diffusion coefficients $\alpha_{2,i} = \epsilon_{\mathrm{bulk}} \bfI_{2\times 2}$ with $\epsilon_{\mathrm{bulk}} = 10^{-5}$; fracture diffusion coefficient $\alpha_{1,i} = 0$; bulk vector fields $\beta_{2,i}$ are $[1,0], [0, -1], [-1, 0], [0, 1]$; fracture vector fields $\beta_{1,i}$ are $[0,1], [-1,0], [0,1], [1,0]$. (b) Exact solution. (c) Numerical solution when $h=0.2$.}

\end{figure}
The exact and numerical solutions for $h=0.2$ are shown in
Figures~\ref{pc:lr_exactsol}--\ref{pc:lr_numersol}. The empirical
findings in Table~\ref{table:lowregularity-convectiondom} are close to
the heuristic fracture rate $0.66$. Neither the energy nor the $L^2$
rate has a theoretical comparison for this coupled problem.

\begin{table}
\scriptsize
\caption{\emph{Low regularity example.} Experimental errors and convergence rates for the convection-dominated regime with a low regularity solution.}
\centering
\label{table:lowregularity-convectiondom}
\begin{tabular}{ c   c c c c } 
\toprule
$h$  & Errors in energy norm  & Energy-OC & Errors in $L^2$-norm  &  $L^2$-OC \\
\midrule
\multirow{4}{*}{}
1/5    & \na{12.2765e-1}  & $-$       & \na{9.91671e-2}    & $-$  \\
1/10   & \na{7.02036e-1}   & 0.80629  &  \na{4.00167e-2}   & 1.30926  \\ 
1/20   & \na{4.49382e-1}   & 0.64360  &  \na{1.79944e-2}   & 1.15305  \\ 
1/40   & \na{2.92701e-1}   & 0.61851  &  \na{1.06223e-2}   & 0.76045 \\
1/80   & \na{1.64321e-1}   & 0.83291  &  \na{3.91005e-3}   & 1.44185 \\ 
\bottomrule
\end{tabular}
\end{table}

\subsection{Numerical Illustrations}\label{sec:illustration}

\begin{figure}
\centering
\begin{subfigure}[t]{.3\linewidth}\centering
\includegraphics[width=0.911\linewidth]{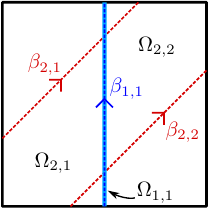}
\caption{Illustration 1} \label{fig:illust1}
\end{subfigure}
\begin{subfigure}[t]{.6\linewidth}\centering
\includegraphics[width=0.9\linewidth]{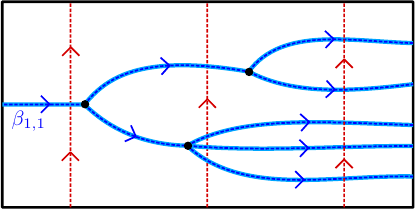}
\caption{Illustration 2}\label{illustration2:setup}
\end{subfigure}
\caption{\emph{Set-ups for numerical illustrations.}
(a) The bulk domains feature a constant diagonal convection field $\beta_{2,1} = \beta_{2,2} = [1, 1]$, while the remaining coefficients are varied in the illustration.
(b) In the bulk domains we have no diffusion ($\epsilon=0$) and constant vector fields $\beta_{2,i}=[0,1]$. In the fractures, we take diffusion parameter $\epsilon=10^{-3}$, vary the inflow fracture vector field $\beta_{1,1}$, and distribute the remaining vector-field magnitudes evenly among the outflow fractures at each junction or bifurcation point.
}
\label{fig:illustrations-setup}
\end{figure}

Finally, we present two numerical illustrations. Some of the illustrated
coefficient configurations are partially degenerate and thus lie outside
the stability and error analysis; they are included only to demonstrate
the behavior of the method. For simplicity, we set the reaction
coefficients $\kappa_{d,i} = 0$ and source functions $f_{d,i} = 0$ in
these examples.

\paragraph{Illustration 1: Transport Across a Single Fracture.} 

We consider the same geometric set-up as in Case I and study how the solution changes with different diffusion parameters and fracture vector fields. In all examples the bulk vector fields cross the fracture, and the mesh size is fixed at $h=0.1$.

First, in Figures~\ref{fig-pc-zero}--\ref{fig-pc-nonzero2}, we consider the purely convective case, i.e. $\alpha_{d,i}=0$. This reproduces the behavior observed in \cite[Section 5, Example 3]{BuHaLaLa19}: if there is no convection in the fracture, then the bulk solution is essentially unaffected by the fracture, see Figure~\ref{fig-pc-zero}. As the fracture vector field increases, see Figures~\ref{fig-pc-nonzero1}--\ref{fig-pc-nonzero2}, the solution is transported progressively along the fracture after crossing the interface.

Next, we introduce diffusion in the fracture with parameter $\epsilon=1/30$, while keeping zero diffusion in the bulk domains. In Figure~\ref{fig:illustr1-zero} we consider the case of zero fracture convection. The fracture then has only a mild influence on the solution, and only a small amount of transport is visible along the fracture. Increasing the fracture vector field in Figures~\ref{fig:illustr1-nonzero1}--\ref{fig:illustr1-nonzero2} produces a more pronounced transport along the fracture.

Finally, Figures~\ref{fig-condiff-zero}--\ref{fig-condiff-nonzero2} show the case where diffusion with parameter $\epsilon=1/30$ is present in both the bulk and fracture domains. This leads to a visibly smoother solution both in the fracture and in the surrounding bulk domains. For smaller diffusion parameters, such as $\epsilon=10^{-3}$, the numerical solutions are qualitatively similar to those in Figures~\ref{fig-pc-zero}--\ref{fig-pc-nonzero2}.

\begin{figure}
\centering
\begin{subfigure}[t]{.32\linewidth}\centering
\includegraphics[width=0.9\linewidth,trim=0 0 55 0,clip]{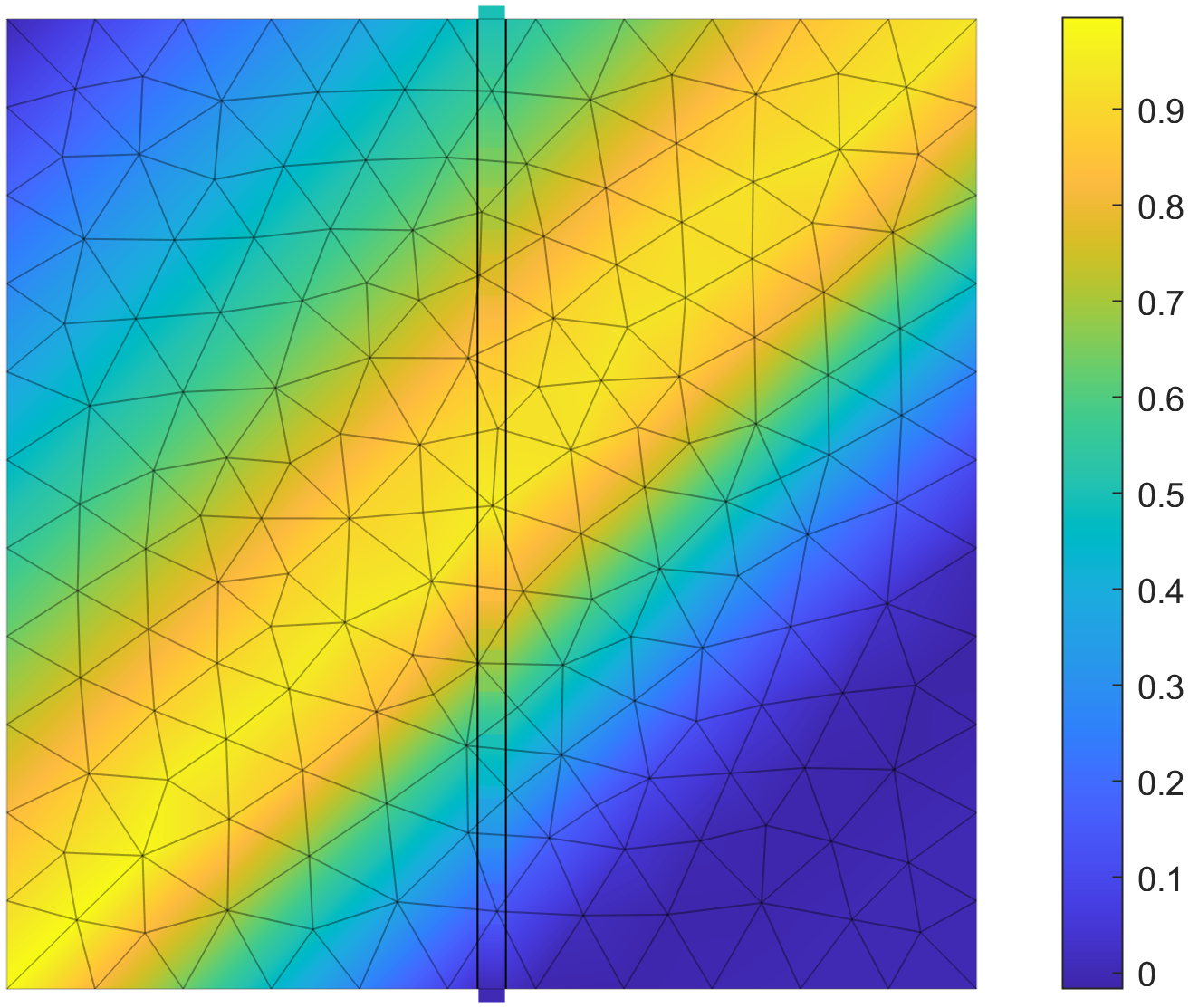}
\caption{$\beta_{1,1} = [0,0]$}\label{fig-pc-zero}
\end{subfigure}
\begin{subfigure}[t]{.32\linewidth}\centering
\includegraphics[width=0.9\linewidth,trim=0 0 55 0,clip]{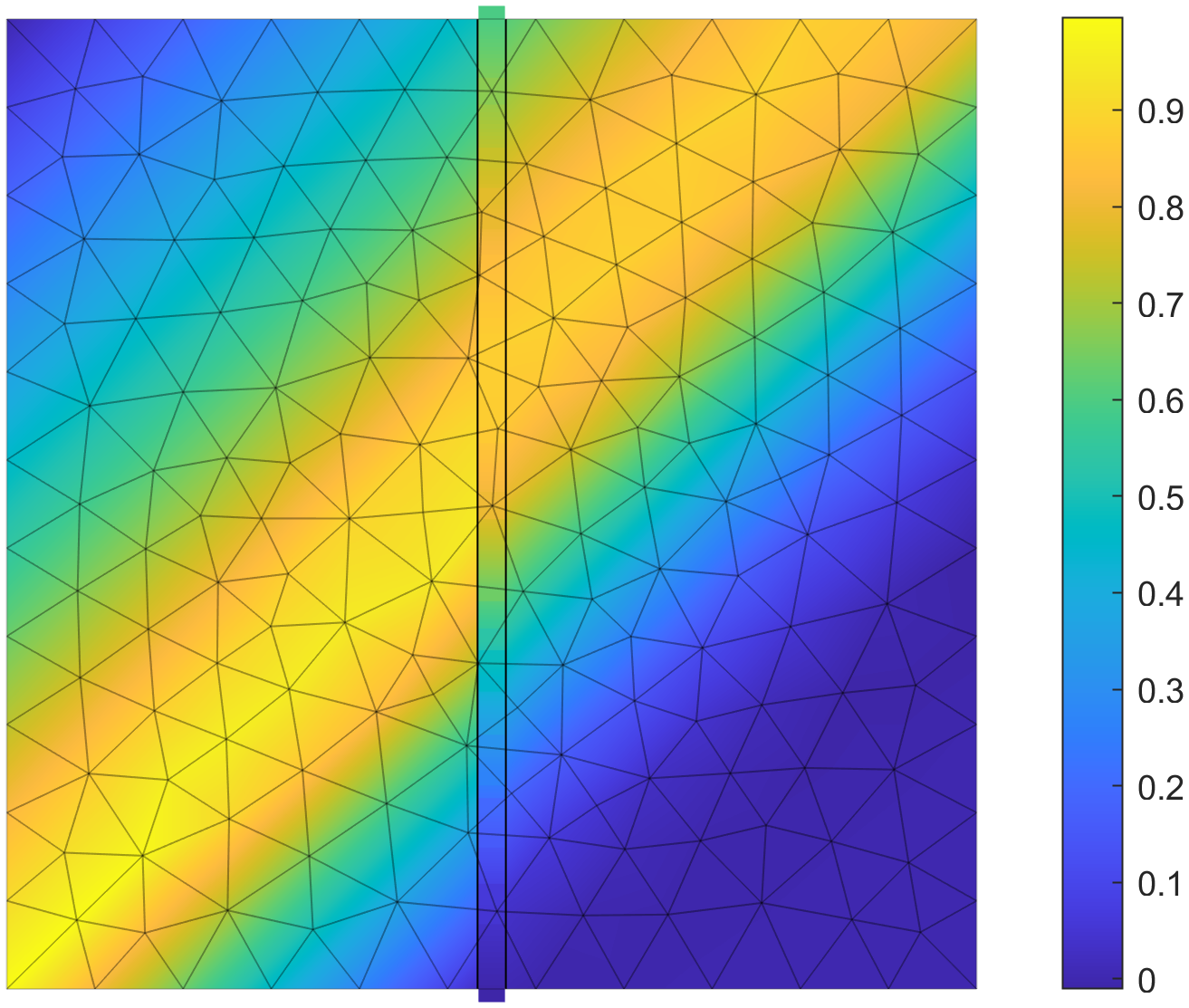}
\caption{$\beta_{1,1} = [0,0.1]$}\label{fig-pc-nonzero1}
\end{subfigure}
\begin{subfigure}[t]{.32\linewidth}\centering
\includegraphics[width=0.9\linewidth,trim=0 0 55 0,clip]{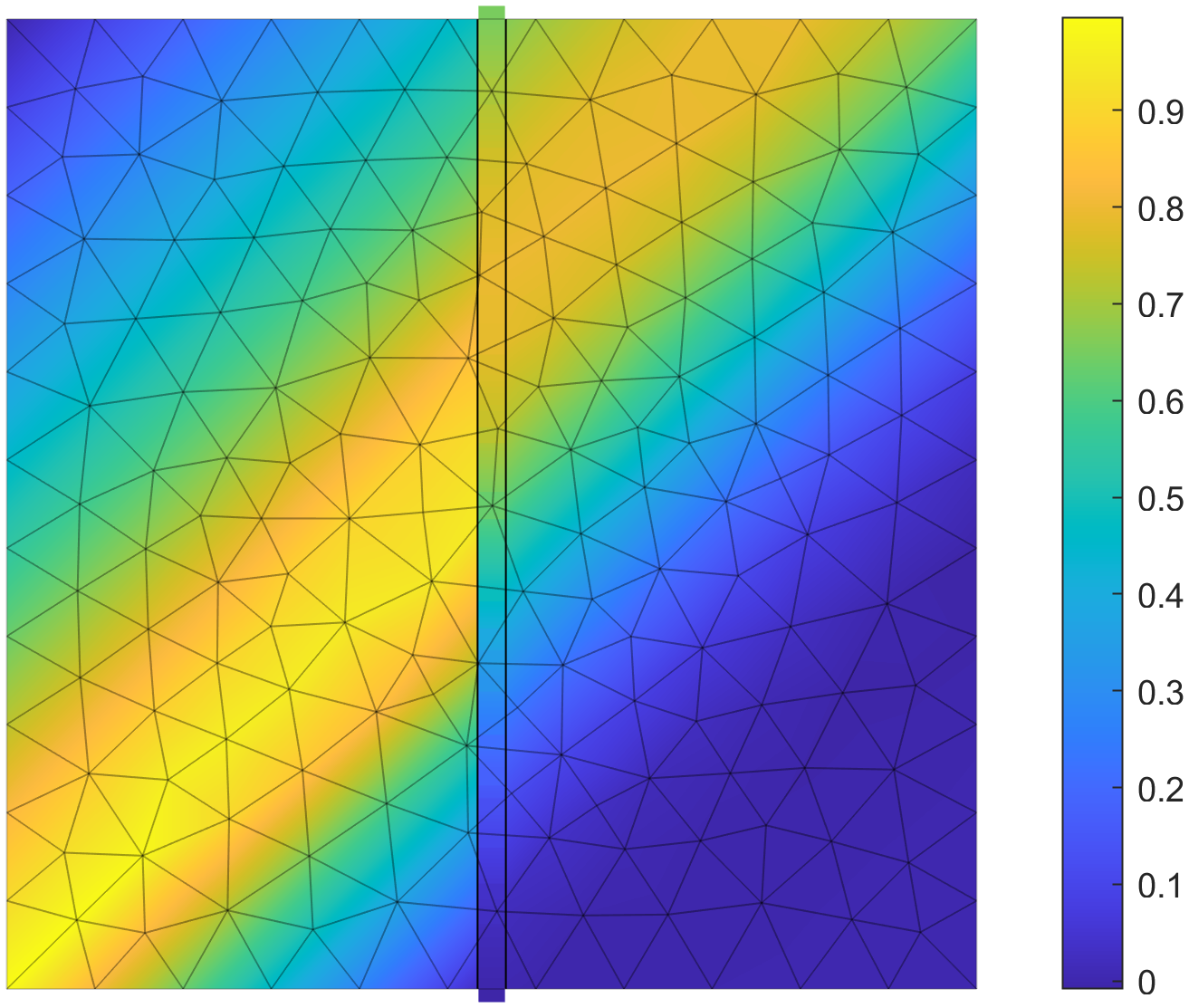}
\caption{$\beta_{1,1} = [0,0.2]$}\label{fig-pc-nonzero2}
\end{subfigure}

\begin{subfigure}[t]{.32\linewidth}\centering
\includegraphics[width=0.9\linewidth,trim=0 0 55 0,clip]{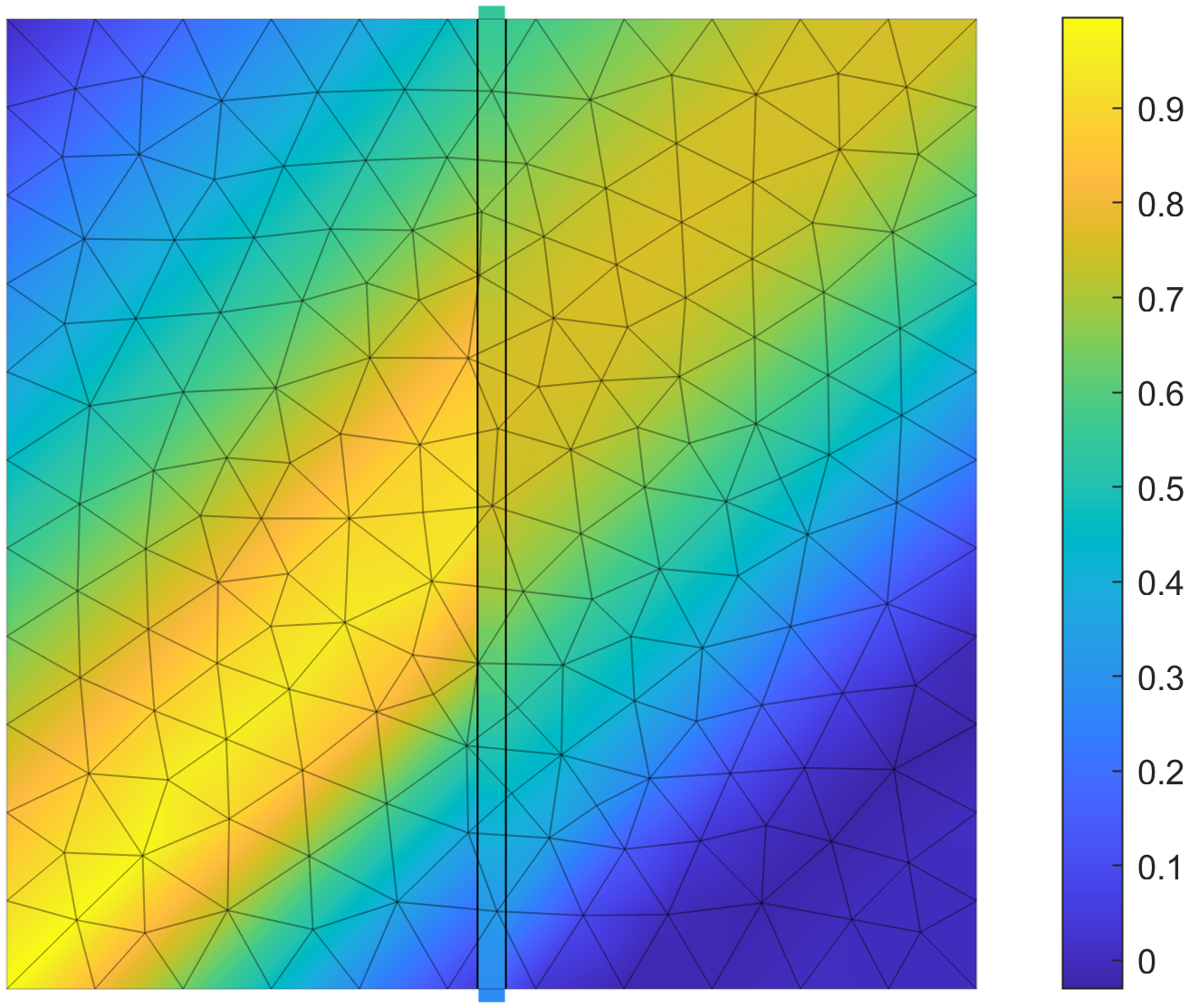}
\caption{$\epsilon_1=\frac{1}{30}$, $\beta_{1,1} = [0,0]$}\label{fig:illustr1-zero}
\end{subfigure}
\begin{subfigure}[t]{.32\linewidth}\centering
\includegraphics[width=0.9\linewidth,trim=0 0 55 0,clip]{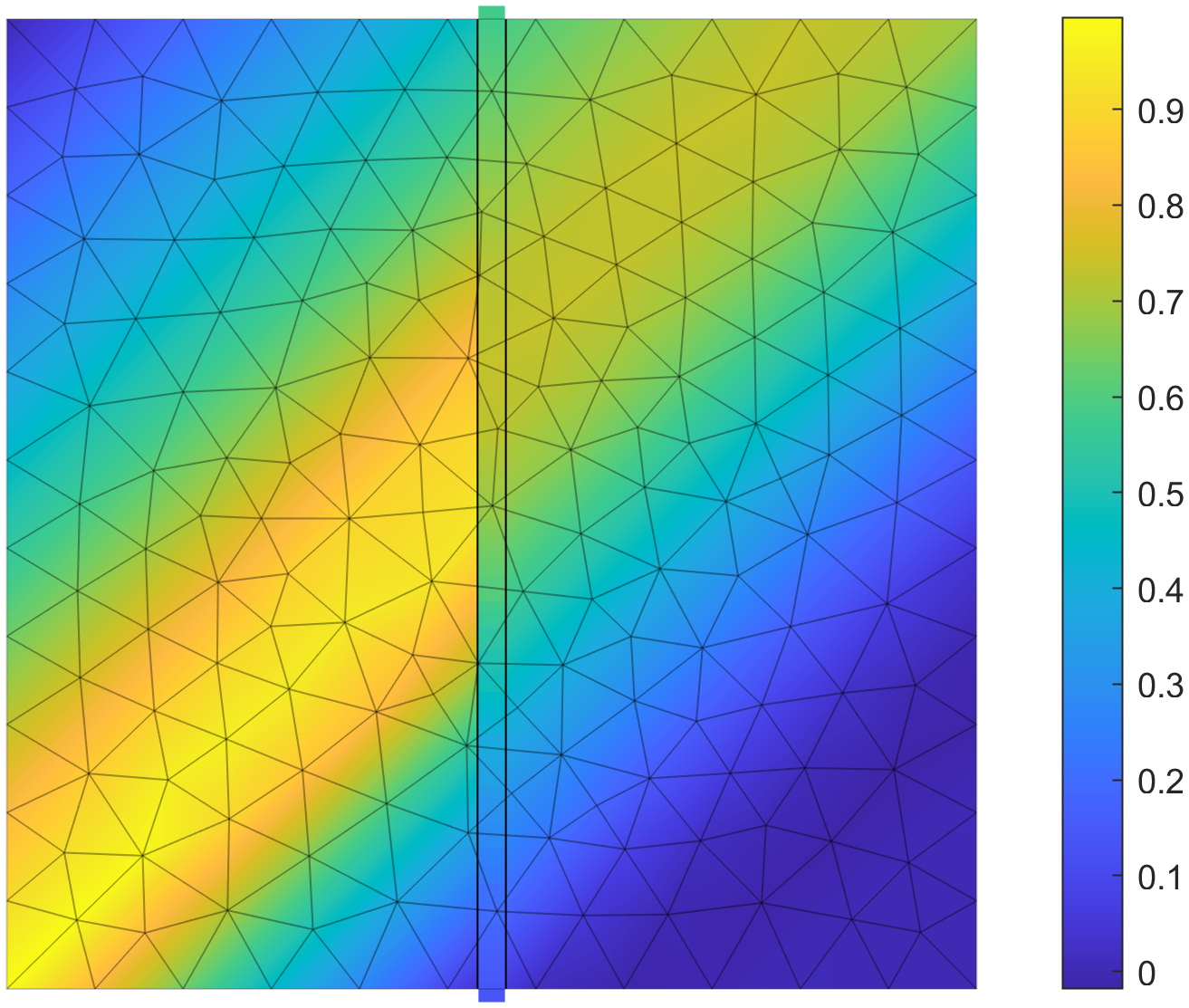}
\caption{$\epsilon_1=\frac{1}{30}$, $\beta_{1,1} = [0,0.1]$}\label{fig:illustr1-nonzero1}
\end{subfigure}
\begin{subfigure}[t]{.32\linewidth}\centering
\includegraphics[width=0.9\linewidth,trim=0 0 55 0,clip]{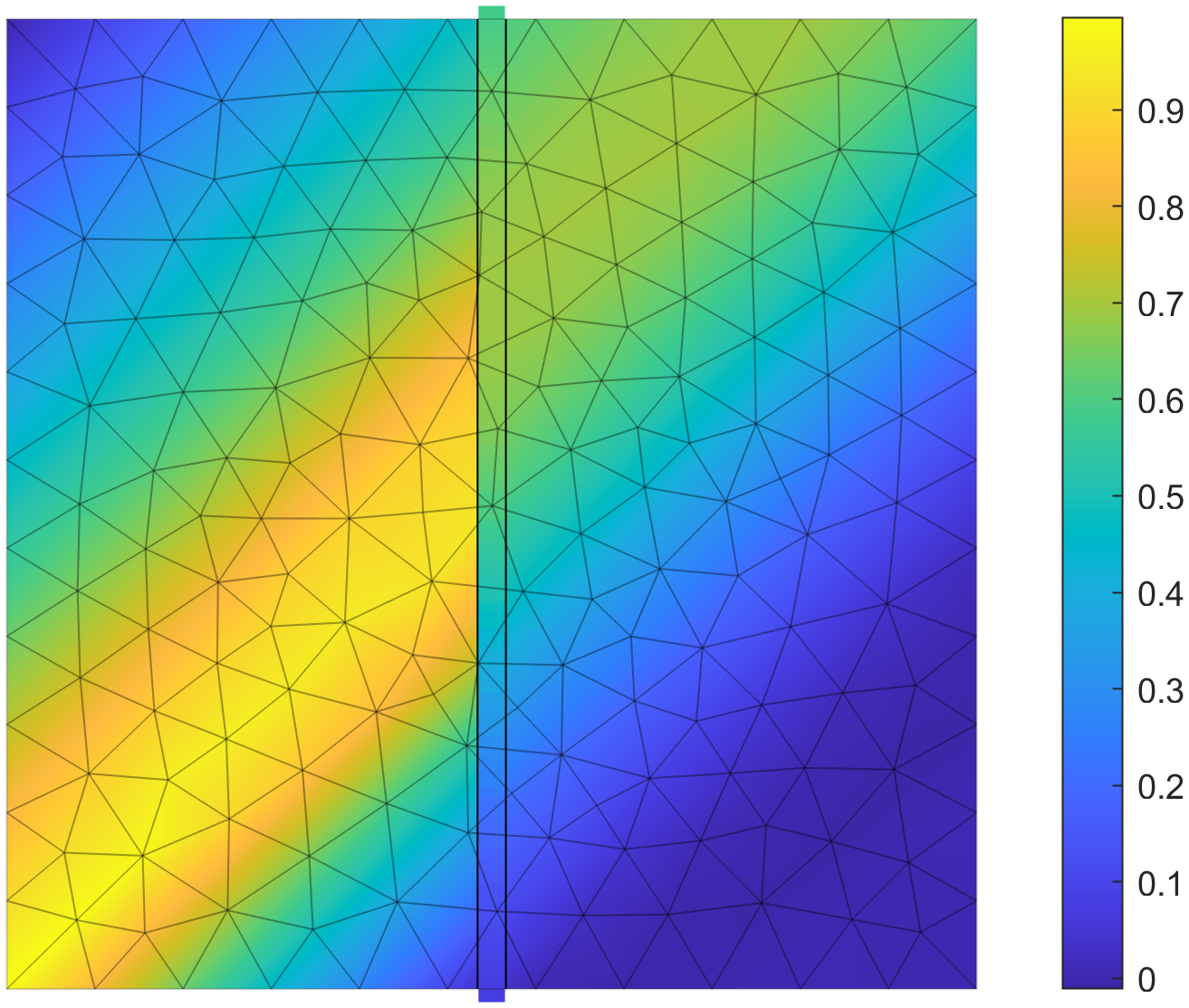}
\caption{$\epsilon_1=\frac{1}{30}$, $\beta_{1,1} = [0,0.2]$}\label{fig:illustr1-nonzero2}
\end{subfigure}

\begin{subfigure}[t]{.32\linewidth}\centering
\includegraphics[width=0.9\linewidth,trim=0 0 55 0,clip]{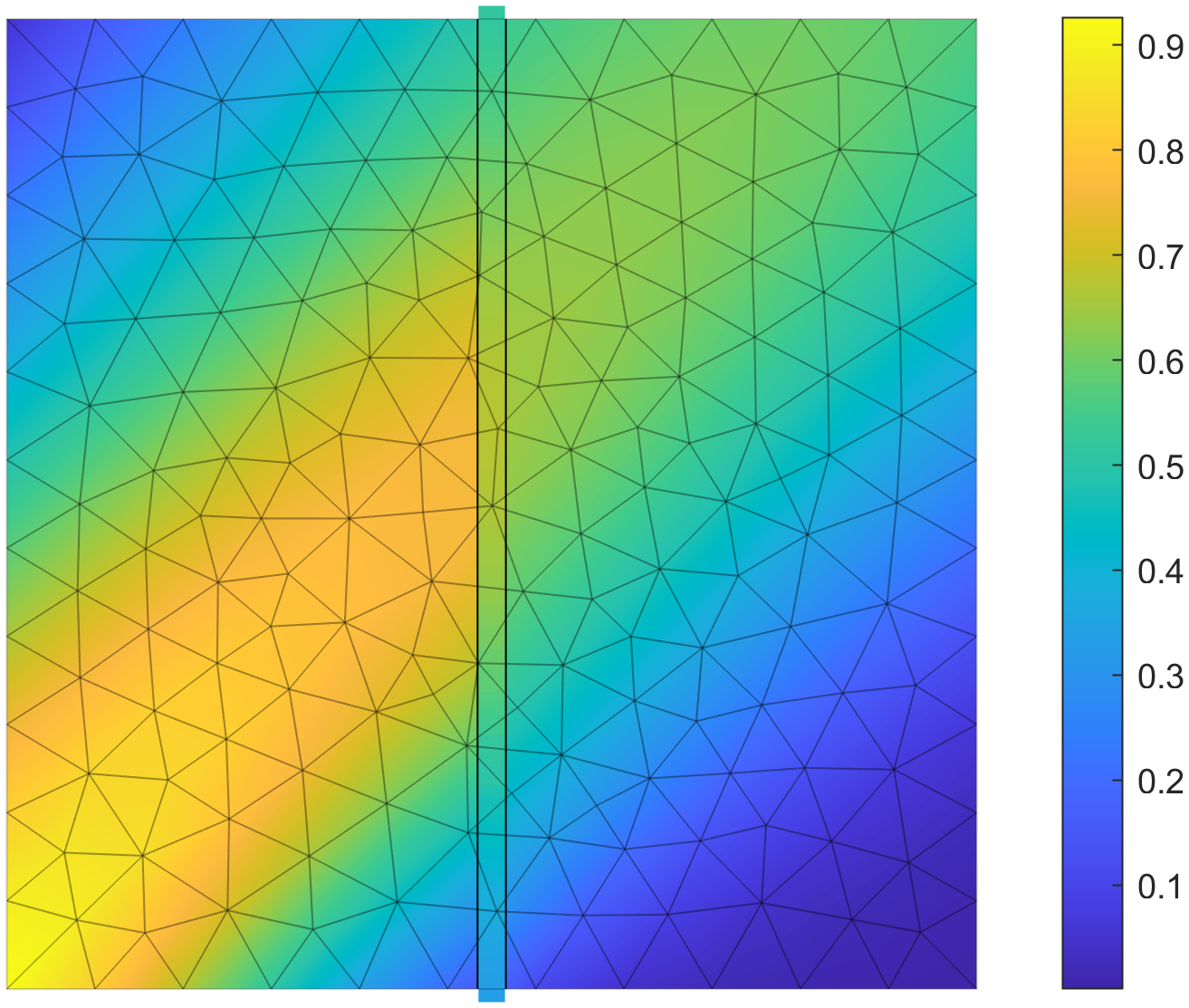}
\caption{$\epsilon_1=\epsilon_2=\frac{1}{30}$, $\beta_{1,1} = [0,0]$}\label{fig-condiff-zero}
\end{subfigure}
\begin{subfigure}[t]{.32\linewidth}\centering
\includegraphics[width=0.9\linewidth,trim=0 0 55 0,clip]{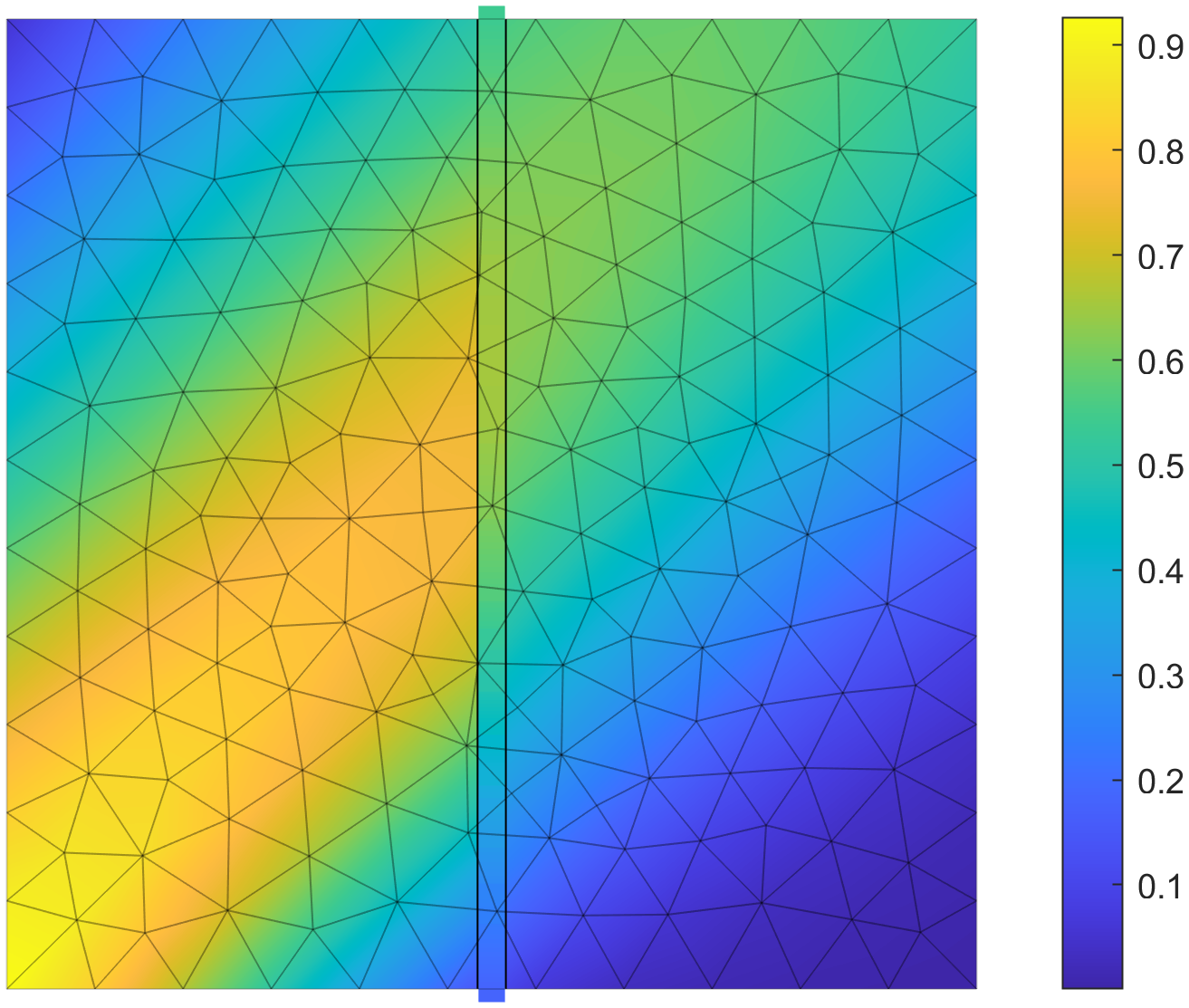}
\caption{$\epsilon_1=\epsilon_2=\frac{1}{30}$, $\beta_{1,1} = [0,0.1]$}\label{fig-condiff-nonzero1}
\end{subfigure}
\begin{subfigure}[t]{.32\linewidth}\centering
\includegraphics[width=0.9\linewidth,trim=0 0 55 0,clip]{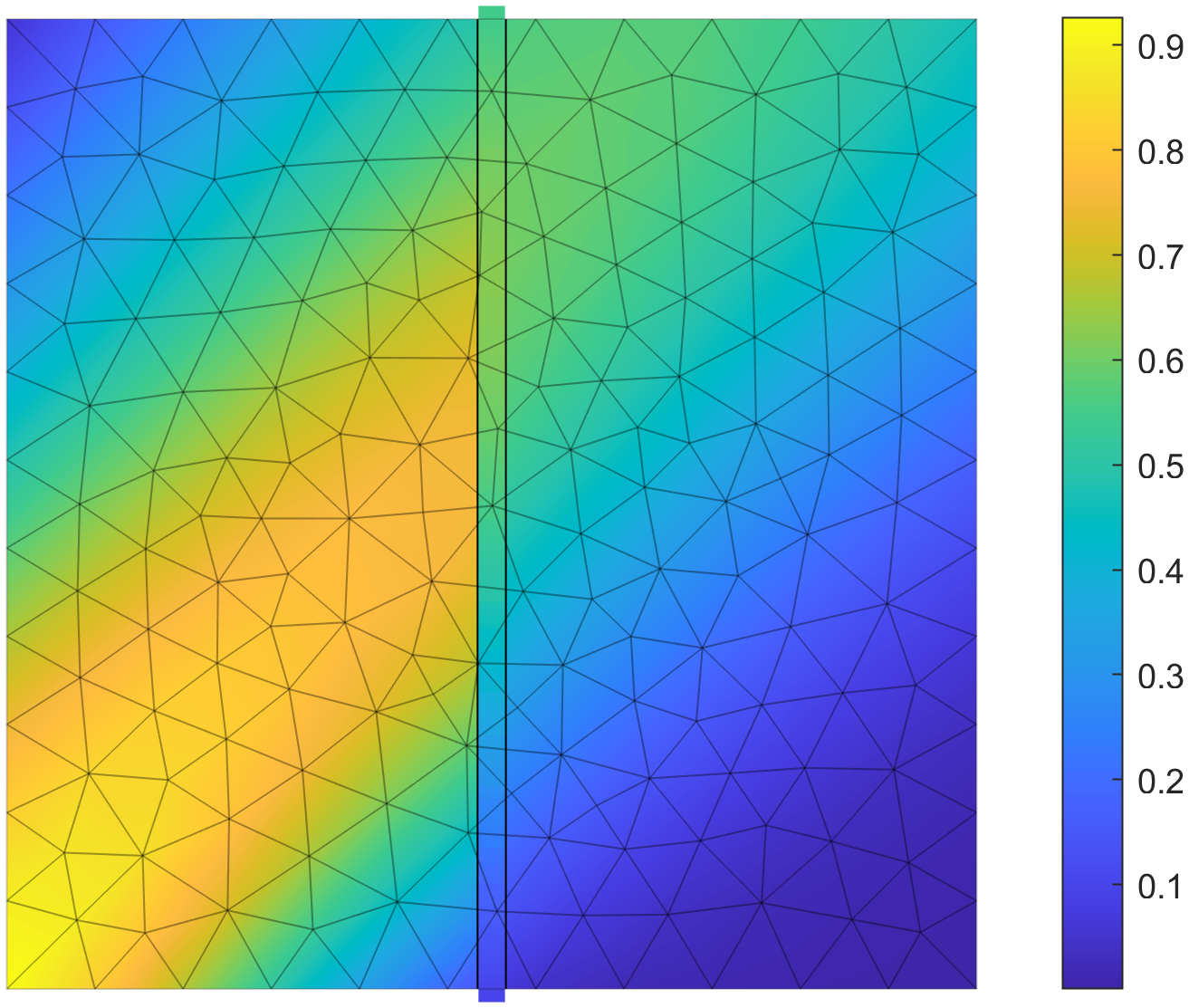}
\caption{$\epsilon_1=\epsilon_2=\frac{1}{30}$, $\beta_{1,1} = [0,0.2]$}\label{fig-condiff-nonzero2}
\end{subfigure}

\caption{\emph{Illustration 1: Transport across a single fracture.} Examples of convection-diffusion on a simple mixed-dimensional domain consisting of two bulk domains and one fracture.
The bulk domains feature a diagonal convection field $\beta_{2,1} = \beta_{2,2} = [1, 1]$ and diffusion $\alpha_{2,i} = \epsilon_2 \bfI_{2\times 2}$.
The fracture domain features a convection field $\beta_{1,1}$ and diffusion $\alpha_{1,1} = \epsilon_1$. Parameter values for the fracture convection field ($\beta_{1,1}$) and the diffusion coefficients in the fracture ($\epsilon_1$) and bulk ($\epsilon_2$) are varied, with nonzero values indicated in the subfigures.
\textbf{(a)--(c)} In the top row, we consider the purely convective case where increasing the convection in the fracture  results in a gradual transport of the solution along the fracture.
\textbf{(d)--(f)} In the middle row, we in addition introduce diffusion in the fracture resulting in a softer appearance of the solution after crossing the fracture.
\textbf{(g)--(i)} In the bottom row, we consider the convection-diffusion case where diffusion is introduced in both the bulk and the fracture domains resulting in further softening of the solution also in the bulk domains.}
\label{fig-pc:illustrations}
\end{figure}

\paragraph{Illustration 2: Transport Across a Curved Fracture Network.}
Here $\Omega=(0,2)\times(0,1)$ and, in contrast to Illustration 1, we consider a geometry with several fractures, some of which are curved, together with bifurcation points and bulk flow crossing the fractures in the vertical direction. The setup is shown in Figure~\ref{illustration2:setup}.

In the bulk domains we take no diffusion, i.e. $\epsilon=0$, choose vector fields $\beta_{2,i}=[0,1]$, and prescribe boundary data $g_{2,i}=1+\sin(15x)$. In the fractures we take diffusion parameter $\epsilon=10^{-3}$ and boundary data $g_{1,i}=1$. Proceeding clockwise, the fracture vector fields after the first, second, and third bifurcation points are chosen to have magnitude one-half, one-fourth, and one-sixth of the vector field on the first fracture.

The corresponding results for different fracture fields are displayed in Figure~\ref{fig:illustrations2}. From Figure~\ref{fig:illustr2-zero} with $\beta_{1,1}=[0,0]$ we observe no transport of the solution while crossing the fractures. As the fracture vector fields increase, see Figures~\ref{fig:illustr2-nonzero1}--\ref{fig:illustr2-nonzero2}, the solution is increasingly affected by the fractures and transported along them while crossing the interfaces.

\begin{figure}
\centering
\begin{subfigure}[t]{.6\linewidth}\centering
\includegraphics[width=0.9\linewidth]{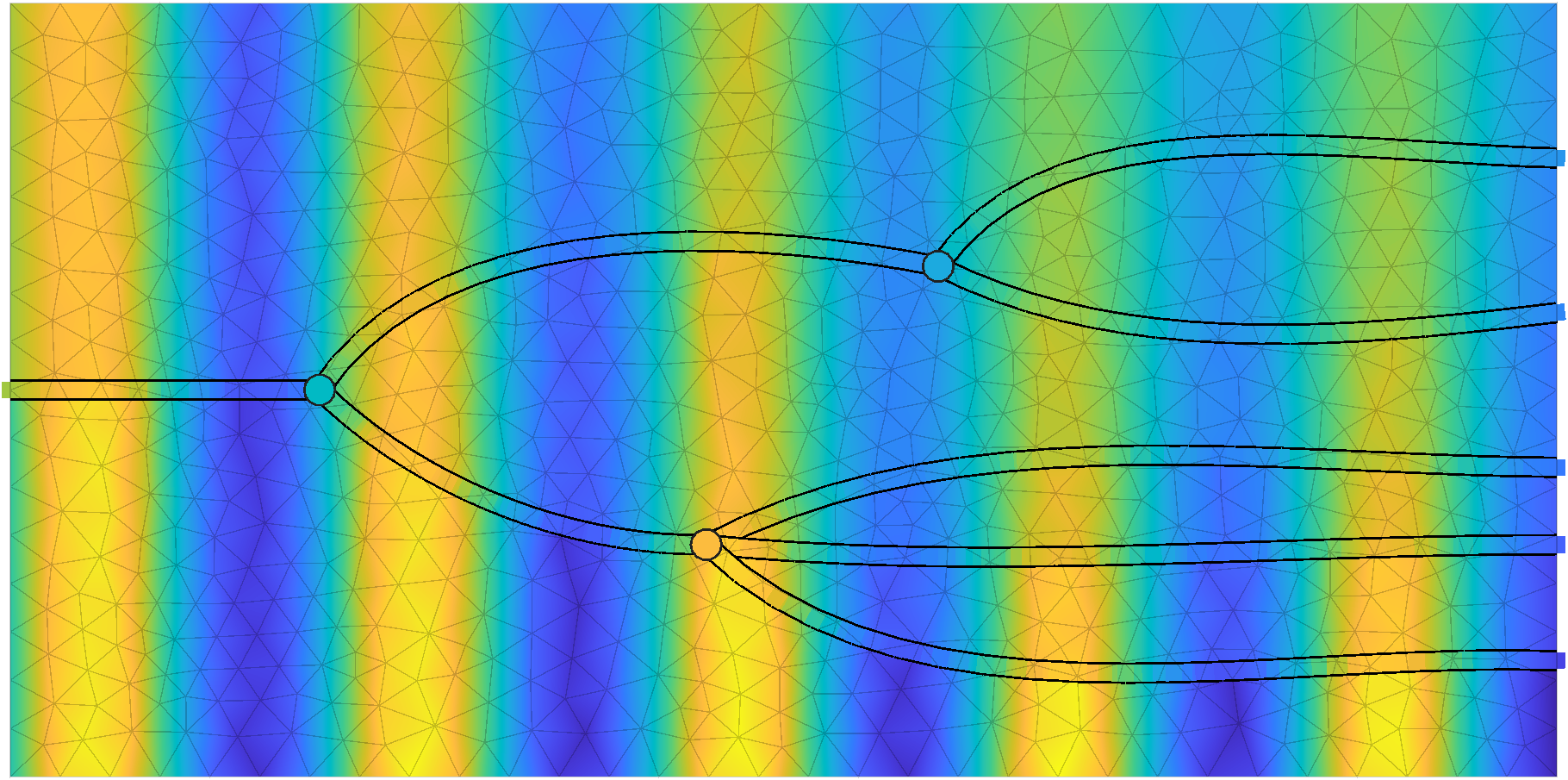}
\caption{$\beta_{1,1} = [0,0]$}\label{fig:illustr2-zero}
\end{subfigure}
\begin{subfigure}[t]{.6\linewidth}\centering
\includegraphics[width=0.9\linewidth]{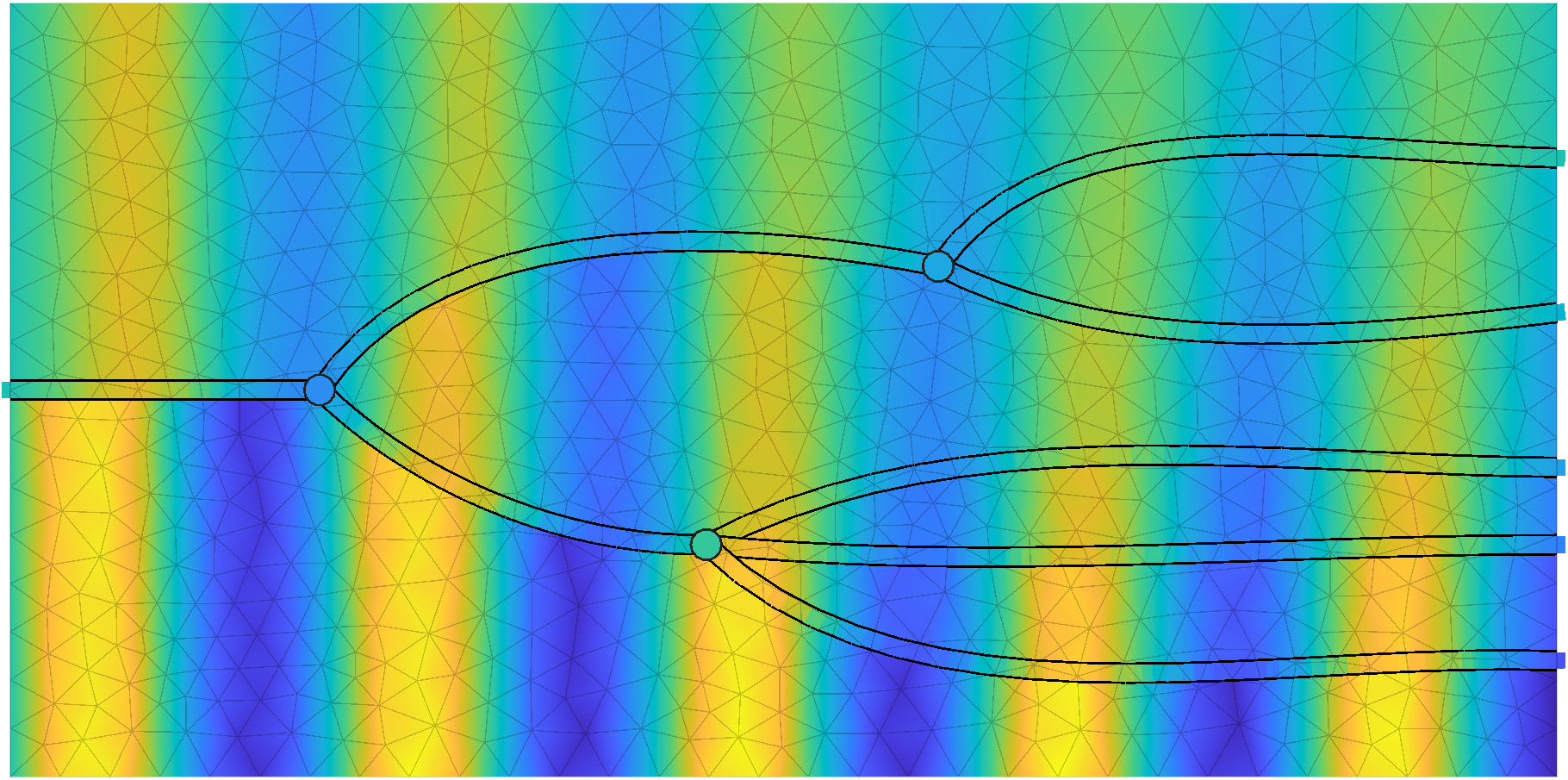}
\caption{$\beta_{1,1} = [0,0.1]$}\label{fig:illustr2-nonzero1}
\end{subfigure}
\begin{subfigure}[t]{.6\linewidth}\centering
\includegraphics[width=0.9\linewidth]{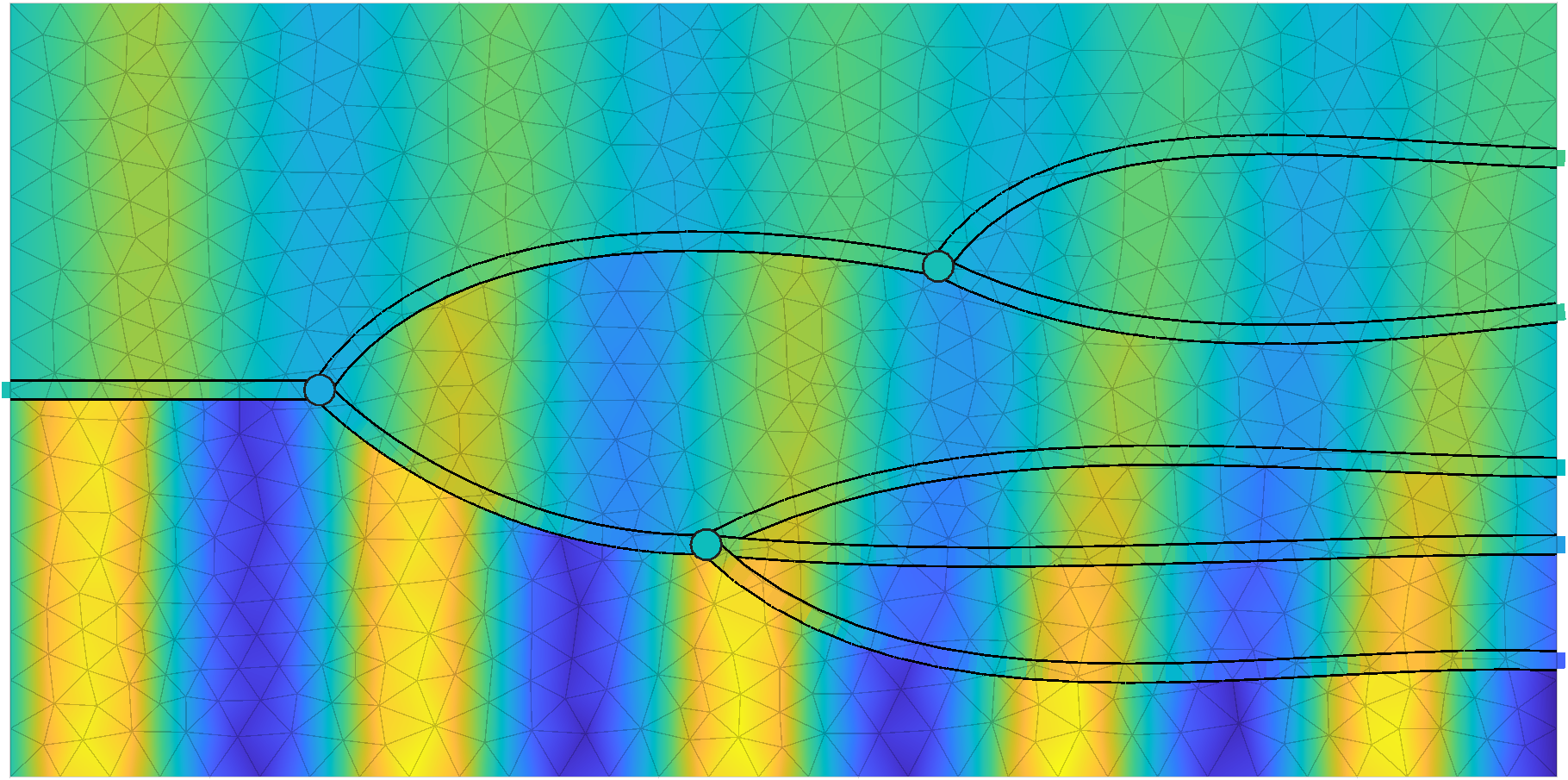}
\caption{$\beta_{1,1} = [0,0.2]$}\label{fig:illustr2-nonzero2}
\end{subfigure}
\caption{\emph{Illustration 2: Transport across a curved fracture network.} Numerical solutions for the geometry shown in Figure~\ref{illustration2:setup}, consisting of several fractures, including curved branches, together with bifurcation points and bulk flow crossing the fractures in the vertical direction. In the bulk domains we take $\alpha_{2,i}=0$ and $\beta_{2,i}=[0,1]$, while in the fractures we take diffusion coefficient $\alpha_{1,i}=10^{-3}$. The fracture vector field on the first fracture is varied as indicated in the subfigures, and the remaining fracture vector fields are scaled according to their position in the network as described in the text. As the fracture convection increases, the solution is transported more strongly along the fracture network while crossing the interfaces.}
\label{fig:illustrations2}
\end{figure}

\section{Conclusions}
In this work we develop and analyze a stabilized CutFEM for convection--diffusion problems posed on mixed-dimensional domains. The main contributions of the paper can be summarized as follows:

\begin{itemize}
\item We introduced a unified operator framework for mixed-dimensional convection--diffusion problems based on abstract directional derivative, divergence, and elliptic operators, together with mixed-dimensional interface operators that encode inter-manifold coupling. This formulation allows the governing equations to be written in a form closely resembling classical convection--diffusion problems on domains in $\mathbb{R}^n$. Interactions between manifolds of different dimensions are represented through the upward trace-sum and downward jump operators, which encode flux exchange in a compact and systematic manner.

\item We proposed a stabilized CutFEM discretization in which the background mesh is independent of the embedded lower-dimensional manifolds. This avoids mesh-fitting and enables straightforward treatment of complex mixed-dimensional geometries. The method combines least-squares stabilization for convection with full-gradient stabilization on cut elements. This yields a stable formulation that performs well in convection-dominated regimes.

\item We established a priori error estimates in the natural energy
norm of the method under a global uniform-diffusion assumption,
including extensions to low-regularity solutions, together with
conditional estimates for globally pure convection. The analysis is
complemented by quantitative numerical experiments, including
partially degenerate configurations outside the precise assumptions of
the theory, that report convergence in both the energy and $L^2$-norms
and illustrate the behavior of the method for different diffusion
regimes and flow configurations.
\end{itemize}

Several directions for future work remain open. On the modeling side, an extension to more general coupling structures, in particular $3$d--$1$d interactions, would be of significant interest. From a computational perspective, the proposed framework enables large-scale simulations on complex three-dimensional mixed-dimensional geometries, where each manifold remains geometrically simple while the overall configuration is intricate. It is also natural to explore multiscale strategies that exploit the inherent hierarchical structure of mixed-dimensional domains within the CutFEM setting. Finally, extensions to time-dependent transport models constitute an important direction for future research.

\bigskip
\paragraph{Acknowledgments.}
This research was supported in part by:
EPSRC grants EP\slash T033126\slash 1 (EB), EP\slash V050400\slash 1 (EB), and EP\slash X042650\slash 1 (EB);
the Swedish Research Council grants 2021-04925 (MGL), 2022-03908 (PH), and 2025-05562 (MGL);
the Swedish Research Programme Essence (MGL, KL);
the Kempe Foundations grant JCSMK22-0139 (KL, SM).

During final preparation of the original and revised manuscripts, the authors used OpenAI's GPT-5.5 and GPT-5.6 Sol models for language editing and to help identify ambiguities and possible corrections in the mathematical arguments. AI was not used to draft the manuscript or to develop, run, or evaluate the numerical experiments. All suggestions were independently verified by the authors, who take full responsibility for the final content.

\bibliographystyle{habbrv}
{
\footnotesize
\bibliography{layer_biblio}
}

\bigskip
\bigskip
{%
\footnotesize
\noindent
{\bf Authors' addresses:}

\smallskip
\noindent
Erik Burman,  \quad \hfill \addressuclshort\\
{\tt e.burman@ucl.ac.uk}

\smallskip
\noindent
Peter Hansbo,  \quad \hfill \addressjushort\\
{\tt peter.hansbo@ju.se}

\smallskip
\noindent
Mats G. Larson,  \quad \hfill \addressumushort\\
{\tt mats.larson@umu.se}

\smallskip
\noindent
Karl Larsson, \quad \hfill \addressumushort\\
{\tt karl.larsson@umu.se}

\smallskip
\noindent
Shantiram Mahata,  \quad \hfill Mathematics and Physics, Linnaeus~University, Sweden\\
{\tt shantiram.mahata@lnu.se}

}%

\end{document}